\documentclass[12pt]{article}
\usepackage{amssymb,amsmath,amscd, amsthm,stmaryrd,verbatim, mathrsfs,graphicx,subfigure,caption,lipsum}
\usepackage[english]{babel}
\numberwithin{equation}{section}
\newtheorem{theorem}{Theorem}[section]
\newtheorem{lemma}{Lemma}[section]
\newtheorem{corollary}{Corollary}[section]
\newtheorem{proposition}{Proposition}[section]

\renewcommand{\figurename}{\textbf{Fig.}}

\begin{document}
\baselineskip=14pt

\newcommand{\la}{\langle}
\newcommand{\ra}{\rangle}
\newcommand{\psp}{\vspace{0.4cm}}
\newcommand{\pse}{\vspace{0.2cm}}
\newcommand{\ptl}{\partial}
\newcommand{\dlt}{\delta}
\newcommand{\sgm}{\sigma}
\newcommand{\al}{\alpha}
\newcommand{\be}{\beta}
\newcommand{\G}{\Gamma}
\newcommand{\gm}{\gamma}
\newcommand{\vs}{\varsigma}
\newcommand{\Lmd}{\Lambda}
\newcommand{\lmd}{\lambda}
\newcommand{\td}{\tilde}
\newcommand{\vf}{\varphi}
\newcommand{\yt}{Y^{\nu}}
\newcommand{\wt}{\mbox{wt}\:}
\newcommand{\rd}{\mbox{Res}}
\newcommand{\ad}{\mbox{ad}}
\newcommand{\stl}{\stackrel}
\newcommand{\ol}{\overline}
\newcommand{\ul}{\underline}
\newcommand{\es}{\epsilon}
\newcommand{\dmd}{\diamond}
\newcommand{\clt}{\clubsuit}
\newcommand{\vt}{\vartheta}
\newcommand{\ves}{\varepsilon}
\newcommand{\dg}{\dagger}
\newcommand{\tr}{\mbox{Tr}}
\newcommand{\ga}{{\cal G}({\cal A})}
\newcommand{\hga}{\hat{\cal G}({\cal A})}
\newcommand{\Edo}{\mbox{End}\:}
\newcommand{\for}{\mbox{for}}
\newcommand{\kn}{\mbox{ker}}
\newcommand{\Dlt}{\Delta}
\newcommand{\rad}{\mbox{Rad}}
\newcommand{\rta}{\rightarrow}
\newcommand{\mbb}{\mathbb}
\newcommand{\lra}{\Longrightarrow}
\newcommand{\X}{{\cal X}}
\newcommand{\Y}{{\cal Y}}
\newcommand{\Z}{{\cal Z}}
\newcommand{\U}{{\cal U}}
\newcommand{\V}{{\cal V}}
\newcommand{\W}{{\cal W}}
\newcommand{\sta}{\theta}
\setlength{\unitlength}{3pt}
\newcommand{\msr}{\mathscr}
\newcommand{\wht}{\widehat}
\newcommand{\mfk}{\mathfrak}
\renewcommand{\figurename}{\textbf{Fig.}}

\begin{center}{\large \bf Hilbert Polynomials of Noncanonical Orthogonal\\ \pse Oscillator  Representations of $sl(n)$} 
\end{center}

\vspace{0.2cm}

\begin{center}{\large Hengjia Zhang\footnote{Corresponding author.} and Xiaoping Xu\footnote{Research supported
 by National Key R\&D Program of China 2020YFA0712600.
} }\end{center}

\begin{center}{
HLM, Institute of Mathematics, Academy of Mathematics \& System
Sciences\\ Chinese Academy of Sciences, Beijing 100190, P.R. China
\\ \& School of Mathematics, University of Chinese Academy of Sciences,\\ Beijing 100049, P.R. China}\end{center}

\begin {abstract}
\quad

By applying Fourier transformations to the natural orthogonal oscillator representations of special linear Lie algebras, Luo and the second author (2013) obtained a large family of infinite-dimensional irreducible representations of the algebras on the homogeneous solutions of the Laplace equation. 
In our earlier work, we proved that the associated varieties of these irreducible representations are the intersection of determinantal varieties.
In this paper, we find the Hilbert polynomial $\mfk p_{_M}(k)$ of these associated varieties. Moreover, we show that the  Hilbert polynomial $\mfk p_{_M, _{M_0}}(k)$ of such an  irreducible module $M$ with respect to any generating subspace $M_0$ satisfies $\mfk p_{_M}(k)\leq \mfk p_{_M, _{M_0}}(k)$ for sufficiently large positive integer $k$ and find a necessary and sufficient condition that
the equality holds. Furthermore, we explicitly determine the leading term of $\mfk p_{_M, _{M_0}}(k)$, which is independent of the choice of  $M_0$.

 \vspace{0.3cm}

\noindent{\it Keywords}:\hspace{0.3cm} special linear Lie
algebra; orthogonal oscillator
 representation; Hilbert polynomial; Hilbert series; determinantal varieties.

\end{abstract}

\section {Introduction}

Let $\msr A=\bigoplus^\infty_{k=0} \msr A_k$ be a finitely generated graded commutative algebra over a field $\mbb F$. The Hilbert function maps a nonnegative integer $k$ to the dimension of the subspace $\msr A_k$.
The Hilbert-Poincar\'{e} series is the formal series
\begin{equation}{\mfk s}_{_{\msr A}}(t)=\sum_{k=0}^{\infty}(\dim \msr A_k)t^k.\end{equation}

Hilbert \cite{Hi} (1893) showed that there exists a unique polynomial ${\mfk p}_{\msr A}(k)$ with rational coefficients, which coincides with the Hilbert function $\dim \msr A_k$ for sufficiently large $k$. This polynomial is later called {\it Hilbert polynomial}.
Moreover, he \cite{Hi} also introduced the notions of Hilbert functions and Hilbert series for a finitely generated graded module over the ring of polynomials in finitely many variables over a field.
Northcott \cite{Nd} (1947) gave a direct proof of abstract Tauberian theorems with applications to power series and Hilbert series.
In 1971, Bernstein \cite{Bj} defined a variety associated with a filtered module of the algebra of differential operators, in order to construct the fundamental solution of a linear partial differential equation with constant coefficients. He also showed that the dimension of the associated variety is equal to the Gelfand-Kirillov dimension of the module in \cite{Bj}, which is the degree of the Hilbert polynomial plus one.  Vogan \cite{Vd1} (1978) developed particularly refined techniques
in the description of the leading coefficient in the Hilbert-Samuel polynomial of a Harish-Chandra module for a semisimple Lie algebra and its maximal compact subalgebra.

Anic \cite{Ad} (1985) modeled Diophantine equations within the category of graded Hopf algebras and used their Hilbert series to determine undecidable spaces. Enright and Willenbring \cite{EW} (2004) found a connection among Hilbert series, Howe duality and branching for classical groups. Moreover,  Enright and Hunziker \cite{EH1} (2004) computed (minimal) resolutions and explicit formulas for the Hilbert series of the unitary highest weight modules of exceptional groups. They also \cite{EH2} calculated the Hilbert series of the Wallach representations and provided explicit formulas for the Hilbert series of determinantal varieties by using BGG resolutions of unitarizable highest weight modules. Gross and Wallach \cite{GW} (2011) studied the Hilbert polynomials and Hilbert series of homogeneous projective varieties. There are also many other interesting significant works related to Hilbert polynomials and Hilbert series (e.g., cf. \cite{CEU, AJ94, FMSV, DGP, Hj, SA21, Jw, K15, KS, Kt, CM04, KDM, Ku96, Nu, Rl,  Sb, Wj}).

Indeed, Hilbert polynomials and Hilbert series have played important roles in physics. For instance, Hanany and Kalveks \cite{Ha14} (2014) found an efficient encoding and analysis of the Hilbert series of the vacuum moduli spaces of instantons. Cremonesi, Mekareeya and Zaffaroni \cite{CMZ} (2016) presented a formula for the Hilbert series that counts gauge invariant chiral operators in a large class of $3d\; {\cal N}\geq 2$ Yang-Mills-Chern-Simons theories.  Gr\'{a}f, Henning, Lu,  Melia and Murayama \cite{GHLM} (2023) calculated the Hilbert series in connection with
the Higgs mechanism and high energy field theories. Kho and Seong \cite{KS} (2024) studied the algebraic structure of mesonic moduli spaces of bipartite field theories by computing the Hilbert series.

By applying Fourier transformations to the natural orthogonal oscillator representations of special linear Lie algebras, Luo and the second author \cite{LX} (2013) obtained a large family of infinite-dimensional irreducible representations of the algebras on  homogeneous solutions of the Laplace equation. In our earlier work \cite{ZX}, we proved that the associated varieties of these irreducible representations are the intersections of explicitly given determinantal varieties. In this paper, we find the Hilbert polynomial $\mfk p_{_M}(k)$ of these associated varieties. Moreover, we show that the  Hilbert polynomial $\mfk p_{M}(k)$ of such an  irreducible module $M$ with respect to any generating subspace $M_0$ satisfies $\mfk p_{_M}(k)\leq \mfk p_{_M, _{M_0}}(k)$ for sufficiently large positive integer $k$ and find a necessary and sufficient condition that
the equality holds. Furthermore, we explicitly determine the leading term of $\mfk p_{_M, _{M_0}}(k)$, which is independent of the choice of  $M_0$  and therefore is an invariant of $M$.

Let $\mbb F$ be a field with characteristic 0, and let $n>1$ be an integer. Denote by $\msr A=\mbb F[x_1,x_2,...,x_n,y_1,y_2,...,y_n]$ the polynomial algebra in $2n$ variables. Let $E_{r,s}$ be the square matrix with 1 as its $(r,s)$-entry and 0
as the others. The special linear Lie algebra
\begin{equation}
sl(n)=\sum_{1\leq i<j\leq
n}(\mbb{F}E_{i,j}+\mbb{F}E_{j,i})+\sum_{r=1}^{n-1}\mbb{F}(E_{r,r}-E_{r+1,r+1})\end{equation}
with the commutator as its Lie bracket.
Denote by $\mbb N$ the set of nonnegative integers. For any
two integers $p\leq q$, we denote $\ol{p,q}=\{p,p+1,\cdots,q\}$.
The natural orthogonal oscillator representation $\pi$ of $sl(n)$ on $\msr A$ is given as follows:
\begin{equation} \label{a1.2}
\pi(E_{i,j})=x_i\ptl_{x_j}-y_j\ptl_{y_i}\qquad\for\;\;i,j\in\ol{1,n}.\end{equation}

With respect to the above representation, the Laplace operator
\begin{equation}\Dlt=\sum_{i=1}^n\ptl_{x_i}\ptl_{y_i}.\end{equation}
 For $\ell_1,\ell_2\in \mbb N$, we denote by $\msr A_{\ell_1,\ell_2}$ the subspace of homogeneous polynomials with the degree $\ell_1$ in $\{x_1,x_2,...,x_n\}$ and the degree $\ell_2$ in $\{y_1,y_2,...,y_n\}$. Then $\msr A=\bigoplus_{\ell_1,\ell_2\in\mbb N}\msr A_{\ell_1,\ell_2}$ is an $\mbb N^2$-graded algebra.

Set the subspace of homogeneous harmonic polynomials as follows:
\begin{equation}{\msr H}_{\ell_1,\ell_2}=\{f\in {\msr A}_{\ell_1,\ell_2}\mid
\Dlt(f)=0\}. \end{equation}
Then ${\msr H}_{\ell_1,\ell_2}$ forms an irreducible $sl(n)$-module with highest weight $\ell_1\lmd_1+\ell_2\lmd_{n-1}$ (e.g., cf. \cite{Xx}), where $\lmd_i$ denotes the $i$th fundamental weight of $sl(n)$.

Denote by $\msr F_z$ the Fourier transformation on $z$. Fix $n_1,n_2\in\ol{1,n}$ with $n_1\leq n_2$. Letting (\ref{a1.2}) act on the space
\begin{equation}{\msr A}'=\msr F_{x_1}\cdots \msr F_{x_{n_1}}\msr F_{x_{n_2+1}}\cdots \msr F_{x_n}(\msr A)\end{equation}
of distributions, one can obtain a large family of infinite-dimensional irreducible representations of $sl(n)$ on certain spaces of homogeneous solutions of the Laplace equation. However, this picture is not convenient to be dealt with. Instead, we consider another equivalent representation as follows.
In the representation formulas given in (\ref{a1.2}), applying the Fourier transformations on operators:
\begin{equation}\label{a1.6}
\ptl_{x_r}\mapsto -x_r,\;
 x_r\mapsto\ptl_{x_r},\;\ptl_{y_s}\mapsto -y_s,\;
 y_s\mapsto\ptl_{y_s}\qquad \for\; r\in\ol{1,n_1},\;s\in\ol{n_2+1,n},\end{equation}
we obtain a new representation $\td{\pi}$ on $\msr A$. The representations $(\pi,\msr A')$ and $(\td{\pi},\msr A)$ are equivalent (or isomorphic).
Under the changes in (\ref{a1.6}), the Laplace operator becomes
\begin{equation}\td\Dlt=\sum_{i=1}^{n_1}x_i\ptl_{y_i}-\sum_{r=n_1+1}^{n_2}\ptl_{x_r}\ptl_{y_r}+\sum_{s=n_2+1}^n
y_s\ptl_{x_s}.\label{a1.7}\end{equation}

Take the degree of $\{x_1,...,x_{n_1},y_{n_2+1},...,y_n\}$ as $-1$ and that of $\{x_{n_1+1},...,x_n,y_1...,y_{n_2}\}$
as 1. Denote by $\mbb Z$ the ring of integers. For $\ell_1,\ell_2\in \mbb Z$, we denote by $\msr A_{\la\ell_1,\ell_2\ra}$ the subspace of homogeneous polynomials with the degree $\ell_1$ in $\{x_1,x_2,...,x_n\}$ and the degree $\ell_2$ in $\{y_1,y_2,...,y_n\}$.
 Then $\msr A=\bigoplus_{\ell_1,\ell_2\in\mbb Z}\msr A_{\la\ell_1,\ell_2\ra}$ is a $\mbb Z^2$-graded algebra. Set
\begin{equation}{\msr H}_{\la\ell_1,\ell_2\ra}=\{f\in {\msr A}_{\la\ell_1,\ell_2\ra}\mid
\td{\Dlt}(f)=0\}.\end{equation}
The following is a result from \cite{LX}:\psp

{\bf Proposition 1}\quad {\it  For
$\ell_1,\ell_2\in\mbb{Z}$ , the necessary and sufficient condition for ${\msr H}_{\la\ell_1,\ell_2\ra}$ to be an infinite-dimensional irreducible $sl(n)$-module is as follows:

(1)  When $n_1+1<n_2$, (a) $\ell_1+\ell_2\leq n_1-n_2+1$; (b) $n_2=n$, $\ell_1\in \mbb N$ and $\ell_2=0$; (c) $n_2=n$, $\ell_2\in \mbb N  $ and $\ell_1 \geq n_1-n+2$;

(2)  If $n_1+1=n_2$,  $\ell_1+\ell_2\leq 0$ or $n_2=n$ and
$0\leq\ell_2\leq\ell_1$;

(3)  In the case $n_1=n_2$, $\ell_1+\ell_2\leq 0$ and: (a)
 $\ell_2\leq 0 $, $n_1<n-1$ and $n\geq 3$;
 (b) $\ell_1\leq 0$, $1<n_1<n $ and $n\geq 3$; (c) $\ell_1,\ell_2\leq
 0$, $n_1=1$ and $n=2$. } \psp

 The above irreducible modules are infinite-dimensional highest weight modules with distinct highest weights. For instance,  if $n_1+1<n_2<n$
and $m_1,m_2\in\mbb{N}$ with $m_1+m_2\geq n_2-n_1-1$, ${\msr
H}_{\la-m_1,-m_2\ra}$ has the highest-weight vector
$x_{n_1}^{m_1}y_{n_2+1}^{m_2}$ of weight
$m_1\lmd_{n_1-1}-(m_1+1)\lmd_{n_1}-(m_2+1)\lmd_{n_2}+m_2(1-\dlt_{n_2,n-1})\lmd_{n_2+1}$.

Let $\mfk g$ be a finite-dimensional semisimple Lie algebra and let $U(\mfk g)$ be the universal enveloping algebra of $\mfk g$. Denote
$U_0(\mfk g)=\mbb F$ and for $m > 0$ with $m \in \mathbb{Z}$, $U_m(\mfk g)=\mbb F+\sum_{r=1}^m\mfk g^r$. Then $\{U_k(\mfk g)\mid k\in\mbb N\}$ forms a filtration of $U(\mfk g)$. Suppose that $M$ is a finitely generated $U(\mfk g)$-module and $M_0$ is a finite-dimensional subspace of $M$ satisfying
\begin{equation}\label{a1.4}
(U(\mfk g))(M_0)=M.\end{equation}
Setting $M_k =(U_k( \mathfrak{g}))(M_0)$, we get a filtration $\{M_k\mid k\in\mbb N\}$ of $M$.

The associated graded module is given by
\begin{equation}\label{a1.10}
\ol{M}=\mbox{gr}(M; M_0) = \bigoplus_{k = 0}^\infty \ol{M_k}\quad \text{with}\quad \ol{M}_k=M_k / M_{k-1},
\end{equation}
and it is a finitely generated graded $S( \mathfrak{g})$-module. Here we identify $S( \mathfrak{g})$ with $\mathbb{F}[\mathfrak{g}^*]$, the polynomial ring over $\mathfrak{g}^*$, in the canonical way through the Killing form of $  \mathfrak{g} $. The annihilator of $\ol{M}$ in $S( \mathfrak{g})$ is defined as
\begin{equation}\label{a1.11}
\operatorname{Ann}_{S( \mfk g)} (\ol{M})=\{ \xi \in  S( \mathfrak{g})\mid \xi(v)=0,\; \forall\; v \in \ol{M} \},\end{equation}
which is a graded ideal of $S( \mathfrak{g})$. It defines the associated variety in the dual space $\mathfrak{g}^*  $:
\begin{equation}
{\msr V}(M)=\{ f \in \mathfrak{g}^*\mid u(f) =0,\;\forall\; u\in \operatorname{Ann}_{S(\mathfrak{g})}(\ol{M})\}.\end{equation}
The variety does not depend on the choice of the subspace $M_0$ satisfying (\ref{a1.4}).
In this paper, we define the Hilbert polynomial and Hilbert series of the coordinate ring $S(\mfk g)/\sqrt{\operatorname{Ann}_{S( \mfk g)} (\ol{M}) }$ as  the Hilbert polynomial $\mfk p_{M}(k) $ and the Hilbert series $\mfk s_M(t)$ of the associated variety ${\msr V}(M)$.

The followings are our main theorems in this paper: \psp

{\bf Theorem 1}\quad
{\it If $n_1\leq n_2<n$, $\ell_1 \leq 0$ or $\ell_2 \leq 0$,
\begin{eqnarray}
\mfk s_{ {{\msr H}_{\la \ell_1,\ell_2 \ra}}}(t)
=
\frac{\sum\limits_{r,s=0}\limits^{\infty}
\left | \begin{matrix}
 \binom {n_2-2}{r}\binom {n - n_2 -1 }{r}  &
 \binom {n_2-2}{r-1}\binom {n - n_2  }{r}   \\
 \binom {n_1-1}{s+1}\binom {n-n_1-2}{s} &
 \binom {n_1-1}{s }\binom {n-n_1-1}{s}
\end{matrix} \right | t^{r+s}
}{(1-t)^{2n-3}}
\end{eqnarray}

If $n_1<n_2=n$,
\begin{equation}
\mfk s_{  {{\msr H}_{\la \ell_1,\ell_2 \ra}}}(t) =  \frac{  \sum_{r=0}^{\infty}\binom { n-n_2-1 }{r} \binom {n_1-1 }{r}t^r    }
{ \left(1-t\right)^{n-n_2+n_1-1} } .\end{equation} }\psp

As a consequence, we show that the arithmetic genus of $\msr V({\msr H}_{\la \ell_1,\ell_2 \ra})$ is 1, and its degree

 \begin{equation}
\deg (\msr V({\msr H}_{\la \ell_1,\ell_2 \ra}))=
\begin{cases}
\left( \frac{n-2}{n_1-1}-\frac{n-2}{n_2-1}\right) \binom {n-3}{n_1-2}\binom {n-3}{n_2-2}  ,\ \quad  \text{if }n_2\neq n \text{ and } n_1\neq n_2; \pse \\
 \frac{1}{n_1-1} \binom {n-4}{ n_1-2}\binom {n-3}{ n_1-2},\ \quad  \text{if }1 < n_1=n_2<n-1; \pse\\
\binom { n  -2}{n_1-1},\ \quad \text{if }   n_1\neq n_2=n ; \pse\\
1,\ \quad \text{if }  1=n_1=n_2<n \text{ or } n_1=n_2=n-1 .
\end{cases}
\end{equation}
The degree of the polynomial $\mfk p_{{\msr H}_{\la \ell_1,\ell_2 \ra}}(k)$ was obtained by Bai \cite{Bz} in term of Gelfand-Kirillov dimension. We difine the Hilbert polynomial of $M$ with respect to the generating subspace $M_0$ as $\mfk p_{M,M_0} =\dim( \ol{M}_k)$ for sufficiently large $k$.

\psp

{\bf Theorem 2}\quad
{\it Let $M_0$ be a subspace of the irreducible module ${\msr H}_{\la\ell_1,\ell_2\ra}$ in Proposition 1 satisfying (\ref{a1.4}).
Then we have $\mfk p_{M }(k) \leq \mfk p_{M,M_0}(k)$ for sufficiently large $k$. A necessary and sufficient condition for $\mfk p_{M }(k)=\mfk p_{M,M_0}(k)$ is that $M_0$ is certain one-dimensional subspace, such as:
$M_0=\mathbb{F}y_n^{-\ell_2}$ when $n_2=n-1$, $\ell_1 = 0$ and $\ell_2\leq 0$;
$M_0=\mathbb{F} x_{1}^{ -\ell_1} y_{n_2}^{ \ell_2}$ when $n_1+1=n_2=2$, $\ell_1 \leq 0$ and $\ell_2\geq 0$; $M_0=  \mbb F (x_ny_{n-1}-x_{n-1}y_n )^{ \ell_1}$ when $n_1=n_2=n-2$, $\ell_1 =-\ell_2 > 0$ etc.
 }\psp

For a polynomial $f(k)$, we denote by $lc(f)$ the leading coefficient of $f(k)$. We find that $lc (\mfk p_{M,M_0} )$ does not depend on the choice of the subspace $M_0$ satisfying (\ref{a1.4}). Thus we can simply redenote $lc(M)=lc (\mfk p_{M,M_0} )$.\psp

{\bf Theorem 3}\quad
{ \it
For the irreducible module ${\msr H}_{\la\ell_1,\ell_2\ra}$ in Proposition 1, we have

\begin{eqnarray}
\frac{lc({\msr H}_{\la\ell_1,\ell_2\ra})}{lc(\mfk p_{{\msr H}_{\la\ell_1,\ell_2\ra}})}=
\begin{cases}
1  & \text{ if }   n_1<n_2<n \text{ or } \ell_1=\ell_2=0; \pse\\
1-\ell_1-\ell_2& \text{ if }  1<n_1=n_2<n-1; \pse\\
\binom{ n+\ell_2 -2}{ \ell_2 } &\text{ if } n_1<n_2=n ; \pse\\
\binom{ n-\ell_2 -2}{-\ell_2 } &\text{ if } 1=n_1=n_2<n,   \ell_2\leq 0\text{ and }\ell_1+\ell_2 < 0 ; \pse\\
\binom{ n-\ell_1 -2}{-\ell_1 } &\text{ if } 1<n_1=n_2=n-1, \ell_1\leq 0\text{ and }\ell_1+\ell_2 < 0; \pse\\
\binom{ n+\ell_1 -3}{\ell_1 } &\text{ if } 1=n_1=n_2<n-1  \text{ and } \ell_2=-\ell_1<0  ; \pse\\
\binom{ n+\ell_2 -3}{\ell_2 } &\text{ if } 1<n_1=n_2=n-1  \text{ and }\ell_1 =-\ell_2<0 .
\end{cases}
\end{eqnarray}  }\psp

For instance, when $1<n_1=n_2<n-1$ and $\ell_1=\ell_2=-1$, we have $ lc({\msr H}_{\la\ell_1,\ell_2\ra})=3\, lc(\mfk p_{{\msr H}_{\la\ell_1,\ell_2\ra}}) $, and then $ \mfk p_{{\msr H}_{\la\ell_1,\ell_2\ra},M_0}>\mfk p_{{\msr H}_{\la\ell_1,\ell_2\ra}} $ for any $M_0$ satisfying (\ref{a1.4}).
When $n_1=2$, $n_2=3$, $n=5$ and $\ell_1=\ell_2=-1$, we have $ lc({\msr H}_{\la\ell_1,\ell_2\ra})=  lc(\mfk p_{{\msr H}_{\la\ell_1,\ell_2\ra}}) $, but $ \mfk p_{{\msr H}_{\la\ell_1,\ell_2\ra},M_0}>\mfk p_{{\msr H}_{\la\ell_1,\ell_2\ra}} $ holds for any $M_0$ satisfying (\ref{a1.4}).

The structure of the paper is as follows. In
Section 2 we prove Theorem 1 and (1.16).
Section 3 is devoted to the proof of Theorem 2. Finally we present the proof of Theorem 3 in Section 4.

\section{Hilbert Polynomial of $\msr V(M)$}
In the rest of this paper, we mainly consider the representation $\td\pi$ of $sl(n)$ obtained under the Fourier transformations (\ref{a1.6}), whose representation formulas are given by
\begin{equation}
\td{\pi}(E_{i,j})=E_{i,j}^x-E_{j,i}^y\qquad\for\;\;i,j\in\ol{1,n}\end{equation} with
\begin{equation}E_{i,j}^x=\left\{\begin{array}{ll}-x_j\ptl_{x_i}-\delta_{i,j}&\mbox{if}\;
i,j\in\ol{1,n_1};\\ \ptl_{x_i}\ptl_{x_j}&\mbox{if}\;i\in\ol{1,n_1},\;j\in\ol{n_1+1,n};\\
-x_ix_j &\mbox{if}\;i\in\ol{n_1+1,n},\;j\in\ol{1,n_1};\\
x_i\partial_{x_j}&\mbox{if}\;i,j\in\ol{n_1+1,n}
\end{array}\right.\end{equation}
and
\begin{equation}
E_{i,j}^y=\left\{\begin{array}{ll}y_i\ptl_{y_j}&\mbox{if}\;
i,j\in\ol{1,n_2};\\ -y_iy_j&\mbox{if}\;i\in\ol{1,n_2},\;j\in\ol{n_2+1,n};\\
\ptl_{y_i}\ptl_{y_j} &\mbox{if}\;i\in\ol{n_2+1,n},\;j\in\ol{1,n_2};\\
-y_j\partial_{y_i}-\delta_{i,j}&\mbox{if}\;i,j\in\ol{n_2+1,n}.
\end{array}\right.\end{equation}
Moreover, we adopt the notion
\begin{equation}\xi(f)=\td{\pi}(\xi)(f)\qquad\for\;\;\xi\in sl(n),\;f\in\msr A.\end{equation}
Denote
\begin{equation}x^\al=x_1^{\al_1}x_2^{\al_2}\cdots x_n^{\al_n},\;\;y^\al=y_1^{\al_1}y_2^{\al_2}\cdots y_n^{\al_n}\qquad \for\;\;\al=(\al_1,\al_2,...,\al_n)\in \mbb N^n.\end{equation}

When $n_1<n_2$, we define the linear operator $T$ on $\msr A$ by
\begin{align}
T(x^\al y^\be)=\sum_{i=0}^{\infty} \frac{\left(x_{n_1+1} y_{n_1+1}\right)^i
	 \left(\tilde{\Delta}+\partial_{x_{n_1+1}}\partial_{ y_{n_1+1}}\right)^i }
 {\prod_{r=1}^i \left(\alpha_{n_1+1} +r\right) \left(\beta_{n_1+1}+ r\right)}(x^\al y^\be)
\end{align}
 for $\al,\be\in \mbb N^n$.
 Then the set
\begin{align}
\left\{ T\left(x^\alpha y^\beta \right) \mid \alpha,\beta  \in \mathbb{N}^n,\alpha_{n_1+1}\beta_{n_1+1}=0,x^\alpha y^\beta \in {\msr A}_{\la\ell_1,\ell_2\ra}\right\}
\end{align}forms a basis of the module ${\msr H}_{\la\ell_1,\ell_2\ra}$, as shown in \cite{LX}. Furthermore, we introduce the notion
\begin{eqnarray}
&&
N\begin{pmatrix}
k_{12} & k_{12} & k_{13} \\
k_{21} & k_{22} &  k_{23}
\end{pmatrix} \nonumber\\ &=&
\left\{\left.   x^\alpha y^\beta \in \msr A_{\la\ell_1,\ell_2\ra}    \right\vert
  \sum_{r=1}^{n_1} \alpha_r =k_{11},\sum_{r=n_1+1}^{n_2} \alpha_r =k_{12},\sum_{r=n_2+1}^{n} \alpha_r =k_{13},
\right.  \nonumber\\  &&  \left.
 \sum_{r=1}^{n_1} \beta_r =k_{21}, \sum_{r=n_1+1}^{n_2} \beta_r =k_{22} ,\alpha_{n_1+1}\beta_{n_1+1}=0
\right\}
\end{eqnarray} and
\begin{align}
TN\begin{pmatrix}
k_{12} & k_{12} & k_{13} \\
k_{21} & k_{22} &  k_{23}
\end{pmatrix}=\text{Span} \left\{T(f)\mid f\in N\begin{pmatrix}
k_{12} & k_{12} & k_{13} \\
k_{21} & k_{22} &  k_{23}
\end{pmatrix}\right\}\end{align}
where $k_{11},k_{12},k_{13},k_{21},k_{22},k_{23} \in \mathbb{N}$.

Define the function ${\mfk d } : {\msr A}_{\la\ell_1,\ell_2\ra}  \rightarrow \mathbb{N}$ by
\begin{eqnarray}
& &{\mfk d} \left(x^\alpha y^\beta \right) =2 \sum_{i \in J_3} \alpha_i +\sum_{i \in J_2}\alpha_i +2\sum_{i \in J_1}\beta_i+\sum_{i \in J_2}\beta_i -\frac{\ell_1+|\ell_1|+\ell_2+|\ell_2|}{2}, \label{a2.46}\\
& &{\mfk d} \left( \lmd f\right) ={\mfk d} \left( f\right) , \  \forall \lmd   \in \mathbb{F} \setminus \{0\} ,\\
& &{\mfk d}\left(\sum_{i=1}^m  f_i\right) =\max \{{\mfk d}\left(f_1 \right) ,\cdots,{\mfk d}\left(f_m \right) \},\;\;\text{where $f_i$ are monomials.}
\label{a2.48}
\end{eqnarray}

We set
\begin{equation}
 N(k) =  \{  x^\alpha y^\beta    \mid
{\mfk d} \left(x^\alpha y^\beta  \right) = k ,x^\alpha y^\beta   \in {\msr A}_{\left \langle \ell_1,\ell_2  \right \rangle  }, \alpha_{n_1+1} \beta_{n_1+1} =0  \},\end{equation}
\begin{equation}
TN(k) = \text{Span} \{  T (x^\alpha y^\beta )      \mid
{\mfk d} \left(x^\alpha y^\beta  \right) = k ,x^\alpha y^\beta   \in {\msr A}_{\left \langle \ell_1,\ell_2  \right \rangle  }, \alpha_{n_1+1} \beta_{n_1+1} =0  \},\end{equation}
\begin{equation}
TN_k = \text{Span} \{  T (x^\alpha y^\beta )      \mid
{\mfk d} \left(x^\alpha y^\beta  \right)\leq k ,x^\alpha y^\beta   \in {\msr A}_{\left \langle \ell_1,\ell_2  \right \rangle  }, \alpha_{n_1+1} \beta_{n_1+1} =0  \}.\end{equation}

Denote
\begin{equation} Z = \{z_{j,i} \mid i \in J_1, j \in J_3 \} \end{equation}
 as a set of $n_1(n-n_2)$ variables.
Set $ \msr B=\mbb F[Z]$,
the polynomial algebra in $Z$. Denote by $Z^{(m)}$ the set of monomials with degree $m$.
Let $\msr C=  \mathbb{F}\left[x_i,y_j,Z \vert i,j\in \ol{1,n}\right]$.
Define an associative algebra homomorphism $\phi: \msr C \rta  \msr A$ by
\begin{equation}\phi(z_{j,i})=x_ix_j-y_iy_j,\;\;\phi(x_t)=x_t,\;\phi(y_t)=y_t\qquad i\in J_1,\;j\in J_3,t \in \ol{1,n}.\end{equation}

Set
\begin{eqnarray}\label{a2.11}
V_{0}=
\begin{cases}
   TN\begin{pmatrix}
-\ell_1 & 0 & 0 \\
0 & 0 & -\ell_2
\end{pmatrix} & \text{if } \ell_1\leq 0,\ell_2\leq 0  , \pse\\
TN\begin{pmatrix}
-\ell_1 & 0 & 0 \\
0 & \ell_2 & 0
\end{pmatrix} & \text{if } \ell_1\leq 0,\ell_2\geq 0 \text{ and } n_1<n_2<n, \pse\\
TN\begin{pmatrix}
0 & \ell_1 & 0 \\
0 & 0 & -\ell_2
\end{pmatrix} & \text{if } \ell_1\geq 0,\ell_2\leq 0 \text{ and } n_1<n_2<n, \pse\\
\text{Span} \left\{  TN\begin{pmatrix}
0 & \ell_1 & 0 \\
k_{21} & k_{22} & 0
\end{pmatrix}
\Big\vert
k_{21}+k_{22}=\ell_2
\right\}  &    \text{if } \ell_1,\ell_2 \geq 0 \text{ and } n_1<n_2=n , \pse\\
\text{Span} \left\{  TN\begin{pmatrix}
-\ell_1 & 0 & 0 \\
k_{21} & k_{22} & 0
\end{pmatrix}
\Big\vert
k_{21}+k_{22}=\ell_2
\right\}  &    \text{if } \ell_1\leq 0,\ell_2>0 \text{ and } n_1<n_2=n , \pse\\
\end{cases}
\quad
\end{eqnarray}
which is a finite-dimensional subspace of $M={\msr H}_{\la\ell_1,\ell_2\ra}$.
When $ 1<n_1=n_2<n $, $\ell_1 \leq 0$ and $\ell_2 >0$,  we set $\ell_1=-m_1-m_2$, $\ell_2= m_2$ with $m_1,m_2 \in \mathbb{N}$ and
\begin{eqnarray}\label{a5.11}
V_0 =  & &\mbox{Span}\Big\{  \left. \left[ \prod_{ 1\leq p<q\leq n_1} (x_py_q-x_q y_p)^{k_{p,q}}\right] \prod_{r=1}^{n_1}x_r^{\al_r} \right\vert
\nonumber \\
 & &\qquad \qquad
k_{p,q}\in \mathbb{N}, \sum_{ 1\leq p<q\leq n_1}k_{p,q} =m_2,\sum_{ 1\leq  i \leq n_1}\al_i =m_1
\Big\} .  \end{eqnarray}
Symmetrically, when $n_1=n_2<n-1 $, $\ell_1> 0$ and $\ell_2 \leq 0$,  we set $\ell_1= m_2$, $\ell_2= -m_1-m_2$ with $m_1,m_2 \in \mathbb{N}$ and
\begin{eqnarray} \label{f2.20}
V_0 =  & &\mbox{Span}\Big\{  \left. \left[ \prod_{ n_1+1\leq p<q\leq n } (x_py_q-x_q y_p)^{k_{p,q}}\right] \prod_{j=n_1+1}^{n }y_j^{\be_j} \right\vert
\nonumber \\
 & &\qquad \qquad
k_{p,q}\in \mathbb{N}, \sum_{n_1+ 1\leq p<q\leq n }k_{p,q} =m_2,\sum_{ n_1+1\leq  j \leq n }\be_j =m_1
\Big\} .  \end{eqnarray}

When $n_1<n_2$, $\ell_1\leq 0$ or $\ell_2\leq 0$,  we take $M_0=V_{0}$, and then we have
\begin{align}\label{a2.22}
M_k &= \text{Span}\{
TN\left(i\right),TN\left(k-i\right) \phi(Z^{(i)}) \mid
i=0,1,\dots,k\}
\nonumber\\
&=\text{Span}\{
TN_k,TN\left(k-i\right)\phi(Z^{(i)})  \mid
i=1,\dots,k\}.
\end{align}
 according to Proposition 2.1 in  \cite{ZX}.

Set
\begin{equation}\label{aa1.13}
J_1=\ol{1,n_1},\quad J_2=\ol{n_1+1,n_2},\quad J_3=\ol{n_2+1,n}.\end{equation}
 In \cite{ZX}, we determined the associated varieties $\mathscr{V}(\mathscr{H}_{\langle \ell_1, \ell_2 \rangle})$  as follows:
\begin{lemma}\label{lemc2.1}
\quad  If $n_1\leq n_2$, $\ell_1 \leq 0$ or $\ell_2 \leq 0$,
\begin{equation}\msr V({\msr H}_{\la\ell_1,\ell_2\ra})\cong \msr{V} _3(J_2\cup J_3, J_1\cup J_2)\cap \msr{V}_2(J_2 , J_1)\cap \msr{V}_2(J_3 , J_2) \cap\msr{V}_1(J_2 , J_2).\end{equation}
If $n_1<n_2=n$ and $\ell_1,\ell_2>0$,
\begin{equation}
\msr V( {\msr H}_{\la\ell_1,\ell_2\ra})\cong \msr{V}_2(J_2 , J_1).\end{equation}
\end{lemma}\pse

Let $\{(j_1,i_1),(j_2,i_2),...,(j_r,i_r)\}\subset \mbb Z^2$.
If there exist three distinct $1\leq k_1,k_2,k_3\leq r$ such that
$i_{k_1}<i_{k_2}< i_{k_3}$ and $j_{k_1}<j_{k_2}<j_{k_3}$, we say that $\{(j_1,i_1),(j_2,i_2),...,(j_r,i_r)\}$ contains a 3-chain $\left( j_{k_1},i_{k_1}\right) \prec   \left( j_{k_2},i_{k_2}\right)\prec  \left( j_{k_3},i_{k_3}\right) $. Similarly, we say that $\{(j_1,i_1),(j_2,i_2),...,(j_r,i_r)\}$ contains a 3-chain $\left( j_{k_1},i_{k_1}\right) \prec   \left( j_{k_2},i_{k_2}\right)  $ if $i_{k_1}<i_{k_2}$ and $j_{k_1}<j_{k_2}$.

Denote $\msr Z=\mathbb{F}\left[  z_{i,j}|1\leq i\leq m ,1\leq j\leq n\right]$ for $m,n \in \mathbb{N}$. Denote by $\msr Z^{(r)}$ the set of monomials in $\msr Z$ with degree $r$ for $r\in \mbb N$.

For $r \in \mbb N$, $a\in \mbb F$ and $\{(j_{t },i_{t })
 \in \ol{1,m} \times \ol{1,n} \vert  1 \leq t \leq r\}$, we set
\begin{eqnarray} \msr I\left( a \prod_{t =1}^{r} z_{i_{t },j_{t }} \right)  =
\{ (i_{t },j_{t })
\vert  1 \leq  t \leq r \}.\end{eqnarray}
For a monomial $f \in \msr Z^{(r)}$, we say that $f$ contains a 3-chain if $\msr I(f)$ contains a 3-chain.

 We denote the function $h_2 \left( m,n ,r\right)$ as the number of monomials in $\msr Z$ of degree $r$ without 2-chains, $h_3 \left( m,n ,r\right)$ as the number of monomials in $\msr Z$ of degree $r$ without 3-chains.

According to \cite{AJ94,Ku96}, the followings hold:
\begin{equation}\label{c2.10}
h_2 \left( m,n ,r\right) =\binom {m-1+r}{m-1} \cdot   \binom {n-1+r}{n-1}  ,
\end{equation}its generating function
\begin{equation}\label{c2.11}
{\msr S}^{m,n}_2(t)=\frac{    \sum_{r} \binom { m-1 }{r} \binom {n-1 }{r}t^r    }
{ \left(1-t\right)^{ m+n-1} },
\end{equation}
and the generating function
\begin{eqnarray}\label{b3.4}
{\msr S}_3^{m,n}(t)=\sum_{r=0}^{\infty} h_3 \left( m,n ,r\right)t^r
&= &
\frac{ \det \left(   \sum_{r} \binom { m-i }{r} \binom {n-j }{r+i-j}t^r\right)_{i,j=1,2}  }
{  \left(1-t\right)^{2\left( m+n-2\right)} }
\nonumber\\
&= &
\frac{ \det \left(   \sum_{r} \binom { m-i }{r} \binom {n-j }{r}t^r \right)_{i,j=1,2}  }
{ t \left(1-t\right)^{2\left( m+n-2\right)} }.
\end{eqnarray}
For convenience, we assume that ${\msr S}_3^{m,n}(t)={\msr S}_2^{m,n}(t)=0$ when $m\leq 0$ or $n\leq 0$.

For $m,m',n,n',r,\xi \in \mathbb{N}$, we define the set
\begin{eqnarray} \label{e2.31}
& &S_1^{\ol{m,m'},\ol{n,n'}}(\xi,r)  =  \Big\{ f=\prod_{t =1}^{r} z_{i_{t },j_{t }}\in \msr Z^{(r)} \vert
 m \leq i_t \leq m', n\leq j_t \leq n'; \nonumber \\& &
 \text{no 3-chains in $\msr I(f)$, no 2-chains in }   \{(i_t,j_t) \vert   j_t < \xi,t\in \ol{1,r}  \}
\Big\}
\end{eqnarray}
when $m \leq m'$ and $n \leq n'$. We also denote  $S_1^{\ol{m,m'},\ol{n,n'}}(\xi,r) =\emptyset$ when $m>m'$ or $n>n'$.
Denote by $P_1(m,n,\xi,r)$ the cardinality of $S_1^{\ol{1,m},\ol{1,n}}( \xi,r )$ for $m,n,\xi,r \in \mbb N$ with $\xi \in \ol{1,n}$.
In particular, we have
\begin{eqnarray}
P_1(m,n,\xi,r)=h_3(m,n,r) ,\quad \text{for} \,\, \xi \leq 2.
\end{eqnarray}
and
\begin{eqnarray}
P_1(m,n,\xi,r)=h_2(m,n,r) ,\quad \text{for} \,\, \xi > n.
\end{eqnarray}

Set
\begin{eqnarray}
{\msr S}^{m,n, \xi}(t)=
\sum_{r=0}^{\infty} P_1(m,n, \xi,r)  t^r,
\end{eqnarray}
then we have a recursion formula as follows:
\begin{lemma}\label{leme2.2}
For $m,n,r \in \mbb N$, $\xi \in \ol{3,n}$ with $m,n \geq 2$ and $r>1$, we have
\begin{eqnarray}\label{c2.17}
& &P_1(m,n,\xi,r) \nonumber\\
&=&
\Big(\sum_{r_1=1}^r \sum_{\xi'=1}^m
\binom { m-\xi'-1 +r_1}{r_1-1} \binom {\xi-2+r_1 }{r_1} P_1 ( n-\xi+1,m,m-\xi'+1,r-r_1 )\Big)
\nonumber\\
& &
+h_3(m,n-\xi+1,r).
\end{eqnarray}
The series form of this formula is
\begin{eqnarray} \label{c2.18}
 {\msr S}^{m,n, \xi}(t)
&=&   \sum_{\xi'=1}^{m}\Big(
 {\msr S}^{ n-\xi+1,m,m-\xi'+1}(t)
({\msr S}^{m-\xi'+1,\xi-1}_2(t)-{\msr S}^{m-\xi' ,\xi-1}_2(t))   \Big)
\end{eqnarray}
\end{lemma}
\begin{proof}
For a monomial $f\in  S_1^{\ol{1,m},\ol{1,n}}(\xi,r) $, we write
\begin{equation}
f=\prod_{1\leq i \leq m,1\leq j\leq n}z_{i,j}^{\gamma_{i,j}}=f_1 f_2
\end{equation}
where
\begin{equation}
f_1=\prod_{1\leq i \leq m,1\leq j\leq \xi-1 }z_{i,j}^{\gamma_{i,j}},\quad f_2=\prod_{1\leq i \leq m,\xi \leq j\leq n }z_{i,j}^{\gamma_{i,j}}
\end{equation}
Setting
\begin{equation}
r_1(f)=\deg(f_1)=\sum_{1\leq i \leq m,1\leq j<\xi } \gamma_{i,j},
\end{equation}
then we have
\begin{equation}
| \{f\in S_1^{\ol{1,m},\ol{1,n}}(\xi,r)\vert r_1(f)=0 \}| = h_3(m,n-\xi+1,r).
\end{equation}

For  $f\in  S_1^{\ol{1,m},\ol{1,n}}(\xi,r)  $ and $r_1(f)>0$, we set
\begin{equation}
\xi' = \min\{ i \in \ol{1,m} \vert \gamma_{i,j}>0 \text{ for some }j\in \ol{1,\xi-1}\},
\end{equation}
and then there exists $j_0\in \ol{1,\xi-1}$ such that $\gamma_{\xi' ,j_0}>0$.
Thus
\begin{equation}
f_1 \in S_1^{\ol{\xi',m},\ol{1,\xi-1}}(\xi ,r_1(f)) \setminus S_1^{\ol{\xi'+1,m},\ol{1,\xi-1}}(\xi ,r_1(f)).
\end{equation}
In particular, $f_1 \in S_1^{\ol{m,m},\ol{1,\xi-1}}(\xi ,r_1(f))$ when $\xi'=m$.

When there exists a 2-chain $(i_1,j_1) \prec (i_2,j_2)$ in $\msr I(f)\cap \{(i,j)  |  i> \xi' ,j\geq \xi \}$, $(\xi',j_0)\prec (i_1,j_1) \prec (i_2,j_2)$ is a 3-chain in $\msr I(f) $. Thus there are no 2-chains in $\msr I(f_2)\cap \{(i,j)  |  i > \xi'   \}$ and no 3-chain in $\msr I(f_2) $.

For $f =\prod_{1\leq i \leq m,1\leq j\leq n}z_{i,j}^{\gamma_{i,j}} \in \msr Z^{(r)}$, we denote \begin{equation}f^t =\prod_{1\leq i \leq m,1\leq j\leq n}z_{j,n-i+1}^{\gamma_{i,j}}\end{equation} in $\mbb F\left[  z_{j,i}|1\leq i\leq m ,1\leq j\leq n\right]$. Then $(\cdot )^t$ is a one-to-one mapping from $\msr Z^{(r)}$ to the set of monomials in $\mbb F\left[  z_{j,i}|1\leq i\leq m ,1\leq j\leq n\right]$ with degree $r$.
Hence $f_2^t \in S_1^{\ol{\xi,n},\ol{1,m} }(m-\xi'+1,r-r_1(f))$.

On the other hand, when $g_1\in  S_1^{\ol{\xi',m},\ol{1,\xi-1}}(\xi ,r_1 ) \setminus S_1^{\ol{\xi'+1,m},\ol{1,\xi-1}}(\xi ,r_1 )$ and $g_2^t \in S_1^{\ol{\xi,n},\ol{1,m} }(m-\xi'+1,r-r_1 )$ for some $\xi' \in \ol{1,m}$ and $r_1\in \ol{1,r}$, we have $g_1g_2 \in S_1^{\ol{1,m},\ol{1,n}}(\xi,r)$.

Therefore,
\begin{eqnarray}
&  &P_1(m,n,\xi,r)\nonumber\\& =&
\Big(\sum_{r_1=1}^r  \sum_{\xi'=1}^m
| S_1^{\ol{\xi,n},\ol{1,m} }(m-\xi'+1,r-r_1 )|
|  S_1^{\ol{\xi',m},\ol{1,\xi-1}}(\xi ,r_1 ) \setminus S_1^{\ol{\xi'+1,m},\ol{1,\xi-1}}(\xi ,r_1 ) | \Big)\nonumber\\& &
+ h_3(m,n-\xi+1,r)\nonumber\\
&=& \Big(\sum_{r_1=1}^r  \sum_{\xi'=1}^m
| S_1^{\ol{1,n-\xi+1},\ol{1,m} }(m-\xi'+1,r-r_1 )|
\left( h_2(m-\xi'+1,\xi-1,r_1)-h_2(m-\xi',\xi-1,r_1)\right)\Big)\nonumber\\& &
+ h_3(m,n-\xi+1,r)\nonumber\\
&=&\Big(\sum_{r_1=1}^r \sum_{\xi'=1}^m
\binom { m-\xi'-1 +r_1}{r_1-1} \binom {\xi-2+r_1 }{r_1} P_1 ( n-\xi+1,m,m-\xi'+1,r-r_1 )\Big)
\nonumber\\
& &
+h_3(m,n-\xi+1,r).
\end{eqnarray}
Thus, equation (\ref{c2.17}) holds.
The above equation can be rewritten in the form of a generating function, leading to (\ref{c2.18}).

\end{proof}

Note that $\deg(h_3(m,n))=2m+2n-5$.
By induction on $m+n$, it is easy to prove the following corollary of Lemma \ref{leme2.2}.
\begin{corollary}\label{core2.1}
For $m,n \in \mbb N$, $\xi \in \ol{1,n}$ with $m,n \geq 2$, we have
 \begin{equation}
\deg_r(P_1(m,n,\xi,r))=
\begin{cases}
2m+2n-5,\ \quad \text{if }\xi=1; \\
2m+2n -\xi-3,\ \quad \text{if }\xi\geq 2.
\end{cases}
\end{equation}
\end{corollary}
\psp

For $m,m',n,n',r,\varrho,\xi \in \mathbb{N}$, we define the set
\begin{eqnarray}
& &S_2^{\ol{m,m'},\ol{n,n'}}(\varrho,\xi,r) =   \Big\{ f=\prod_{t =1}^{r} z_{i_{t },j_{t }}\in \msr Z^{(r)} \vert
 m \leq i_t \leq m', n\leq j_t \leq n',  \nonumber \\& & i_t \leq \varrho \text{ or } j_t\geq \xi \text{ for } t\in \ol{1,r};
 \text{ no 3-chains in }  \msr I(f),  \text{no 2-chains in } \nonumber \\& &  \{(i_t,j_t) \vert   j_t < \xi,t\in \ol{1,r}\} \text{ and }\{ (i_t,j_t) \vert   i_t > \varrho,t\in \ol{1,r} \}
\Big\}
\end{eqnarray}
when $m \leq m'$ and $n \leq n'$, and $S_2^{\ol{m,m'},\ol{n,n'}}(\varrho,\xi,r) =\emptyset$ when $m>m'$ or $n>n'$.
Denote by $P_2(m,n,\varrho,\xi,r)$ the cardinality of $S_2^{\ol{1,m },\ol{1,n }}(\rho,\xi,r)$ for $m,n,\varrho,\xi,r \in \mbb N$ with $\varrho\in \ol{1,m}$ and $\xi \in \ol{1,n}$.
In particular, \begin{equation}\label{c2.28}
P_2(m,n,m, \xi,r)=P_1(m,n,   \xi,r) .
\end{equation}
For $f =\prod_{1\leq i \leq m,1\leq j\leq n}z_{i,j}^{\gamma_{i,j}} \in S_2^{\ol{m,m'},\ol{n,n'}}(\varrho,\xi,r)$, we set the one-to-one mapping $(\cdot)^{T'}$ from $ S_2^{\ol{1,m },\ol{1,n }}(\varrho,\xi,r)$ to $ S_2^{\ol{1,n},\ol{1,m}}(n-\xi+1,m-\varrho+1,r)$ by
\begin{equation}f^{T'} =\prod_{1\leq i \leq m,1\leq j\leq n}z_{n-j+1,m-i+1}^{\gamma_{i,j}}.\end{equation}
Hence
\begin{equation}\label{c2.30}
P_2(m,n,\varrho, \xi,r)=P_2(n,m,n-\xi+1,m-\varrho+1,r) .
\end{equation}
By (\ref{c2.28}) and (\ref{c2.30}),
 \begin{equation}
P_2(m,n,\varrho,1,r)=P_2(n,m,n ,m-\varrho+1,r)= P_1(n,m,m-\varrho+1,r).
\end{equation}

Denote
\begin{eqnarray} \label{c2.32}
{\msr S}^{m,n, \varrho,\xi}(t)=
\sum_{r=0}^{\infty} P_2(m,n, \varrho,\xi,r)  t^r,
\end{eqnarray}
then we have the recursion formula as follows:

\begin{lemma}\label{lemd2.3}
For $m,n\in \mbb N$, $\varrho \in \ol{1,m-1}$, $\xi \in \ol{2,n}$ and $r>1$, we have
\begin{eqnarray} \label{c2.33}
 & & P_2(m,n,\varrho,\xi,r) \nonumber \\
&=  &
\sum_{\varrho'=1}^{\varrho} \sum_{\xi'=\xi}^{n}
\bigg( \sum_{r_1,r_2>0;r_1+r_2\leq r}
\binom {\varrho-\varrho'-1 +r_2}{r_2-1} \binom {\xi-2+r_2 }{r_2}
\binom {m-\varrho-1 +r_1 }{r_1 }
\nonumber \\
& &\;\;\;
 \binom {\xi'-\xi-1+r_1 }{r_1-1}   P_2(\varrho,n-\xi+1,\varrho',\xi'-\xi+1,r-r_1-r_2)    \bigg)
  +P_1 ( \varrho ,n,\xi,r )
\nonumber \\
& &\;\;\;
 +P_1(n-\xi+1,m,m-\varrho+1,r)
-h_3( \varrho ,n-\xi +1,r)
.\end{eqnarray}
The series form of this formula is
\begin{eqnarray}  \label{c2.34}
{\msr S}^{m,n, \varrho,\xi}(t)
&=&
\sum_{\varrho'=1}^{\varrho} \sum_{\xi'=\xi}^{n}
 {\msr S}^{\varrho,n-\xi+1,\varrho',\xi'-\xi+1}(t)
({\msr S}^{\varrho-\varrho'+1,\xi-1}_2(t)-{\msr S}^{\varrho-\varrho' ,\xi-1}_2(t))
\nonumber \\& &\qquad\qquad
({\msr S}^{m-\varrho,\xi'-\xi+1}_2(t)-{\msr S}^{m-\varrho,\xi'-\xi}_2(t)) .
\end{eqnarray}
\end{lemma}

\begin{proof}
 For $f\in S_2^{\ol{1,m},\ol{1,n}}(\varrho,\xi,r)$, we write
\begin{equation}
 f=\prod_{1\leq i \leq m,1\leq j\leq n}z_{i,j}^{\gamma_{i,j}}.
\end{equation}
We set
\begin{equation}\label{c2.36}
 r_1(f)=\sum_{\varrho<i\leq m, 1\leq j\leq n } \gamma_{i,j}, \;\;\;
 r_2(f)=\sum_{ 1\leq i\leq m,1\leq j<\xi  } \gamma_{i,j} .
\end{equation}

Then we have
\begin{equation}\label{c2.37}
| \{f\in S_2^{\ol{1,m},\ol{1,n}}(\varrho,\xi,r)\vert  r_1(f)=r_2(f)=0\}|
= h_3(\varrho, n-\xi+1,r) .
\end{equation}
Since
\begin{equation}
\{f\in S_2^{\ol{1,m},\ol{1,n}}(\varrho,\xi,r)\vert  r_1(f)= 0\}=
 S_1^{\ol{1,\varrho},\ol{1,n}}(\xi, r),
\end{equation}
\begin{eqnarray}\label{c2.39}
| \{f\in S_2^{\ol{1,m},\ol{1,n}}(\varrho,\xi,r)\vert  r_1(f)= 0\}|= P_1 ( \varrho ,n,\xi,r )  .
\end{eqnarray}
Similarly,
\begin{equation}\label{c2.40}
| \{f\in S_2^{\ol{1,m},\ol{1,n}}(\varrho,\xi,r)\vert  r_2(f)=0\}|
=P_1(n-\xi+1,m,m-\varrho+1,r).
\end{equation}

Suppose that $r_1(f),r_2(f)>0$. We denote
\begin{equation}\label{c2.41}
\xi'=\max\{ j\in \ol{\xi,n} \vert \gamma_{i,j}>0 \text{ for some }i\in \ol{\varrho+1,m}\}
\end{equation}
and
\begin{equation}
\varrho'=\max\{ i\in \ol{1,\varrho} \vert \gamma_{i,j}>0 \text{ for some }j\in \ol{1,\xi-1}\}.
\end{equation}
Then there exist $i_0 \in \ol{\varrho+1,m}$ and $j_0\in \ol{1,\xi-1}$ such that $\gamma_{\varrho' ,j_0}>0$ and $\gamma_{i_0,\xi' }>0$.
We set
\begin{equation}
f_1=\prod_{\varrho<i\leq m, 1\leq j\leq n }z_{i,j}^{\gamma_{i,j}},\quad
f_2=\prod_{ 1\leq i\leq m,1\leq j<\xi  }z_{i,j}^{\gamma_{i,j}} ,\quad
f_3=\prod_{1\leq i \leq \varrho,\xi \leq j\leq n }z_{i,j}^{\gamma_{i,j}},
\end{equation}
and then $f=f_1f_2f_3$.
By (\ref{c2.36}) and (\ref{c2.41}), we have
\begin{equation}
f_1 \in S_1^{\ol{\varrho+1,m},\ol{\xi,\xi'}}(n+1,r_1(f)) \setminus S_1^{\ol{\varrho+1,m},\ol{\xi,\xi'-1}}(n+1,r_1(f)).
\end{equation}
Similarly,
\begin{equation}
f_2 \in S_1^{\ol{\varrho',\varrho},\ol{1,\xi-1}}(n+1,r_2(f)) \setminus S_1^{\ol{\varrho'+1,\varrho},\ol{1,\xi-1}}(n+1,r_2(f)).
\end{equation}

When there exists $\gamma_{i',j'}>0$ for some $i'\in \ol{\varrho'+1,\varrho}$ and $j'\in \ol{\xi,\xi'-1}$, $(\varrho',j_0) \prec (i',j') \prec (i_0,\xi')$ is a 3-chain in $\msr I(f)$. Hence $\{(i,j)\in \msr I(f_3) \vert i>\varrho'  ,j< \xi'  \}=\emptyset$.
When there exists a 2-chain $(i_1,j_1) \prec (i_2,j_2)$ in $  \{(i,j)  \in \msr I(f_3)|  i> \varrho'  \}$, $(\varrho',j_0)\prec (i_1,j_1) \prec (i_2,j_2)$ is a 3-chain in $\msr I(f) $. Thus there are no 2-chains in $ \{(i,j)  \in \msr I(f_3)|  i> \varrho'  \}$.
Symmetrically, there are no 2-chains in $ \{(i,j)  \in \msr I(f_3)|  j< \xi'  \}$.
Hence,
\begin{equation}
f_3 \in S_2^{\ol{1,\varrho},\ol{\xi,n}}(\varrho',\xi',r-r_1(f)-r_2(f)) .
\end{equation}

On the other hand, when $g_1\in   S_1^{\ol{\varrho+1,m},\ol{\xi,\xi'}}(n+1,r_1(f)) \setminus S_1^{\ol{\varrho+1,m},\ol{\xi,\xi'-1}}(n+1,r_1(f)) $ and $g_2  \in S_1^{\ol{\varrho',\varrho},\ol{1,\xi-1}}(n+1,r_2(f)) \setminus S_1^{\ol{\varrho'+1,\varrho},\ol{1,\xi-1}}(n+1,r_2(f))$ and $g_3 \in S_2^{\ol{1,\varrho},\ol{\xi,n}}(\varrho',\xi',r-r_1(f)-r_2(f))$ for some $\varrho' \in \ol{1,\varrho}$, $\xi' \in \ol{\xi,n}$; $r_1,r_2\in \ol{1,r}$ with $r_1+r_2\leq r$, we have $g_1g_2g_3 \in S_2^{\ol{1,m},\ol{1,n}}(\varrho,\xi,r)$.

Therefore,
\begin{eqnarray} \label{c2.47}
&  &
| \{f\in S_2^{\ol{1,m},\ol{1,n}}(\varrho,\xi,r)\vert  r_1(f),r_2(f)>0\}|
\nonumber\\& =&
\sum_{r_1,r_2>0,r_1+r_2\leq r}   \sum_{\xi'=\xi}^m \sum_{\varrho'=1}^{\varrho}\Big(
|  S_1^{\ol{\varrho+1,m},\ol{\xi,\xi'}}(n+1,r_1 ) \setminus S_1^{\ol{\varrho+1,m},\ol{\xi,\xi'-1}}(n+1,r_1 ) |
\nonumber\\& &
| S_1^{\ol{\varrho',\varrho},\ol{1,\xi-1}}(n+1,r_2 ) \setminus S_1^{\ol{\varrho'+1,\varrho},\ol{1,\xi-1}}(n+1,r_2 )|
| S_2^{\ol{1,\varrho},\ol{\xi,n}}(\varrho',\xi',r-r_1 -r_2 )|\Big)
\nonumber\\
&=&\sum_{r_1,r_2>0;r_1+r_2\leq r}   \sum_{\xi'=\xi}^m \sum_{\varrho'=1}^{\varrho}\Big(
(h_2(m-\varrho, \xi'-\xi+1,r_1 )- h_2(m-\varrho, \xi'-\xi ,r_1 ) )
\nonumber\\& &
(h_2( \varrho-\varrho'+1,  \xi-1,r_2 )- h_2(\varrho-\varrho', \xi -1 ,r_2 ) )
| S_2^{\ol{1,\varrho},\ol{1,n-\xi+1}}(\varrho',\xi'-\xi+1,r-r_1 -r_2 )|\Big)
\nonumber\\
&=&\sum_{\varrho'=1}^{\varrho} \sum_{\xi'=\xi}^{n}
\bigg( \sum_{r_1,r_2>0;r_1+r_2\leq r}
\binom {\varrho-\varrho'-1 +r_2}{r_2-1} \binom {\xi-2+r_2 }{r_2}
\binom {m-\varrho-1 +r_1 }{r_1 }
\nonumber \\
& &\;\;\;
 \binom {\xi'-\xi-1+r_1 }{r_1-1}   P_2(\varrho,n-\xi+1,\varrho',\xi'-\xi+1,r-r_1-r_2)    \bigg)
\end{eqnarray}
Since
\begin{eqnarray}
& & P_2(m,n,\varrho,\xi,r)   \nonumber \\
& =&| \{f\in S_2^{\ol{1,m},\ol{1,n}}(\varrho,\xi,r)\vert  r_1(f),r_2(f)>0\}|-
| \{f\in S_2^{\ol{1,m},\ol{1,n}}(\varrho,\xi,r)\vert  r_1(f)=r_2(f)=0\}|
\nonumber \\
& & +
| \{f\in S_2^{\ol{1,m},\ol{1,n}}(\varrho,\xi,r)\vert   r_2(f)=0\}|
+| \{f\in S_2^{\ol{1,m},\ol{1,n}}(\varrho,\xi,r)\vert  r_1(f)=0\}|,
\end{eqnarray}
  (\ref{c2.33}) holds by (\ref{c2.37}),(\ref{c2.39}),(\ref{c2.40}) and (\ref{c2.47}).
Moreover, Equation (\ref{c2.33}) can be rewritten as a generating function, yielding (\ref{c2.34}).

\end{proof}
\pse
By induction on $m+n$, it is easy to prove the following corollary of Lemma \ref{lemd2.3}.

\begin{corollary}\label{core2.2}
For $m,n \in \mbb N$, $\xi \in \ol{1,n}$ and $\varrho \in \ol{1,m}$ with $m,n \geq 2$, we have
 \begin{equation}
\deg_r(P_2(m,n,\varrho,\xi,r))=
\begin{cases}
2m+2n-5,\ \quad &\text{if } \xi=1 \text{ and } \varrho=m; \\
 m+2n+\varrho-\xi-3,\ \quad &\text{if } \xi > 1 \text{ or } \varrho<m.
\end{cases}
\end{equation}
\end{corollary}

\begin{lemma} \label{lemd2.4}
\begin{enumerate}
For $p,q,n\in \mbb N$,
  \item[(1)]  Vandermonde's identity:
\begin{equation}
\sum_{r=0}^n \binom {p}{r}\binom {q}{n-r}=\binom {p+q}{n}
\end{equation}
\item [(2)]
\begin{equation}
\sum_{r=0}^n \binom {r}{p}\binom {n-r}{q}=\binom {n+1}{p+q+1}
\end{equation}
\item [(3)]
\begin{equation}\label{d2.51}
\sum_{r=0}^{\infty } \binom {p}{r+n}\binom {q}{r}=\binom {p+q}{q+n}
\end{equation}
\item [(4)]
\begin{equation} \label{f2.69}
\sum_{r=0}^{n} \binom {p+r}{p}\binom {n-r}{q}=\binom {p+n+1}{p+q+1}
\end{equation}
\item [(5)]
\begin{equation} \label{f2.70}
\sum_{r=0}^{n} \binom {p+r}{p}\binom {n-r+q}{q}=\binom {p+q+n+1}{p+q+1}
\end{equation}
\end{enumerate}
\end{lemma}
\begin{proof}
(1) is the Vandermonde's identity which is well-known and (2) can be easily proved by induction on $n$.
When $p=1$, it is straightforward to verify that (\ref{d2.51}) holds.
Since
\begin{equation}
\binom {p+1}{r+n}=\binom {p}{r+n}-\binom {p}{r+n-1},
\end{equation}
we can prove that (\ref{d2.51}) holds by induction on $p$.
Similarly, we can prove (\ref{f2.69}) and (\ref{f2.70}) by induction on $n$.

\end{proof}

\begin{proposition} \label{propd2.1}
For $m,n,\varrho,\xi \in \mbb N$ with $1\leq  \varrho\leq m $ and $1\leq \xi \leq n$, we have
 \begin{eqnarray}\label{d2.53}
{\msr S}^{m,n, \varrho,\xi}(t)
=
\frac{
\left | \begin{matrix}
\sum_{r=0}^{\infty}\binom {m-2}{r}\binom {n-\xi}{r}t^r &  \sum_{r=0}^{\infty}\binom {m-2}{r-1}\binom {n-\xi+1}{r}t^r  \\
\sum_{r=0}^{\infty}\binom {\varrho-1}{r+1}\binom {n-2}{r}t^r & \sum_{r=0}^{\infty}\binom {\varrho-1}{r }\binom {n-1}{r}t^r
\end{matrix} \right |
}{(1-t)^{m+2n+\varrho-\xi-2}}
\end{eqnarray}
In particular,
 \begin{eqnarray}
{\msr S}^{m,n,  \xi}(t)
=
\frac{
\left | \begin{matrix}
\sum_{r=0}^{\infty}\binom {m-2}{r }\binom {n-\xi}{r}t^r  &  \sum_{r=0}^{\infty}\binom {m-1}{r }\binom {n-2}{r-1}t^r  \\
\sum_{r=0}^{\infty}\binom {m-2}{r}\binom {n-\xi+1}{r+1}t^r & \sum_{r=0}^{\infty}\binom {m-1}{r}\binom {n-1}{r}t^r
\end{matrix} \right |
}{(1-t)^{2m+2n -\xi-2}}
\end{eqnarray}
\end{proposition}
\begin{proof}
We prove (\ref{d2.53}) by induction on $m+n$.
When $m+n=1$, we have $m=n=\varrho=\xi=1$, and then
\begin{eqnarray}
{\msr S}^{1,1,1,1}(t)
=\frac{1}{1-t},
\end{eqnarray}
and (\ref{d2.53}) holds.
We assume that (\ref{d2.53}) holds when $m+n< l$ and consider $m+n= l$.

When $\varrho=m$ and $\xi=1$, we have
\begin{eqnarray}
{\msr S}^{m,n, m,1}(t)={\msr S}_3^{m,n }(t),
\end{eqnarray}
then (\ref{d2.53}) holds by Corollary 1 in \cite{AJ94}.

Suppose that $ \varrho<m$ or $\xi>1$, then $\varrho+n-\xi+1<m+n =l$.
By the inductional assumption, for $\xi'\in \ol{\xi,n}$ and $\varrho'=\ol{1,\varrho}$, we have
\begin{eqnarray} \label{d2.56}
{\msr S}^{\varrho,n-\xi+1,\varrho',\xi'-\xi+1}(t)
=
\frac{
\left | \begin{matrix}
\sum_{r=0}^{\infty}\binom {\varrho-2}{r}\binom {n-\xi'}{r}t^r &  \sum_{r=0}^{\infty}\binom {\varrho-2}{r-1}\binom {n-\xi'+1}{r}t^r  \\
\sum_{r=0}^{\infty}\binom {\varrho'-1}{r+1}\binom {n-\xi-1}{r}t^r & \sum_{r=0}^{\infty}\binom {\varrho'-1}{r }\binom {n-\xi}{r}t^r
\end{matrix} \right |
}{(1-t)^{ 2n+\varrho+\varrho'-\xi-\xi'-1}}
\end{eqnarray}

Denote
\begin{equation}
\mfk f_1^{\xi,\varrho,\varrho'}=\sum_{r=1}^{\infty}
\binom {\varrho-\varrho'-1}{r-1} \binom {\xi-1}{r} t^r
\end{equation}
and
\begin{equation}
\mfk f_2^{x,\xi,\xi'}=\sum_{r=1}^{\infty}
\binom {x}{r} \binom {\xi'-\xi-1}{r-1} t^r
\end{equation}
for $x,\xi,\xi',\varrho,\varrho' \in \mbb N$.

By (\ref{c2.11}), we have
\begin{eqnarray}   \label{d2.57}
 {\msr S}^{\varrho-\varrho'+1,\xi-1}_2(t)-{\msr S}^{\varrho-\varrho' ,\xi-1}_2(t)
=\frac{\mfk f_1^{\xi,\varrho,\varrho'}}{(1-t)^{\varrho-\varrho'+\xi-1}}
\end{eqnarray}
and
\begin{eqnarray}   \label{d2.58}
{\msr S}^{m-\varrho,\xi'-\xi+1}_2(t)-{\msr S}^{m-\varrho,\xi'-\xi}_2(t)
=\frac{\mfk f_2^{m-\varrho,\xi,\xi'}
}{(1-t)^{m-\varrho+\xi'-\xi }}.
\end{eqnarray}

According to Lemma \ref{lemd2.3}, (\ref{d2.56}), (\ref{d2.57}) and (\ref{d2.58}), we have
\begin{eqnarray}
& &{\msr S}^{m,n, \varrho,\xi}(t)(1-t)^{ m+2n+\varrho-\xi-2}
\nonumber \\
&=&\sum_{\varrho'=1}^{\varrho} \sum_{\xi'=\xi}^{n}
\frac{
\left | \begin{matrix}
\sum_{r=0}^{\infty}\binom {\varrho-2}{r}\binom {n-\xi'}{r}t^r &  \sum_{r=0}^{\infty}\binom {\varrho-2}{r-1}\binom {n-\xi'+1}{r}t^r  \\
\sum_{r=0}^{\infty}\binom {\varrho'-1}{r+1}\binom {n-\xi-1}{r}t^r & \sum_{r=0}^{\infty}\binom {\varrho'-1}{r }\binom {n-\xi}{r}t^r
\end{matrix} \right |
}{(1-t)^{ 2n+\varrho+\varrho'-\xi-\xi'-1}}
 \cdot
\frac{\mfk f_1^{\xi,\varrho,\varrho'}
}{(1-t)^{\varrho-\varrho'+\xi-1}}
\nonumber \\
& &\quad \cdot
\frac{\mfk f_2^{m-\varrho,\xi,\xi'}
}{(1-t)^{m-\varrho+\xi'-\xi }}
\cdot (1-t)^{ m+2n+\varrho-\xi-2}
\nonumber \\
&=&\sum_{\varrho'=1}^{\varrho} \sum_{\xi'=\xi}^{n}
\left | \begin{matrix}
\mfk f_2^{m-\varrho,\xi,\xi'}\left( \sum_{r=0}^{\infty}\binom {\varrho-2}{r}\binom {n-\xi'}{r}t^r\right)
 & \mfk f_2^{m-\varrho,\xi,\xi'}\left(  \sum_{r=0}^{\infty}\binom {\varrho-2}{r-1}\binom {n-\xi'+1}{r}t^r \right)  \\
\mfk f_1^{\xi,\varrho,\varrho'}\left(\sum_{r=0}^{\infty}\binom {\varrho'-1}{r+1}\binom {n-\xi-1}{r}t^r\right) &\mfk f_1^{\xi,\varrho,\varrho'}\left( \sum_{r=0}^{\infty}\binom {\varrho'-1}{r }\binom {n-\xi}{r}t^r\right)
\end{matrix} \right |
\nonumber \\
&=&
\left | \begin{matrix}
\sum\limits_{\xi'=\xi}\limits^{n} \mfk f_2^{m-\varrho,\xi,\xi'}\left( \sum_{r=0}^{\infty}\binom {\varrho-2}{r}\binom {n-\xi'}{r}t^r\right)
 & \sum\limits_{\xi'=\xi}\limits^{n} \mfk f_2^{m-\varrho,\xi,\xi'}\left(  \sum_{r=0}^{\infty}\binom {\varrho-2}{r-1}\binom {n-\xi'+1}{r}t^r \right)  \\
\sum\limits_{\varrho'=1}\limits^{\varrho} \mfk f_1^{\xi,\varrho,\varrho'}\left(\sum_{r=0}^{\infty}\binom {\varrho'-1}{r+1}\binom {n-\xi-1}{r}t^r\right) &
\sum\limits_{\varrho'=1}\limits^{\varrho} \mfk f_1^{\xi,\varrho,\varrho'}\left( \sum_{r=0}^{\infty}\binom {\varrho'-1}{r }\binom {n-\xi}{r}t^r\right)
\end{matrix} \right |
\end{eqnarray}

By Lemma \ref{lemd2.4}, we have
\begin{eqnarray}
& &\sum\limits_{\xi'=\xi}\limits^{n} \mfk f_2^{m-\varrho,\xi,\xi'}\left( \sum_{r=0}^{\infty}\binom {\varrho-2}{r}\binom {n-\xi'}{r}t^r\right)
\nonumber \\
&=&\sum_{\xi'=\xi}^{n}   \sum_{r=0}^{\infty} \left(\sum_{r'=0}^{r}
\binom {\varrho-2}{r-r'}\binom {n-\xi'}{r-r'}
 \binom {m-\varrho}{r'} \binom {\xi'-\xi-1}{r'-1}
\right)t^r
\nonumber \\
&=&   \sum_{r=0}^{\infty}\left[ \sum_{r'=0}^{r} \left(\sum_{\xi'=\xi}^{n}
\binom {n-\xi'}{r-r'} \binom {\xi'-\xi-1}{r'-1}\right)
\binom {\varrho-2}{r-r'}
 \binom {m-\varrho}{r'} \right]  t^r
\nonumber \\
&=&    \sum_{r=0}^{\infty} \left( \sum_{r'=0}^{r}
\binom {n-\xi }{r }
\binom {\varrho-2}{r-r'}
 \binom {m-\varrho}{r'} \right)  t^r \nonumber \\
&=&    \sum_{r=0}^{\infty}
\binom {n-\xi }{r }
\binom {m -2}{r }   t^r
\end{eqnarray}

Similarly, we can prove other three equations.
Hence \begin{eqnarray}
(1-t)^{m+2n+\varrho-\xi-2}
{\msr S}^{m,n, \varrho,\xi}(t)
=
\left | \begin{matrix}
\sum_{r=0}^{\infty}\binom {m-2}{r}\binom {n-\xi}{r}t^r &  \sum_{r=0}^{\infty}\binom {m-2}{r-1}\binom {n-\xi+1}{r}t^r  \\
\sum_{r=0}^{\infty}\binom {\varrho-1}{r+1}\binom {n-2}{r}t^r & \sum_{r=0}^{\infty}\binom {\varrho-1}{r }\binom {n-1}{r}t^r
\end{matrix} \right |
\end{eqnarray}
and (\ref{d2.53}) holds.

\end{proof}
 \pse

For $ i_1,i_2,i_3,j_1,j_2,j_3 \in \ol{1,n} $, we denote

 \begin{equation}
\Delta^{i_1,i_2}_{ j_1,j_2} =z_{i_1,j_1}z_{i_2,j_2}-z_{i_1,j_2}z_{i_2,j_1}
\end{equation}
and
\begin{eqnarray}
\Delta^{i_1,i_2,i_3}_{ j_1,j_2,j_3}&=& z_{i_1,j_1}z_{i_2,j_2}z_{i_3,j_3}+
z_{i_2,j_1}z_{i_3,j_2}z_{i_1,j_3}+z_{i_3,j_1}z_{i_1,j_2}z_{i_2,j_3} \nonumber \\
& &
-z_{i_2,j_1}z_{i_1,j_2}z_{i_3,j_3}-z_{i_1,j_1}z_{i_3,j_2}z_{i_2,j_3}-z_{i_3,j_1}z_{i_2,j_2}z_{i_1,j_3}.
\end{eqnarray}

Denote $\msr Z'=\mbb F \left[ z_{i,j} \mid 1\leq i\leq n-n_1,  1\leq j \leq n_2  \right]$. Denote by $\msr Z'_m$ the subspace of its homogeneous polynomials with degree $m$.
Set the ideal of $\msr Z'$
\begin{eqnarray}\label{c2.51}
I&=&  \Big\langle z_{i_0,j_0}, \Delta^{k_1',k_2'}_{ k_1,k_2} , \Delta^{k_3',k_4'}_{ k_3,k_4} ,\Delta^{i_1,i_2,i_3}_{ j_1,j_2,j_3} \big\vert i_0\in \ol{1, n_2-n_1},j_0\in \ol{n_1+1,n_2}; 1\leq k_1<k_2\leq n_2-n_1, \nonumber \\& &   \quad 1\leq k_1'<k_2'\leq n_1,
n_2-n_1+1\leq k_3<k_4\leq n-n_1,     n_1+1 \leq k_3'<k_4'\leq n_2 ;\nonumber \\& &  \quad
 1\leq i_1<i_2<i_3\leq n-n_1,1\leq j_1<j_2<j_3 \leq n_2   \Big\rangle.
\end{eqnarray}
then by Lemma \ref{lemc2.1}, the homogeneous ideal of $\msr V({\msr H}_{\la \ell_1,\ell_2 \ra} )$ is $I$.
Denote $I_{(i)}= I \cap \msr Z'_m$, then $I=\bigoplus_{i} I_{(i)}$, and $R =\bigoplus_{i} \msr Z'_i /I_{(i)} $ is the homogeneous coordinate ring of the associated variety $\msr V({\msr H}_{\la \ell_1,\ell_2 \ra})$.

\begin{proposition} \label{propf2.2}
Let $\mfk g= sl(n)$, $M={\msr H}_{\la \ell_1,\ell_2 \ra} $. If $n_1\leq n_2<n$, $\ell_1 \leq 0$ or $\ell_2 \leq 0$, the Hilbert polynomial of $\msr V(M)$ is
\begin{equation}\label{c2.52}
\mfk p_{{M}}(k) = P_2( n_2,n-n_1,n_1,n_2-n_1+1 ,k ).\end{equation}
If $n_1<n_2=n$,
\begin{equation}\label{c2.53}
\mfk p_{{M}}(k) = h_2(n-n_2,n_1).\end{equation}
\end{proposition}

\begin{proof}
(\ref{c2.53}) holds by Lemma \ref{lemc2.1} and (\ref{c2.10}), we only need to prove (\ref{c2.52}).

We take $\leq$ be the lexical order on $\ol{1,m} \times \ol{1,n}$. For
\begin{equation}S_1=\{(j_1,i_1),...,(j_k,i_k)\},\quad S_2=\{(j_1',i_1'),...,(j_k',i_k')\}\end{equation}
with $(j_s,i_s),(j_s',i_s')\in \ol{1,m} \times \ol{1,n}$.
Since the notion of sets $\{...\}$ is independent of the order of its elements, we may assume
\begin{equation}(j_1,i_1)\leq\cdots \leq (j_k,i_k),\qquad (j_1',i_1')\leq\cdots \leq (j_k',i_k').\end{equation}
Take the lexical order $\leq$ on $(\ol{1,m} \times \ol{1,n})^k$ and define $S_1\leq S_2$ if
\begin{equation}\label{f2.92}((j_1,i_1),\dots,(j_k,i_k))\leq ((j_1',i_1'),\dots,(j_k',i_k'))\;\;\mbox{in}\;\;(\ol{1,m} \times \ol{1,n})^k.\end{equation}
Assume the above  $S_1\leq S_2$ and
\begin{equation}T_1=\{(m_1,\ell_1),...,(m_\iota,\ell_\iota)\},\quad T_2=\{(m_1',\ell_1'),...,(m_\iota',\ell_\iota')\}\end{equation}
with $(m_t,\ell_t),(m_t',\ell_t')\in \ol{1,m} \times \ol{1,n}$
also satisfying $T_1\leq T_2$. Then
\begin{equation}\{S_1,T_1\}\leq \{S_2,T_2\}.\end{equation}

Consider elements in the set of generating  elements of $I$ in (\ref{c2.51}).
Each $\Delta^{k_1',k_2'}_{ k_1,k_2}$ corresponds to a 2-chain $(k_1,k_1') \prec (k_2,k_2')$, each $\Delta^{k_3',k_4'}_{ k_3,k_4}$ corresponds to a 2-chain $(k_3,k_3') \prec (k_4,k_4')$, and each $\Delta^{i_1,i_2,i_3}_{ j_1,j_2,j_3}$ corresponds to a 3-chain $(i_1,j_1) \prec (i_2,j_2) \prec (i_3,j_3)$. Thus this set of generating elements corresponds to
\begin{eqnarray}\label{c2.59}
S_I&=&\Big\{ \{(i_0,j_0)\}, \{(k_1,k'_1), (k_2,k'_2)\} ,\{(k_3,k'_3), (k_4,k'_4)\} ,\{(i_1,j_1), (i_2,j_2) , (i_3,j_3)\}\vert   \nonumber \\& & \quad i_0\in \ol{1, n_2-n_1},j_0\in \ol{n_1+1,n_2};  1\leq k_1<k_2\leq n_2-n_1,      1\leq k_1'<k_2'\leq n_1,\nonumber \\& &  \quad
n_2-n_1+1\leq k_3<k_4\leq n-n_1,      n_1+1 \leq k_3'<k_4'\leq n_2 ;    \nonumber \\& &\quad
 1\leq i_1<i_2<i_3\leq n-n_1,1\leq j_1<j_2<j_3 \leq n_2    \Big\}
\nonumber \\ &=&
\{ w_1,\cdots,w_{N} \}
\end{eqnarray}
where $N=|S_I|$.

We set the linear map $ (\cdot)^{T } $ from $\msr Z $ to  $\mbb F\left[z_{j,i}\vert i\in \ol{1,m},j\in \ol{1,n}\right] $ by
\begin{equation} \left(\prod_{1\leq i \leq m,1\leq j\leq n}z_{i,j}^{\gamma_{i,j}}\right) ^{T } =\prod_{1\leq i \leq m,1\leq j\leq n}z_{j,i}^{\gamma_{i,j}}.\end{equation}
Then $(f^T)^T=f$ for $f\in \msr Z$.

For $f \in \msr Z'_i$, we write $f = \sum_{j=1}^{m_0}  f_j$, where $f_j$ are distinct monomials in $\msr Z'_i$, with $\msr I(f_1) \leq \msr I(f_2) \leq \cdots \leq \msr I(f_{m_0})$ according to the lexical order defined in (\ref{f2.92}). 
Define $\msr I(f) = \{\msr I(f_1) ,\msr I(f_2) ,\cdots,\msr I(f_{m_0}) \}$.
For another $g\in \msr Z'_i$ with $g=\sum_{j=1}^{m'_0} g_j$ where $g_j$ are distinct monomials in $\msr Z'_i$.
We define $\msr I(f) \leq \msr I(g)$ if $\msr I(f_j) \leq \msr I(g_j) $ for $j\in \ol{1,\min \{m_0,m_0'\}}$. 

We want to show that there exists a $\widetilde{f} \in  \msr Z'_i  $ such that $\widetilde{f} +I_{(i)}=f+I_{(i)}$ and 
\begin{equation}\label{f2.97}
\widetilde{f}^T\in \text{Span} \left\{S_2^{\ol{1,n_2},\ol{1,n-n_1}}(n_1,n_2-n_1+1,i)\right\} .
\end{equation}  
by induction on $\msr I(f)$, since $|\{ \msr I(f) \mid f\in \msr Z'_i\}|<\infty$.

When \begin{equation}\label{e2.98} w_{\eta} \not\subset \msr I(f_j)\qquad
\for\,\, \eta\in \ol{1,N} \text{ and } j\in \ol{1,m_0},
\end{equation} we have $f_j^{T } \in S_2^{\ol{1,n_2},\ol{1,n-n_1}}(n_1,n_2-n_1+1,i)$ for $j\in \ol{1,m_0}$.
Then we take $\widetilde{f}=f$ in this case, and (\ref{f2.97}) holds. In particular, when $\msr I(f)\geq \msr I(g)$ for each $ g\in \msr Z'_i $, we have $f\in \mbb F z_{n-n_1,n_2}^{i}$. Then we take $\widetilde{f}=f$ and (\ref{f2.97}) holds.
 
For $g \in \msr Z'_i$ with $ \msr I(g) >\nu $, there exists a $\widetilde{g}\in  \msr Z'_i  $ such that $\widetilde{g} +I_{(i)}=g+I_{(i)}$ and (\ref{f2.97}) holds. Consider $f  \in \msr Z'_i  $ with  $ \msr I(f) =\nu $.

Suppose that there exists some $w_{\eta}  \in S_I$ such that $w_{\eta} \subset \bigcup_{j=1}^{m_0} \msr I(f_j)$.  
Set
\begin{equation}\label{e2.99}
 \eta_0(f)=\min\left\{\eta \in \ol{1,N} \Big\vert w_{\eta} \subset \bigcup_{j=1}^{m_0} \msr I(f_j) \right\}.
\end{equation} 
and \begin{equation}
J_0= \{ j\in \ol{1,m_0} \vert w_{\eta_0(f)} \subset \msr I(f_j)\} \neq \emptyset.
\end{equation} 
Now we want to show that there exists $\ol{f} \in \msr Z'_i $ such that $\ol{f} +I_{(i)}=f+I_{(i)}$, $\msr I(\ol{f})> \msr I(f)$ or (\ref{e2.98}) holds for $\ol{f}$.

(1) When $w_{\eta_0(f)} =\{(k_1,k'_1), (k_2,k'_2)\}$ for some $k_1,k_2 \in \ol{1,n_2-n_1}$ and $k_1',k_2'\in \ol{1,n_1}$ with $k_1<k_2$ and $k_1'<k_2'$.\pse

Write $f_j =f_j' z_{k_1,k_1'}z_{k_2,k_2'} $ for $j\in J_0$, and
we set
\begin{equation}
\ol{f}= \sum_{j\not\in J_0}f_j+ \sum_{j \in J_0}( f_j' z_{k_1,k_2'}z_{k_2,k_1'}) .
\end{equation} 
Then we have
\begin{equation}
\ol{f}+I_{(i)}= f+I_{(i)}
\end{equation} 
and
\begin{eqnarray}
\msr I(f_j' z_{k_1,k_2'}z_{k_2,k_1'})&=&\{ (k_1,k_2') , (k_2,k_1'),  \msr I(f_j')\}\nonumber\\
&>&\{ (k_1,k_1') , (k_2,k_2') \msr I(f_j')\}=\msr I(f_j),
\end{eqnarray}
for $j\in J_0$. 
Hence $\msr I(\ol{f})> \msr I(f)$.

(2) When $w_{\eta_0(f)} =\{(k_3,k'_3), (k_4,k'_4)\}$ with $n_2-n_1+1\leq k_3<k_4\leq n-n_1$ and $ n_1+1 \leq k_3'<k_4'\leq n_2 $. \pse

Similar to the case (2), we set
\begin{equation}
\ol{f}= \sum_{j\not\in J_0}f_j+ \sum_{j \in J_0}( f_j' z_{k_3,k_4'}z_{k_3,k_4'}) .
\end{equation} 
Then we also have $\ol{f}+I_{(i)}= f+I_{(i)}$ and $\msr I(\ol{f})> \msr I(f)$.

(3) When $w_{\eta_0(f)} =\{(i_1,j_1),(i_2,j_2),(i_3,j_3)\}$ with  $1\leq i_1<i_2<i_3\leq n-n_1,1\leq j_1<j_2<j_3 \leq n_2  $. \pse

Write $f_j=z_{i_1,j_1}z_{i_2,j_2}z_{i_3,j_3}f_j'$ for $j\in J_0$, we set 
\begin{eqnarray}
\ol{f}&=& \sum_{j\not\in J_0}f_j + \sum_{j \in J_0}( (z_{i_2,j_1}z_{i_3,j_2}z_{i_1,j_3}+z_{i_3,j_1}z_{i_1,j_2}z_{i_2,j_3} 
-z_{i_2,j_1}z_{i_1,j_2}z_{i_3,j_3}\nonumber \\  
& &\quad -z_{i_1,j_1}z_{i_3,j_2}z_{i_2,j_3}-z_{i_3,j_1}z_{i_2,j_2}z_{i_1,j_3} )f_j') .
\end{eqnarray}  
Then we have $\ol{f}+I_{(i)} =f+I_{(i)}$ and
\begin{eqnarray}\msr I(f_j)&=&\{(i_1,j_1),(i_2,j_2),(i_3,j_3), \msr I(f_j')\}\nonumber\\&<&\{(i_1,j_2),(i_2,j_1),(i_3,j_3), \msr I(f_j')\}=\msr I(z_{i_1,j_2}z_{i_2,j_1}z_{i_3,j_3}f_j'),\end{eqnarray}
\begin{eqnarray}\msr I(f_j)&=&\{(i_1,j_1),(i_2,j_2),(i_3,j_3), \msr I(f_j')\}\nonumber\\&<&\{(i_1,j_3),(i_2,j_2),(i_3,j_1), \msr I(f_j')\}=\msr I(z_{i_1,j_3}z_{i_2,j_2}z_{i_3,j_1}f_j'),\end{eqnarray}
\begin{eqnarray}\msr I(f_j)&=&\{(i_1,j_1),(i_2,j_2),(i_3,j_3), \msr I(f_j')\}\nonumber\\&<&\{(i_1,j_1),(i_2,j_3),(i_3,j_2), \msr I(f_j')\}=\msr I(z_{i_1,j_1}z_{i_2,j_3}z_{i_3,j_2}f_j'),\end{eqnarray}
\begin{eqnarray}\msr I(f_j)&=&\{(i_1,j_1),(i_2,j_2),(i_3,j_3), \msr I(f_j')\}\nonumber\\&<&\{(i_1,j_2),(i_2,j_3),(i_3,j_1), \msr I(f_j')\}=msr I(z_{i_1,j_2}z_{i_2,j_3}z_{i_3,j_1}f_j'),\end{eqnarray}
\begin{eqnarray}\msr I(f_j)&=&\{(i_1,j_1),(i_2,j_2),(i_3,j_3), \msr I(f_j')\}\nonumber\\&<&\{(i_1,j_3),(i_2,j_1),(i_3,j_2), \msr I(f_j')\}=\msr I(z_{i_1,j_3}z_{i_2,j_1}z_{i_3,j_2}f_j').\end{eqnarray}
Hence $\msr I(\ol{f})> \msr I(f)$.

(4) When $w_{\eta_0(f)} =\{(i_0,j_0)\}$ for some $i_0\in \ol{1,n_2-n_1}$ and $j_0\in \ol{n_1+1,n_2}$, we have $f_j \in I_{(i)}$ for $j\in J_0$. \pse

We set
\begin{equation}
F(f)=f-\sum_{j\in J_0}f_j, 
\end{equation}  thus $F(f) +I_{(i)}=f+I_{(i)}$, $\msr I(F(f))\geq \msr I(f)$ and $  (i',j')\not\in \bigcup_{j=1}^{m_0} \msr I(f_j) $ for $(i',j')\in \ol{1,n_2-n_1}  \times \ol{n_1+1,n_2}$. 
If $\eta_0(f)=N$, we take $\ol{f}=F(f)$ and (\ref{e2.98}) holds for $\ol{f}$.  
When $\eta_0(f)<N$, we have $\eta_0(F(f))>\eta_0(f)$.
By induction on $\eta_0(f)$, there exists $\ol{F(f)} \in \msr Z'_i $ such that $\ol{F(f)} +I_{(i)}=F(f)+I_{(i)}$, $\msr I(\ol{F(f)}) > \msr I(F(f))$ or (\ref{e2.98}) holds for $\ol{F(f)}$. 
We take $\ol{f}=\ol{F(f)} $, then $ \ol{f} +I_{(i)}=f+I_{(i)}$, $\msr I(\ol{f})>\msr I(f)$ or (\ref{e2.98}) holds for $\ol{f}$. 

By inductional assumption, there exists a $\widetilde{\ol{f}} \in \msr Z'_i$ such that $\widetilde{\ol{f}} +I_{(i)}=\ol{f}+I_{(i)}$ and 
\begin{equation} 
\widetilde{\ol{f} }^T\in \text{Span} \left\{S_2^{\ol{1,n_2},\ol{1,n-n_1}}(n_1,n_2-n_1+1,i)\right\} .
\end{equation}  
We take $\widetilde{f}=\widetilde{\ol{f}} $, then $\widetilde{f} +I_{(i)}=f+I_{(i)}$ and (\ref{f2.97}) holds. 
Moreover, according to our inductive process, we can specifically determine the $\widetilde{f}$ related to $f$.

Note that $\widetilde{F_1+F_2} = \widetilde{F_1}+ \widetilde{F_2}$ for $F_1,F_2 \in \msr Z'_i$.
Hence we can define a linear map $\varPhi$ from $\msr Z'_i  $ to $\text{Span} \left\{S_2^{\ol{1,n_2},\ol{1,n-n_1}}(n_1,n_2-n_1+1,i)\right\} $ by $\varPhi(f)=\widetilde{f} ^T$.

Suppose $f\in \ker \varPhi$, we have $f\in I_{(i)}$ since $\widetilde{f} +I_{(i)}=f+I_{(i)}$. Hence $\ker \varPhi  \subset I_{(i)}$.

Denote by $\mfk{W}_j$ the generating element in (\ref{c2.51}) corresponding to $w_j$ in (\ref{c2.59})  and by $\mfk{w}_j$ the corresponding monomial in $\msr Z'$ for $j\in \ol{1,N}$. For example, $\mfk{W}_j=\Delta^{i_1,i_2,i_3}_{ j_1,j_2,j_3}$ and $\mfk{w}_j=z_{i_1,j_1}z_{i_2,j_2}z_{i_3,j_3}$ if $w_j=\{(i_1,j_1), (i_2,j_2) , (i_3,j_3)\}$.

Now we suppose $f=f' \mfk{W}_j$ where $f'$ is a monomial in $\msr Z'$.
When $j=1$, we have $\ol{f}=f' \ol{\mfk{W}_1}=0$, thus $\varPhi(f)=0$.
We assume that $\varPhi(f)=\varPhi( f' \mfk{W}_j)=0$ for $j \in \ol{1,j_0-1}$. Consider $j=j_0$, we write $f=\sum_{l=1}^{m_1} f_l$ where $f_l$ are distinct monomials in $\msr Z'$.
By the inductional assumption, $\varPhi(f)=\varPhi(g)$ if $f-g\in \la \mfk{W}_j\vert j \in \ol{1,j_0-1}\ra$.
When $w_s \not\subset \msr I(f_l)  $ for $s\in \ol{1,j_0-1}$ and $l\in \ol{1,m_1}$, we have $\ol{f}=f' \ol{\mfk{W}_{j_0}}=0$, thus $\varPhi(f)=0$.
Otherwise, we suppose that $w_s \subset \msr I(f_l)  $ for some $s\in \ol{1,j_0-1}$ and $l\in \ol{1,m_1}$. We write $f_l = f_l'\mfk{w}_s$, then $\varPhi(f)=\varPhi(f-f_l'\mfk{W}_{s})$ and $\msr I(f-f_l'\mfk{W}_{s}) >\msr I(f)$.
Continue this process for $f-f_l'\mfk{W}_{s}$, we can prove that there exists $g\in \la \mfk{W}_j\vert j \in \ol{1,j_0-1}\ra$ such that $w_s \not\subset \msr I(h_l)  $ for each $l$, where $h_l$ are monomials in $f-g=\sum_{l} h_l$.
Hence $\varPhi(f)=0$.
Note that $\varPhi$ is a linear map, we have $\varPhi(I_{(i)})=\{0\}$.
Therefore, $\ker \varPhi=I_{(i)}$.

Then $\varPhi$ induces a linear isomorphism $\Phi$ from $\msr Z'_i /I_{(i)} $ to $\text{Span} \left\{S_2^{\ol{1,n_2},\ol{1,n-n_1}}(n_1,n_2-n_1+1,i)\right\} $ by 
 \begin{equation}
\Phi(f+I_{(i)})=\varPhi(f).
\end{equation} 
Hence  
\begin{equation}
\dim (\msr Z'_i /I_{(i)}) = \left|  \left\{S_2^{\ol{1,n_2},\ol{1,n-n_1}}(n_1,n_2-n_1+1,i)\right\} \right|
=P_2(n_2,n-n_1,n_1,n_2-n_1+1,i).
\end{equation} 
Therefore,
\begin{equation}
\mfk p_{M}(k) = P_2( n_2,n-n_1,n_1,n_2-n_1+1 ,k ).\end{equation} 
and the lemma holds.
\end{proof}

From (\ref{c2.32}), Proposition \ref{propd2.1} and  Proposition \ref{propf2.2},  we have our main theorem in this section:

\begin{theorem}\label{thef2.1}
Let $\mfk g= sl(n)$, $M={\msr H}_{\la \ell_1,\ell_2 \ra} $.
 If $n_1\leq n_2<n$, $\ell_1 \leq 0$ or $\ell_2 \leq 0$, the Hilbert series of $\msr V(M)$ is
\begin{eqnarray}\label{c2.79}
\mfk s_{_{M}}(t)
=\frac{\sum\limits_{r,s=0}\limits^{\infty}
\left | \begin{matrix}
 \binom {n_2-2}{r}\binom {n - n_2 -1 }{r}  &
 \binom {n_2-2}{r-1}\binom {n - n_2  }{r}   \\
 \binom {n_1-1}{s+1}\binom {n-n_1-2}{s} &
 \binom {n_1-1}{s }\binom {n-n_1-1}{s}
\end{matrix} \right | t^{r+s}
}{(1-t)^{2n-3}}
\end{eqnarray} 
If $n_1<n_2=n$,
\begin{equation}\label{c2.80}
\mfk s_{_{M}}(t) =  \frac{  \sum_{r=0}^{\infty}\binom { n-n_2-1 }{r} \binom {n_1-1 }{r}t^r    }
{ \left(1-t\right)^{n-n_2+n_1-1} } .\end{equation}
\end{theorem}

This completes the proof of Theorem 1 that we stated in the introduction.

Recall that the constant term of the Hilbert polynomial $\mfk p_{M}$ is the arithmetic genus of $\msr V(M)$, which we denote as $p_a(\msr V(M))$, then we can find $p_a(\msr V(M))$ by the above corollary.
\begin{corollary}\label{cord2.2}
Let $\mfk g= sl(n)$ and $M={\msr H}_{\la \ell_1,\ell_2 \ra} $ under the conditions stated in Lemma \ref{lemc2.1},  the arithmetic genus of $\msr V(M)$ is 1.
\end{corollary}

\begin{proof}
When $n_1\leq n_2<n$, $\ell_1 \leq 0$ or $\ell_2 \leq 0$, we have
\begin{eqnarray}
p_a(\msr V(M))=\mfk s_{_{M} } (0)
=
\left | \begin{matrix}
 \binom {n_2-2}{0}\binom {n - n_2 -1 }{0}  &  \binom {n_2-2}{0-1}\binom {n- n_2  }{0}   \\
 \binom {n_1-1}{ 1}\binom {n-n_1-2}{0} & \binom {n_1-1}{r }\binom {n-n_1-1}{0}
\end{matrix} \right | =1,
\end{eqnarray}
by Proposition \ref{propd2.1} and Corollary \ref{thef2.1}.
When $n_1<n_2=n$, we also have $p_a(\msr V(M))=1$ by (\ref{c2.10}).
Therefore, the arithmetic genus of $\msr V(M)$ is 1.
\end{proof}

Note that the Hilbert series may be rewritten as $\mfk s_{ M }(t)=\frac{Q(t)}{(1-t)^d}$, where $Q(t)$ is a polynomial with integer coefficients, and $d$ is the Krull dimension of the coordinate ring of the variety $\msr V(M)$.
Then the degree of the variety is $\deg (\msr V(M))=  Q(1)=  (d-1)!  \cdot lc (\mfk p_{{M}} )  $.
Hence we have the following corollary from Proposition \ref{propf2.2}.

\begin{corollary}
Let $\mfk g= sl(n)$ and $M={\msr H}_{\la \ell_1,\ell_2 \ra} $.
\quad
\begin{equation} \label{d2.99}
\deg (\msr V({\msr H}_{\la \ell_1,\ell_2 \ra}))=
\begin{cases}
\left( \frac{n-2}{n_1-1}-\frac{n-2}{n_2-1}\right) \binom {n-3}{n_1-2}\binom {n-3}{n_2-2}  ,\ \quad  \text{if }n_2\neq n \text{ and } n_1\neq n_2; \pse \\
 \frac{1}{n_1-1} \binom {n-4}{ n_1-2}\binom {n-3}{ n_1-2},\ \quad  \text{if }1 < n_1=n_2<n-1; \pse\\
\binom { n +n_1-n_2-2}{n_1-1},\ \quad \text{if }   n_1\neq n_2=n ; \pse\\
1,\ \quad \text{if }  1=n_1=n_2<n \text{ or } n_1=n_2=n-1 .
\end{cases}
\end{equation}
\end{corollary}

\begin{proof}
According to Proposition \ref{propd2.1} and Corollary \ref{thef2.1} ,
 \begin{eqnarray}
\mfk s_{_{M} }(t)
=
\frac{
\left | \begin{matrix}
\sum_{r=0}^{\infty}\binom {n_2-2}{r}\binom {n - n_2 -1 }{r}t^r &  \sum_{r=0}^{\infty}\binom {n_2-2}{r-1}\binom {n- n_2  }{r}t^r  \\
\sum_{r=0}^{\infty}\binom {n_1-1}{r+1}\binom {n-n_1-2}{r}t^r & \sum_{r=0}^{\infty}\binom {n_1-1}{r }\binom {n-n_1-1}{r}t^r
\end{matrix} \right |
}{(1-t)^{2n-3}},
\end{eqnarray}
when $n_1\leq n_2<n$, $\ell_1 \leq 0$ or $\ell_2 \leq 0$.
Note that
\begin{equation}\label{d2.100}
\dim ( I(\msr{V}(M)))=
\begin{cases}
2n-3,\ \quad \text{if }n_2\neq n \text{ and } n_1\neq n_2; \\
2n-4,\ \quad \text{if }1 < n_1=n_2<n-1; \\
n-1,\ \quad \text{if } 1=n_1=n_2<n \text{ or } n_1=n_2=n-1\text{ or } n_1\neq n_2=n .
\end{cases}
\end{equation}
Hence, when $n_2 \neq n$ and $n_1\neq n_2$, we have
\begin{eqnarray}
 \deg (\msr{V}(M))
&=&
\left | \begin{matrix}
\sum_{r=0}^{\infty}\binom {n_2-2}{r}\binom {n - n_2 -1 }{r}  &  \sum_{r=0}^{\infty}\binom {n_2-2}{r-1}\binom {n- n_2  }{r}  \\
\sum_{r=0}^{\infty}\binom {n_1-1}{r+1}\binom {n-n_1-2}{r}  & \sum_{r=0}^{\infty}\binom {n_1-1}{r }\binom {n-n_1-1}{r}
\end{matrix} \right |
\nonumber  \\
&=&
\left | \begin{matrix}
 \binom {n-3}{n_2-2} &  \binom {n-2}{n_2-1}   \pse\\
 \binom {n-3}{n_1-2} &  \binom {n-2}{n_1-1}
\end{matrix} \right | \nonumber \\
&=& \binom {n-3}{n_1-2}\binom {n-3}{n_2-2} \left( \frac{n-2}{n_1-1}-\frac{n-2}{n_2-1}\right)
\end{eqnarray}
by Lemma \ref{lemd2.4}.

When $n_1<n_2=n$, according to (\ref{d2.100}), (\ref{c2.53}) and (\ref{c2.11}), we have
\begin{equation}
 \deg \msr{V}(M) = \binom { n +n_1-n_2-2}{n_1-1}.
\end{equation}

When $ 1<n_1=n_2<n-1$, we have $\msr{V}(M) \cong \msr V_3(n-n_1,n_1)$, and then
\begin{equation}
 \deg( \msr{V}(M)) = \frac{1}{n_1-1} \binom {n-4}{ n_1-2}\binom {n-3}{ n_1-2}  .
\end{equation}
by (\ref{b3.4}) and Lemma \ref{lemd2.4}.

When $ 1=n_1=n_2<n$ or $n_1=n_2=n-1$, we have $\msr{V}(M) \cong \mbb P^{n-2}$,  and then
\begin{equation}
 \deg( \msr{V}(M)) = 1.
\end{equation}
Therefore, (\ref{d2.99}) holds.

\end{proof}

By the same method, we can derive the following more general conclusion.
 \begin{proposition}
\quad  For $m,n,\varrho,\xi \in \mbb N$ with $1\leq \varrho \leq m $ and $1\leq \xi \leq n$, we have
\begin{eqnarray}
& &lc_r(P_1(m,n, \xi,r))\cdot(\deg_r(P_1(m,n, \xi,r))!)   \nonumber \\ &=&
\begin{cases}
1,\ \quad  \text{if }m=1 \text{ or } n<3; \pse \\
\frac{1}{n-1}  \binom {m+n -4}{m-2}\binom {m+n-3}{n -2}
,\ \quad  \text{if }\xi=1, m\geq 2\text{ and } n\geq 3; \pse \\
 \frac{\xi-1}{n-1 }\binom {m+n-\xi-2}{m-2}\binom {m+n-3}{n -2}   ,\ \quad  \text{else. }
\end{cases}
\end{eqnarray}
and
\begin{eqnarray}
& & lc_r(P_2(m,n,\varrho,\xi,r)) \cdot(\deg_r(P_2(m,n,\varrho,\xi,r))!)
\nonumber \\
&=&
\begin{cases}
1,\ \quad \text{if }m =1\text{ or }n=1; \pse \\
\frac{1}{n-1}  \binom {m+n -4}{m-2}\binom {m+n-3}{n -2} ,\ \quad  \text{if }m=\varrho,\xi=1 \text{ and } m,n>1; \pse \\
 \frac{(m-\varrho)(n-1)+(\varrho-1)(\xi-1)}{(m-1)(n-1)}\binom {m+n-\xi-2}{m-2}\binom {\varrho+n-3}{n -2}   ,\ \quad \text{else. }
\end{cases}
\end{eqnarray}
\end{proposition}

\section{Relations between $\mfk p_{M}$ and $\mfk p_{M,M_0}$}

Recall that the Hilbert polynomial $\mfk p_{M,M_0}(k)$ depends on the choice of the subspace $M_0$.
In this section, we study the relationship between $\mfk p_{M,M_0}(k) $ and $\mfk p_{M }(k)$.
We show that the  Hilbert polynomial $\mfk p_{M, M_0}(k)$ of such a completely reducible module $M$ satisfies $\mfk p_{_M}(k)\leq \mfk p_{_M, _{M_0}}(k)$ for sufficiently large positive integer $k$.
In particular, when $\mfk g =sl(n)$ and $M=\msr H_{\la \ell_1,\ell_2\ra}$, we find a necessary and sufficient condition for $\mfk p_{M}(k) = \mfk p_{M,\mbb F v}(k)$. Recall $M_k =(U_k( \mathfrak{g}))(M_0)$ and $M_k' =(U_k( \mathfrak{g}))(M_0')$.\psp

\begin{lemma}\label{lemd3.1}
Suppose that $M_0$ and $M_0'$ are finite-dimensional subspaces of $M$ satisfying $(U(\mfk g))(M_0) = (U(\mfk g))(M_0') = M$.
If $\dim (M_k) \leq \dim (M_k')$ for sufficiently large $k$, then $\mfk p_{M,M_0}(k) \leq \mfk p_{M,M_0'}(k)$ for sufficiently large $k$.
In particular, if $M_0 \subset M_0'$, we have $\mfk p_{M,M_0}(k) \leq \mfk p_{M,M_0'}(k)$ for sufficiently large $k$.
\end{lemma}

\begin{proof}
Assume that $\mfk p_{M,M_0'} \neq \mfk p_{M,M_0}$.
By the definition of $\mfk p_{M,M_0}$, there exist $N_1, N_2 \in \mbb N$ such that $\mfk p_{M,M_0}(i) = \dim (M_i) - \dim (M_{i-1})$ for $i \geq N_1$ and $\mfk p_{M,M_0'}(i) = \dim (M_i') - \dim (M_{i-1}')$ for $i \geq N_2$.
Define
\begin{equation}
d_1 = \sum_{i=1}^{N_1} \mfk p_{M,M_0}(i) - \dim (M_{N_1}), \quad
d_2 = \sum_{i=1}^{N_2} \mfk p_{M,M_0'}(i) - \dim (M_{N_2}) \in \mbb N.
\end{equation}
Then, for sufficiently large $k$,
\begin{equation}
\sum_{i=1}^k \mfk p_{M,M_0}(i) = d_1 + \dim (M_k), \quad
\sum_{i=1}^k \mfk p_{M,M_0'}(i) = d_2 + \dim (M_k').
\end{equation}

Since
\begin{equation}\label{f3.3}
\sum_{i=1}^k i^p = \frac{1}{p+1} k^{p+1} + \text{lower degree terms in } k,
\end{equation}
for $p \in \mbb N$,
\begin{equation}\sum_{i=1}^k (\mfk p_{M,M_0'}(i) - \mfk p_{M,M_0}(i))+d_1-d_2=\dim (M_k')- \dim (M_k)\end{equation}
is a polynomial in $k$ with positive degree and takes nonnegative integers when $k$ sufficiently large.
Thus the leading coefficient of $\sum_{i=1}^k (\mfk p_{M,M_0'}(i) - \mfk p_{M,M_0}(i))$ is positive.
By (\ref{f3.3}), the leading coefficient of $\mfk p_{M,M_0'}(i) - \mfk p_{M,M_0}(i)$ is also positive. Hence, $\mfk p_{M,M_0}(k) \leq \mfk p_{M,M_0'}(k)$ for sufficiently large $k$.
\end{proof}

For $i\in \mbb N$, we denote by $S^i(\mfk g)$ the subspace of homogeneous polynomials with the degree $i$.
Denote
\begin{equation}
\left(\operatorname{Ann}_{S( \mfk g)} (\ol{M}) \right)^{ i }= \operatorname{Ann}_{S( \mfk g)} (\ol{M})\cap S^i(\mfk g),
\end{equation}
and
\begin{equation}\label{a6.6}
I_{(i)}= \sqrt{ \operatorname{Ann}_{S( \mfk g)} (\ol{M})}\cap S^i(\mfk g),
\end{equation}
Denote by $I(\msr V(M))=\bigoplus_{i} I_{(i)}$ the homogeneous ideal of $\msr V(M)$ and $R(M)=\bigoplus_{i} S^i(\mfk g)/I_{(i)} $ is the homogeneous coordinate ring of the associated variety $\msr V(M)$.

We say that $M_0$ is radical for $M$, if
\begin{equation}
I_{(i)} =\left(\operatorname{Ann}_{S( \mfk g)} (\ol{M}) \right)^{i}
\end{equation}for $i$ large enough.
Note that for $p\in \mbb N$, $M_p =(U_p( \mathfrak{g}))(M_0)$ is radical for $M$, if $M_0$ is radical for $M$.

\begin{lemma} \label{lem:b5.2}
$M$ is a finitely generated $U(\mfk g)$-module. We take $M_0= \mbb F v_0$, where $v_0 \in M$ and $(U(\mfk g))(\mbb F v_0)=M$. Then we have
\begin{equation}\label{a6.7}
\mfk p_{M}(k) \leq \mfk p_{M,M_0}(k) .
\end{equation}
Moreover, $\mbb F v_0$ is radical for $M$ if and only if
\begin{equation}\label{a6.8}
\mfk p_{M}(k)=\mfk p_{M,M_0}(k) .
\end{equation}
\end{lemma}
\begin{proof}

For $f\in U_k(\mfk g)$, we denote by $\ol{f} \in S^p(\mfk g)$ the image of $f$ under the natural map from $U_k(\mfk g)$ to $U_k(\mfk g)/U_{k-1}(\mfk g)$.

For $p$ large enough, we define the linear map $ \theta_p : M_p  \rta  S^{p}(\mfk g)$ by $ \theta_p (f(v_0) )  = \ol{f} $ with $f \in U_p(\mfk g)$.

Suppose that $f \in U_p(\mfk g)$ satisfying $f (v_0) \in M_{p-1}$.
For any $v\in M_k$, we may assume that $v=g(v_0)$ with $g\in  U_k(\mfk g)$.
Then
\begin{equation}
f ( v)=f(g(v_0)) \in g(f(v_0)) +M_{k+p-1} \subset M_{k+p-1}.
\end{equation}
Hence $ \theta_p(M_{p-1}) \subset I_{(p)} $, and we can define the linear map $ \ol{\theta_p} : M_p/M_{p-1}  \rta S^p(\mfk g)/I_{(p)}$ by $ \theta_p (f(v_0) +M_{p-1})  = \ol{f} +I_{(p)}$ with $f \in U_p(\mfk g)$ for $p$ large enough.
Since $ \theta_p$ is surjective, $\ol{\theta_p}$ is surjective. This yields
 \begin{equation}
\dim (M_p/M_{p-1}) = \dim (  S^p(\mfk g)/I_{(p)})
\end{equation}
for $p$ large enough.
Therefore, (\ref{a6.7}) holds.

The injectivity of $\theta_p$ for sufficiently large $p$ is equivalent to (\ref{a6.8}). That is,
\begin{equation}
\ker \ol{\theta_p} =0+ M_{p-1},
\end{equation}
or equivalently, for any $f \in  U_p(\mfk g)$ such that $\ol{f} \in I_{(p)}$, we have $f(v_0) \in M_{p-1}$.

When $I_{(p)} =\left(\operatorname{Ann}_{S( \mfk g)} (\ol{M}) \right)^{p}$ for $p$ large enough, we have $\ol{f} \in \left(\operatorname{Ann}_{S( \mfk g)} (\ol{M}) \right)^{p}$ and $f(v_0) \in M_{p-1}$.
On the other hand, for $f\in I_{(p)}$, $f(v_0) \in M_{p-1}$ yields that $f(v) \in M_{p+k-1}$ for any $v\in M_k$. Then $f\in \left(\operatorname{Ann}_{S( \mfk g)} (\ol{M}) \right)^{p}$.
\end{proof}

\begin{corollary} \label{cord3.1}
For a finitely generated completely reducible infinite-dimensional $U(\mfk g)$-module $M$ with $M_0 \in M$ such that $(U(\mfk g))( M_0 )=M$,  we have $\mfk p_{M}(k) \leq \mfk p_{M,M_0}(k)  $.
In particular, when $M$ is irreducible, $\mfk p_{M}(k) \leq \mfk p_{M,M_0}(k)  $ holds.
\end{corollary}

\begin{proof}

We assume that
\begin{equation}
M=\bigoplus_{t=1}^s M^{(t)}
\end{equation}
where $M^{(t)}$ are irreducible modules for $t\in \ol{1,s}$.
Since $(U(\mfk g))( M_0 )=M$, we can also decompose $M_0= \bigoplus_{t=1}^s M^{(t)}_0$ such that $(U(\mfk g))( M^{(t)}_0 )=M^{(t)}$ for $t\in \ol{1,s}$.
Then
\begin{equation}
\ol{M}=\bigoplus_{t=1}^s \ol{M^{(t)}}.
\end{equation}
Since the Hilbert polynomial is additive with respect to exact sequences,
\begin{equation}\label{d3.13}
\mfk p_{M,M_0} (k) = \sum_{t=1}^s \mfk p_{M^{(t)},M^{(t)}_0} (k).
\end{equation}
For $t\in \ol{1,s}$, we take $v_t\in M^{(t)}_0$ with $v_t\neq 0$, and then $(U(\mfk g))( \mbb F v_t )= M^{(t)}$ by the irreducibility of $M^{(t)}$.
According to Lemma \ref{lem:b5.2} and Lemma \ref{lemd3.1}, we have
\begin{equation}\label{d3.14}
\mfk p_{M^{(t)},M^{(t)}_0} (k) \geq \mfk p_{M^{(t)},\mbb F v_t} (k) \geq \mfk p_{M^{(t)} } (k).
\end{equation}
According to Lemma 4.2 in \cite{Ja2}, we have
\begin{equation}
\msr V(M) =\bigcup_{t=1}^s \msr V(M^{(t)}).
\end{equation}
Thus
\begin{equation}\label{d3.16}
\mfk p_{M  } (k) \leq \sum_{t=1}^s     \mfk p_{M^{(t)}  } (k).
\end{equation}
Hence,
\begin{equation}
\mfk p_{M,M_0} (k) \geq  \sum_{t=1}^s  \mfk p_{M^{(t)} } (k) \geq \mfk p_{M  } (k)
\end{equation}
by (\ref{d3.13}), (\ref{d3.14}) and (\ref{d3.16}).

\end{proof}

We set
\begin{equation}\label{b5.13}
\eta=\sum_{i\in J_1}y_i\ptl_{x_i}+\sum_{r\in J_2}x_r y_r+\sum_{s\in J_3}x_s \ptl_{y_s}.
\end{equation}
Then we have
\begin{equation} \label{a6.13}
\eta \td{\pi}(E_{i,j})= \td{\pi}(E_{i,j}) \eta\qquad\for\;\;  i,j\in \ol{1,n}
\end{equation}
and
\begin{equation}
\eta (\msr H_{\la \ell_1,\ell_2 \ra}) \subset \msr A_{ \la \ell_1+1,\ell_2+1 \ra} .
\end{equation}

For $n_1,n_2 ,n \in \mathbb{N}$ with $n_1\leq n_2 \leq n$, we denote
\begin{equation}
{\msr H}_{\la\ell_1,\ell_2\ra}^{n_1,n_2,n}={\msr H}_{\la\ell_1,\ell_2\ra} \qquad \td\Dlt^{n_1,n_2,n}=\td\Dlt, \qquad \mfk d_{\la\ell_1,\ell_2\ra}^{n_1,n_2,n}=\mfk d,
\end{equation}
and
\begin{equation}\label{a2.24}
 ({\msr H}_{\la\ell_1,\ell_2\ra}^{n_1,n_2,n})_k= U(\mfk g)(V_0).
\end{equation}

For any $g\in \msr H_{\la \ell_1,\ell_2 \ra}^{n_1,n_2,n} $ and $0<m\in \mbb N$,
\begin{equation}
\td{\Delta }(\eta^m (g )) = m (-\ell_1-\ell_2+n_1-n_2-m+1)\eta^{m-1} (g ) .
\end{equation}
Thus, when $ -\ell_1-\ell_2+n_1-n_2+1 \geq 0$,
\begin{equation}
\eta^{ -\ell_1-\ell_2+n_1-n_2+1}(   {\msr H}_{\la \ell_1,\ell_2 \ra}^{n_1,n_2,n}  )
\subset {\msr H}_{\la n_1-n_2-\ell_2+1,n_1-n_2-\ell_1+1 \ra}^{n_1,n_2,n}.
\end{equation}
Furthermore, for $g \in   \msr H_{\la \ell_1,\ell_2 \ra}^{n_1,n_2,n}  $, we have
\begin{equation} \label{a6.17}
\mfk d_{\la \ell_1+1,\ell_2+1\ra}^{n_1,n_2,n}  (  \eta (g))
=\mfk d_{\la \ell_1 ,\ell_2  \ra}^{n_1,n_2,n}( g)+2.
\end{equation}
Then for $k\in \mbb N$,
\begin{equation} \label{b5.19}
\eta^{ -\ell_1-\ell_2+n_1-n_2+1} \left(\left( \msr H_{\la \ell_1,\ell_2 \ra}^{n_1,n_2,n}  \right)_{k} \right)
\subset \left( {\msr H}_{\la n_1-n_2-\ell_2+1,n_1-n_2-\ell_1+1 \ra}^{n_1,n_2,n}\right)_{k+2( -\ell_1-\ell_2+n_1-n_2+1)},
\end{equation}
when $ -\ell_1-\ell_2+n_1-n_2+1 \geq 0$, by (\ref{a6.13}) and (\ref{a6.17}).

Denote
\begin{equation}
L= (J_3\times J_1) \cup(J_2\times J_1) \cup (J_3\times J_2)
\end{equation}and
\begin{equation}
L^a= (J_1\times J_3) \cup(J_2\times J_3) \cup (J_1\times J_2).
\end{equation}

We define the Lie subalgebra $\mfk a=\text{Span}\{ E_{i,j}, E_{r,r}-E_{r+1,r+1} \mid (i,j)\notin L, 1\leq r\leq n-1 \}$.
We denote
\begin{equation}
\mfk{g}_+ =sl(n)_+=\sum_{1\leq i<j\leq n}\mbb F E_{i,j}.
\end{equation}
\begin{lemma}\label{lem:a6.3}
Let $\mfk g=sl(n)$ and $ M={\msr H}_{\la\ell_1,\ell_2\ra}\neq 0$ be an infinite-dimensional $\mfk g$-module with $n_1<n_2$, $\ell_1 \leq 0$ or $\ell_2\leq 0$. $M_0$ is a finite-dimensional $U(\mfk a)$-invariant subspace (i.e. $U(\mfk a) (X_0) \subset X_0$). If there exists $ f\in M_0 \cap V_0$ with $f\neq 0 $ (cf. (\ref{a2.11})), we have $V_0 \subset M_0$.
\end{lemma}

\begin{proof}

Denote
\begin{eqnarray}
\ol{J}=\begin{cases}
  J_1 & \text{if } \ell_1\leq 0,  \pse\\
 J_2 & \text{if } \ell_1>0 ;
\end{cases}
\qquad
\text{ and }
\quad
\ol{J}'=\begin{cases}
  J_3 & \text{if } \ell_2\leq 0,  \pse\\
 J_2 & \text{if } \ell_2>0  .
\end{cases}
\end{eqnarray}
 Write
\begin{equation}
f= \sum_{k=1}^N \left( a_k \prod_{s=1}^{|\ell_1|}x_{i_{k,s}}\prod_{t=1}^{|\ell_2|}y_{j_{k,t}}  \right)
\end{equation}
with $0 \neq a_k \in \mbb F$, $i_{k,s} \in \ol{J}$, $j_{k,t} \in \ol{J}'$ for $s\in \ol{1,|\ell_1|}$, $t\in \ol{1,|\ell_2|}$ and $k\in\ol{1,N}$.
For convenience, we assume that the monomials $\{\prod_{s=1}^{|\ell_1|}x_{i_{k,s}}\prod_{t=1}^{|\ell_2|}y_{j_{k,t}} \}_{1\leq k\leq N} $ are distinct.

Consider the generators of $V_0$ and we set
\begin{equation}
g=  \prod_{s=1}^{|\ell_1|}x_{i'_s}\prod_{t=1}^{|\ell_2|}y_{j'_t}
\end{equation}
where $i'_s \in \ol{J}$, $j'_t \in \ol{J}'$ for $s\in \ol{1,|\ell_1|}$ and $t\in \ol{1,|\ell_2|}$.
Denote
\begin{equation}
n_x=\vert \{i_{1,s} \mid s\in \ol{1,  |\ell_1| }\}\cap \{ i'_s  \vert 1\leq s \leq |\ell_1| \}\vert,\end{equation}  \begin{equation}
n_y=\vert \{j_{1,t} \mid s\in \ol{1,  |\ell_2| }\}\cap \{ j'_t  \vert 1\leq t\leq |\ell_2| \}\vert.
\end{equation}
 For convenience, we assume that
\begin{equation}
 \{i_{1,s} \mid s\in \ol{1,  n_x}\}= \{i_{1,s} \mid s\in \ol{1,  |\ell_1| }\}\cap \{ i'_s  \vert 1\leq s \leq |\ell_1| \},
\end{equation}
\begin{equation}
 \{j_{1,t} \mid t\in \ol{1,  n_y}\}= \{j_{1,t} \mid t\in \ol{1,  |\ell_2| }\}\cap \{ j'_t  \vert 1\leq t \leq |\ell_2| \},
\end{equation}
and
\begin{eqnarray}
& &n_x+n_y  \leq  \min  \big\{ \left\vert \{i_{k,s},j_{k,t}\mid s\in \ol{1,  |\ell_1|}, t \in \ol{1,  |\ell_2|}\}\cap \{ i'_s,j'_t \vert 1\leq s \leq |\ell_1|, 1 \leq t \leq |\ell_2|\}\right\vert \big\vert \nonumber\\
& &\qquad\qquad  k\in \ol{1,N} \big\} .
\end{eqnarray}

Since $X_0$ is $U(\mfk a)$-invariant, we have $E_{i,j}(X_0) \subset X_0$ if $(i,j) \not\in L$.

Then
\begin{eqnarray}
& &\prod_{t=n_y+1}^{|\ell_2|} E_{j_{1,t},j'_t} \prod_{s=n_x+1}^{|\ell_1|} E_{i_{1,s},i'_s} (f)\nonumber \\
&=&
\prod_{t=n_y+1}^{|\ell_2|}\left( y_{j'_t}   \ptl_{ y_{j_{1,t}}}\right)
\prod_{s=n_x+1}^{|\ell_1|} \left( -x_{i'_s}   \ptl_{ x_{i_{1,s}}}\right)(f)
\nonumber \\
 &=&
\prod_{t=n_y+1}^{|\ell_2|}\left( y_{j'_t}   \ptl_{ y_{j_{1,t}}}\right)
\prod_{s=n_x+1}^{|\ell_1|}\left( -x_{i'_s}   \ptl_{ x_{i_{1,s}}}\right)\left(a_1 \prod_{s=1}^{|\ell_1|}x_{i_{1,s}}\prod_{t=1}^{|\ell_2|}y_{j_{1,t}} \right)
\nonumber \\
&=& c\cdot g\in X_0
\end{eqnarray}
with some $0\neq c\in \mbb F$.
Hence $M_0 \subset X_0$ and the lemma holds.
\end{proof}

\begin{lemma}\label{lemd3.4}
Let $\mfk g=sl(n)$ and let $M={\msr H}_{\la\ell_1,\ell_2\ra} $ be an infinite-dimensional $\mfk g$-module under the conditions  in Lemma \ref{lemc2.1}. If
$N$ is a submodule of $M$, then $N=M$ or $\eta^m ( ({\msr H}_{\la\ell_1-m,\ell_2-m\ra}) ) \subset N$ with $m= \ell_1+\ell_2+n_2-n_1-1$.
\end{lemma}

\begin{proof}
Since $\msr A_{\la \ell_1,\ell_2 \ra}$ is nilpotent with respect to $sl(n)_+$, there exists $v\in N$ such that $\mfk{g} _+(v)=0$.
Then $N=M$ if $\dim (\{ v\in M\vert \mfk{g} _+(v)\})=1$.
According to the Theorem 6.4.3, the Theorem 6.5.1 and the Theorem 6.6.3 in \cite{Xx}, we only need to consider the follwing: (1) $n_1+1<n_2 $ and $\ell_1+\ell_2+n_2-n_1-1>0$; (2) $n_1+1<n_2=n$ and $n_1-n+1-\ell_2 < \ell_1 \leq n_1-n+1$; (3) $n_1+1=n_2 $, $\ell_1+\ell_2>0$ and $\ell_1 \ell_2 \leq 0$.

(1) Case 1 :$n_1+1<n_2 $ and $\ell_1+\ell_2+n_2-n_1-1>0$.\pse

 When $n_1+1<n_2 $ and $\ell_1+\ell_2+n_2-n_1-1>0$, according to Lemma 6.4.2 in \cite{Xx}, when $\mfk{g} _+(g)=0$, yields that
\begin{equation}
g \in \text{Span}\{f^0_{\la \ell_1,\ell_2\ra} ,\eta^{\ell_1+\ell_2+n_2-n_1-1} (f^0_{\la  -\ell_2+n_1+1-n_2, -\ell_1+n_1+1-n_2 \ra}) \},
\end{equation}
where
\begin{eqnarray}
f^0_{\la \ell_1,\ell_2 \ra}  =\begin{cases}
  x_{n_1+1}^{\ell_1}  y_{n_2}^{\ell_2} & \text{if } \ell_1,\ell_2 \geq 0;  \pse\\
  x_{n_1+1}^{\ell_1}  y_{n_2+1}^{-\ell_2} & \text{if } \ell_1\geq 0,\ell_2 \leq 0;  \pse\\
  x_{n_1 }^{-\ell_1}  y_{n_2 }^{\ell_2} & \text{if } \ell_1\leq 0,\ell_2\geq 0;  \pse\\
  x_{n_1 }^{-\ell_1}  y_{n_2+1}^{-\ell_2} & \text{if } \ell_1,\ell_2\leq 0.
\end{cases}
\end{eqnarray}
According to (\ref{a6.13}),
\begin{equation}
\eta^{\ell_1+\ell_2+n_2-n_1-1} (f^0_{\la  -\ell_2+n_1+1-n_2, -\ell_1+n_1+1-n_2 \ra}) \in N
\end{equation}
also yields that $\eta^m ( ({\msr H}_{\la\ell_1-m,\ell_2-m\ra})_0) \subset N$ with $m=\ell_1+\ell_2+n_2-n_1-1$.

Since $\ell_1+\ell_2+n_2-n_1-1>0$, we suppose that
\begin{equation}
g=a f^0_{\la \ell_1,\ell_2\ra} +b \eta^{\ell_1+\ell_2+n_2-n_1-1} (f^0_{\la  -\ell_2+n_1+1-n_2, -\ell_1+n_1+1-n_2 \ra})
\end{equation}
 with $a,b\neq 0$.
We assume $\ell_1>0$ and $\ell_2 \leq 0$, and in another case $\ell_1\leq 0$ and $\ell_2 >0$ is symmetric.
Then
\begin{equation}
g=
\begin{cases}
a x_{n_1+1}^{\ell_1}  y_{n_2+1}^{-\ell_2} +b \eta^{m} (  x_{n_1+1}^{ -\ell_2+n_1+1-n_2}  y_{n_2+1}^{\ell_1-n_1-1+n_2 })
 & \text{if }  -\ell_2+n_1+1-n_2\geq 0,  \pse\\
a x_{n_1+1}^{\ell_1}  y_{n_2+1}^{-\ell_2} +b \eta^{m} (  x_{n_1}^{ \ell_2-n_1-1+n_2}  y_{n_2+1}^{\ell_1-n_1-1+n_2 }) & \text{if } -\ell_2+n_1+1-n_2\leq 0.
\end{cases}
\end{equation}
When $-\ell_2+n_1+1-n_2\leq 0$,
\begin{equation}
E_{n_2,n_1+1}^{ -\ell_2+n_1+2-n_2}(g)=\frac{ a \ell_1 ! x_{n_2}^{ -\ell_2+n_1+2-n_2} x_{n_1+1}^{\ell_1 +\ell_2-n_1+n_2-2}  y_{n_2+1}^{-\ell_2}}{(\ell_1 +\ell_2-n_1+n_2-2)!}   \in N,
\end{equation}
then $l=0$ and $V_0 \subset N$ by Lemma \ref{lem:a6.3}.
When $-\ell_2+n_1+1-n_2\geq 0$ and $n_2<n$, we have
\begin{eqnarray}
E_{n,n_2+1}^{1-\ell_2}(g)&=& \frac{ b (\ell_1-n_1-1+n_2 )!}{(\ell_1 +\ell_2-n_1+n_2-2)!}\eta^{m}\left(   x_{n_1}^{\ell_2-n_1-1+n_2}  y_{n_2+1}^{\ell_1+\ell_2-n_1-2+n_2 }y_n^{1-\ell_2}
\right)\nonumber \\
&\in &\eta^m ( ({\msr H}_{\la\ell_1-m,\ell_2-m\ra})_0).
\end{eqnarray}
Then $\eta^m ( ({\msr H}_{\la\ell_1-m,\ell_2-m\ra})_0) \subset N$ by Lemma \ref{lem:a6.3} and (\ref{a6.13}).
Since we have shown that $U(\mfk g)(V_0)=M$ in \cite{ZX}, we have $M\subset N$ or $\eta^m ( ({\msr H}_{\la\ell_1-m,\ell_2-m\ra}) ) \subset N$.

(2) Case 2: $n_1+1<n_2=n$ and $n_1-n+1-\ell_2 < \ell_1 \leq n_1-n+1$.

When $n_1+1<n_2=n$ and $n_1-n+1-\ell_2 < \ell_1 \leq n_1-n+1$, we have $\ell_2 \geq 0$ (otherwise $M=\{0\}$) and
\begin{equation}
g=
a x_{n_1 }^{-\ell_1}  y_{n }^{ \ell_2} +b \eta^{m} (  x_{n_1 }^{ \ell_2-n_1-1+n}  y_{n }^{-\ell_1+n_1+1-n}).
\end{equation}
Then for $n_1>1$, we have
\begin{eqnarray}
E_{n_1,1}^{1-\ell_1}(g)&=& \frac{ b (\ell_2-n_1-1+n)!}{(\ell_1 +\ell_2-n_1+n_2-2)!}\eta^{m}\left(    x_{n_1 }^{ \ell_2-n_1-2+n+\ell_1}  y_{n }^{-\ell_1+n_1+1-n}x_1^{1-\ell_1}
\right)\nonumber \\
&\in &\eta^m ( ({\msr H}_{\la\ell_1-m,\ell_2-m\ra})_0)
\end{eqnarray}
yields that $\eta^m ( ({\msr H}_{\la\ell_1-m,\ell_2-m\ra})_0) \subset N$.
Consider $n_1=1$, we have
\begin{eqnarray}
E_{n,2}^{-\ell_1+n_1+2-n}(g)&=& \frac{ b  \ell_2  !
 x_{n_1 }^{-\ell_1}  y_{n }^{ \ell_2+\ell_1-n_1-2+n)} y_2^{-\ell_1+n_1+2-n}
}{(\ell_1 +\ell_2-n_1+n_2-2)!}  \in N
\end{eqnarray}
yields that $V_0 \subset N$ by Lemma \ref{lem:a6.3}.

(3) Case $n_1+1=n_2 $, $\ell_1+\ell_2>0$ and $\ell_1 \ell_2 \leq 0$.\pse

According to Lemma 6.6.2 in \cite{Xx},  we assume $\mfk g_+(g)=0 $ with $n_1+1=n_2 $, $\ell_1+\ell_2>0$ and $\ell_1 \ell_2 \leq 0$, and then we have
\begin{eqnarray}
g \in \text{Span} \{  f_{\la \ell_1,\ell_2\ra } ,  \eta^{m}( f_{\la \ell_1-m,\ell_2-m\ra }) ,\eta^{m_1+m_2} ( x_{n_1}^{m_2}y_{n_1+1}^{-m_1}  ) ,\eta^{m_1+m_2 } ( y_{n_1+1}^{ m_2 }x_{n_1+1}^{-m_1 }  )  \}.
\end{eqnarray}
when $n_2<n$, and
\begin{eqnarray}
g &\in& \text{Span} \{  x_{n-1}^{m_1}y_{n}^{m_2} , x_{n }^{m_1} ,\eta^{m_1+1} ( x_{n-1}^{m_1+m_2+1}y_{n}^{m_2}  ) ,
\nonumber \\
& & \qquad
\eta^{m_1+m_2+1} ( x_{n-1}^{ m_2+1}y_{n}^{-m_1-1}  )  \vert m_1,m_2\in \mbb N \}.
\end{eqnarray}
when $n_2=n$.

According to (6.5.60) and (6.5.62) in \cite{Xx}, we have $\dim (\{ v\in M\vert \mfk{g} _+(v)\}) \leq 2 $, and we can use the similar method as (1) and (2) to prove that the lemma still holds in this case.
\end{proof}

\begin{lemma}\label{lem:b5.5}
Let $\mfk g=sl(n)$ and $M={\msr H}_{\la\ell_1,\ell_2\ra} $ be an infinite-dimensional $\mfk g$-module, under the conditions stated in Lemma \ref{lemc2.1}. Let $v\in M $ and $(U(\mfk g))(\mbb F v)=M$, the necessary and sufficient condition for $\mfk p_{M}(k) = \mfk p_{M,\mbb F v}(k)$ is $\mbb Fv=V_0$.
\end{lemma}
\begin{proof}
We suppose that $ \mfk p_{M}(k) = \mfk p_{M,\mbb F v}(k)$.

According to Lemma \ref{lem:b5.2}, we have $I_{(i)} =\left(\operatorname{Ann}_{S( \mfk g)} (\ol{M}) \right)^{i}$ for $i$ large enough.

\psp

(1)  Consider the case of $n_1+1<n_2<n$.\pse

(1.a) Suppose that $|J_1| \geq 2$ and $|J_3| \geq 2$.\pse

For $i,i'\in J_1$ with $i\neq i'$, $E_{i,i'} \in I_{(1)}$. Assume that $E_{i,i'}(v)\neq 0$, then $E_{n,i'}^{p-1} E_{i,i'}\in I_{(p)} =\left(\operatorname{Ann}_{S( \mfk g)} (\ol{M}) \right)^{p}$  for $p$ large enough.

Denote
\begin{equation} \label{d3.46}
(\deg_{x_{n}}+\deg_{x_{i'}}+\deg_{y_{n}}+\deg_{y_{i'}})(v)=d_0
\end{equation}
then
\begin{equation}
(\deg_{x_{n}}+\deg_{x_{i'}}+\deg_{y_{n}}+\deg_{y_{i'}})(E_{n,i'}^{p-1} E_{i,i'}(v))=d_0+2p-1.
\end{equation}
On the other hand, for each $f\in (U_{p-1}(\mfk g))(\mbb F v) $, we have
\begin{equation} \label{d3.48}
(\deg_{x_{n}}+\deg_{x_{i'}}+\deg_{y_{n}}+\deg_{y_{i'}})(f)\leq d_0+2p-2.
\end{equation}
Hence $E_{i,i'}(v)=0$ for $i,i'\in J_1$ with $i\neq i'$.
Similarly, we have $E_{j,j'}(v)=0$ for $j,j'\in J_3$ with $j\neq j'$.

Assume that $E_{r,r'}(v)\neq 0$ for $r,r'\in J_2$ with $r\neq r'$.
Then $E_{n,1}^{p-1} E_{r,r'}(v) \in  (U_{p-1}(\mfk g))(\mbb F v)$ for $p$ large enough.
Write $v=\sum_{i=1}^m v_i $, where $v_i$ are monomials in $\msr A$ for $i\in{1,m}$.
For $d\in \mbb N$, we set
\begin{equation}
v'=\sum_{\{ i:{\mfk d} (v_i)= \mfk d (v ) \}}v_i
\end{equation}
Note that $ \mfk d \left( E_{n,1}^{p-1} E_{r,r'}(v') \right) =2p-2+ \mfk d  (v )$, then there exist $N\in \mbb N$ and $f_i \in Z^{(p-1)}$ for $i\in \ol{1,N}$ such that
 \begin{equation} \label{b5.76}
\phi(z_{n,1})^{p-1} E_{r,r'}(v') = \sum_{i=1}^N \phi(f_i)  v'.
\end{equation}
Since $\deg_{x_{r'}}(E_{r,r'}(v')) =\deg_{x_{r'}}( v')-1$, we have
\begin{equation}
\deg_{x_{r'}}\left( \phi(z_{n,1})^{p-1} E_{r,r'}(v') \right) \neq \deg_{x_{r'}}\left(\sum_{i=1}^N \phi(f_i)  v'\right).
\end{equation}
Then $ E_{r,r'}(v')=0$, and so $v'=0$ by (\ref{b5.76}).
Hence $E_{r,r'}(v  )=0$ for $r,r'\in J_2$ with $r\neq r'$.

Denote $w_1=x_1y_2-x_2y_1$, $w_n=x_ny_{n-1}-x_{n-1}y_n$ and $\eta_2=\sum_{r\in J_2}x_r y_r$.
When $n_1+1<n_2$, $|J_1|\geq 2$ and $|J_3| \geq 2$, $v$ can be written as a rational function $g$ in $\{x_i,y_j, w_1,w_n,\eta_2,x_{n_1+1},y_{n_2} \vert 2\leq i \leq n,1\leq j\leq n-1 \}$ according to Lemma 6.2.3 in \cite{Xx}.
Since $E_{n_2,n_1+1}(v)=0$, we have
 \begin{equation}  \label{b5.85}
E_{n_2,n_1+1}(v)= x_{n_2}g_{x_{n_1+1}}-y_{n_1+1}g_{y_{n_2}}=0.
\end{equation}
Then $g_{x_{n_1+1}}=g_{y_{n_2}}=0$.

When $|J_1| \geq 2$, we have
 \begin{equation}
E_{1,2}(v)=(-x_2\ptl_{x_{1}}-y_2\ptl_{y_{1}})(v)=g_{y_1}=0.
\end{equation}
For $i\in \ol{2,n_1}$,
 \begin{eqnarray}
E_{i,1}(v)&=&(-x_1\ptl_{x_{i}}-y_1\ptl_{y_{i}})(v)
\nonumber \\
&=& -x_1( g_{x_i}+x_2^{-1} y_2 g_{y_i})-x_2^{-1}w_1 g_{y_i}=0.
\end{eqnarray}
Since $g_{x_i}+x_2^{-1} y_2 g_{y_i}$ and $x_2^{-1}w_1 g_{y_i}$ are independent of $x_1$, we have
$ g_{x_i}+x_2^{-1} y_2 g_{y_i} =0$ and $x_2^{-1}w_1 g_{y_i}=0$. Thus $ g_{x_i}=g_{y_i}=0$ for $i\in J_1$ when $|J_1| \geq 2$.

Similarly, we have $g_{x_j}=g_{y_j}=0$ for $j\in J_3$ when $|J_3| \geq 2$.

Assume that $E_{i_1,i_2}(v)\neq 0$ for some $i_1\in J_1$ and $i_2 \in J_2$. When $|J_1|\geq 2$, we denote
\begin{equation}
(\deg_{x_{n}}+\deg_{x_{i'}}+\deg_{y_{n}}+\deg_{y_{i'}})(v)=d'_0.
\end{equation}
for some $i_1\neq i ' \in J_1$.
Then
\begin{equation}
(\deg_{x_{n}}+\deg_{x_{i'}}+\deg_{y_{n}}+\deg_{y_{i'}})\left( E_{n,i'}^{p-1}E_{i_1,n_2}v\right)=2p+d_0'-2.
\end{equation}
For each $f\in (U_{p-1}(\mfk g))(\mbb F v) $, we have
\begin{equation}
(\deg_{x_{n}}+\deg_{x_{i'}}+\deg_{y_{n}}+\deg_{y_{i'}} )(f)\leq d_0'+2p-2.
\end{equation}
Moreover, the equality holds if and only if $ (y_ny_{i'}-x_nx_{i'})^{p-1} \vert f$.
The inequality holds with equality if and only if $ (y_ny_{i'}-x_nx_{i'})^{p-1} \vert f$.
Hence $E_{n,i'}^{p-1}E_{i_1,i_2}v\in (U_{p-1}(\mfk g))(\mbb F v)$ yields that $E_{i_1,i_2}v \in \mbb F v$ which leads to a contradiction. Therefore $E_{i_1,i_2}(v)=0$ for $i_1\in J_1$ and $i_2\in J_2$.

Then we have
\begin{eqnarray}\label{d3.58}
(\ptl_{x_{1}}\ptl_{x_{n_2}}-y_{n_2}\ptl_{y_{1}})(v)=
g_{\eta_2 ,w_1} y_2 y_{n_2} -y_{n_2}x_2 g_{w_1}=0.
\end{eqnarray}
Then there exists a rational function $g'$ such that $\frac{\ptl (g')}{\ptl \eta_2 }=0$ and
\begin{equation} \label{d3.59}
g_{w_1} = e^{\frac{x_2}{y_2}\eta_2} g'
\end{equation}
which is not a rational function when $g'\neq 0$.
Hence $g_{w_1}=0$.
Similarly, $E_{i_3,i_2}(v)=0$ for $i_3\in J_3$,$i_2 \in J_2$ and $g_{w_n}=0$ when $|J_3| \geq 2$. Hence $v$ only depends on $ \eta_2 $.

Since $  \eta_2 =\eta(1)$, $\mbb F [\eta_2] =\text{Span}\{\eta^i(1)\vert i\in \mbb N\}$.
Note that $1\in  {\msr H}_{\la 0,0 \ra} $ and $\eta^i(1) \in {\msr H}_{\la i,i \ra} $, we have $v=\eta^{i_0}(1)$ for some $i_0 \in \mbb N$.
Then $i_0=m$ and $\ell_1-m=\ell_2-m=0$ with $m=\ell_1+\ell_2+n_2-n_1-1$ by (\ref{b5.19}). Hence $i_0=n_1+1-n_2 \leq 0$ yields that $i_0=0$.
Therefore, we have $\ell_1=\ell_2=0$ and $v\in \mbb F$.\pse

 (1.b) Suppose that $|J_1|=1$ and $|J_3| \geq 2$. \pse

In this case, we also have $E_{i,i'}(v)=0$ with $i,i' \in J_j$ and $i'\neq i$ for $j=2,3$. Then $v$ can be written as a rational function $g$ in $\{x_1,y_1, w_n,\eta_2 \}$.

Assume that $E_{1,i_2}(v)\neq 0$ for some $i_2 \in J_2$. Suppose that $\ptl_{y_{1}}(v)=0$, we denote
\begin{equation} \label{d3.60}
(\deg_{y_{n}}+ \deg_{y_{1}} )(v)=\deg_{y_{n}}(v)=d_0.
\end{equation}
Then
\begin{equation}
(\deg_{y_{n}}+ \deg_{y_{1}} )\left( E_{n,1}^{p-1}E_{1,i_2}(v)\right)=2p-2+d_0  .
\end{equation}
For each $f\in (U_{p-1}(\mfk g))(\mbb F v) $, we have
\begin{equation}
(\deg_{y_{n}}+ \deg_{y_{1}}  )(f)\leq d_0 +2p-2
\end{equation}
and $(\deg_{y_{n}}+ \deg_{y_{1}}  )(f)\leq d_0+2p-2 $ if and only if $f\in \mbb F E_{n,1}^{p-1}(v)$.
Hence $E_{n,1}^{p-1}E_{1,i_2}(v)\in \mbb F E_{n,1}^{p-1}(v) $ yields that $E_{1,i_2}(v)=0$.
Now we suppose that $\ptl_{y_{1}}(v) \neq 0$, and denote
\begin{equation}
(\deg_{x_{1}}+ \deg_{x_{n}} )(v)= (\deg_{x_{1}}+ \deg_{x_{n}} )(E_{1,i_2}(v))=d'_0.
\end{equation}
Then
\begin{equation} \label{d3.64}
(\deg_{x_{1}}+ \deg_{x_{n}} )\left( E_{n,1}^{p-1}E_{1,i_2}(v)\right)=2p-2+d_0' ,
\end{equation}
which can also yields that $E_{1,i_2}(v)=0$.
Note that we also have $E_{i_3,i_2}(v)=0$ and $g_{w_n}=0$ since $|J_2| \geq 2$. Then $v$ only depends on $ \{ x_1,y_1,\eta_2 \}$.
Hence $E_{1,j}(v) =0$ for $j\in J_3$.
Therefore, $\mbb Fv$ is a $U(\mfk a)$-invariant subspaces. Note that $\mfk g_+(v)=0$ yields that $v\in V_0$ by Lemma 6.4.2 in \cite{Xx}. So we have $V_0 \subset \mbb Fv$ by Lemma \ref{lem:a6.3} and the proof of Lemma \ref{lemd3.4}. Hence $\dim V_0 =1$ yields $\ell_1 \leq 0$, $\ell_2=0$ and $\mbb F v= \mbb F x_1^{-\ell_1}$.\pse

(1.c) The case of $|J_3| =1$ and $|J_1| \geq 2$ is symmetrical to the case (1.b), and we have $\ell_1= 0$, $\ell_2 \leq 0$ and $\mbb F v= \mbb F y_n^{-\ell_2}$.\pse

(1.d) When $|J_1|=|J_3|=1$, we have $n_1=1$ and $n_2=n-1$.\pse

Since (\ref{d3.60})-(\ref{d3.64}) still holds, we also have $E_{1,i}(v)=E_{n,i}(v)=0$ for $i\in J_2$.
Assume that $E_{n,1}(v) \neq 0$, then $\ptl_{y_1}\ptl_{y_n}(v) \neq 0$ or $\ptl_{x_1}\ptl_{x_n}(v) \neq 0$.
When $\ptl_{x_1}\ptl_{x_n}(v) \neq 0$, we denote
\begin{equation}
(\deg_{y_{n}}+ \deg_{y_{1}} )(v) =d_0.
\end{equation}
Then
\begin{equation}
(\deg_{y_{n}}+ \deg_{y_{1}} )\left( E_{n,1}^{p-1}E_{1,n}(v)\right)=2p-2+d_0  .
\end{equation}
For each $f\in (U_{p-1}(\mfk g))(\mbb F v) $, we have
\begin{equation}
(\deg_{y_{n}}+ \deg_{y_{1}}  )(f)\leq d_0 +2p-2 ,
\end{equation}
which yields $E_{n,1}(v)=0$.
Similarly, $E_{n,1}(v)=0$ also holds when $\ptl_{x_1}\ptl_{x_n}(v) \neq 0$.
Then $\mfk g_+(v)=0$ and $\mbb Fv$ is a $U(\mfk a)$-invariant subspaces.
According to Lemma 6.4.2 in \cite{Xx} and Lemma \ref{lem:a6.3}, we have $\dim V_0 =1$ which yields that $\ell_1 \leq 0$, $\ell_2 \leq 0$ and $\mbb F v= \mbb F x_1^{-\ell_1}y_n^{-\ell_2}$.\psp

(2) Consider the case of $n_1+1<n_2=n$.\psp

(2.a) Suppose that $|J_1| \geq 2$, we also have $E_{i,i'}(v)=0$ for $i,i'\in J_j$, $i'\neq i$ and $j=1,2$.\pse

Assume that $E_{i_1,i_2}(v)=0$ for some $i_1 \in J_1$ and $i_2\in J_2$.
For some $i_1'\in J_1$ and $i_2' \in J_2$ with $i_1' \neq i_1$ and $i_2' \neq i_2$, we denote
\begin{equation}
(\deg_{x_{i_2'}}+\deg_{x_{i_1'}} )(v)=d_0
\end{equation}
then
\begin{equation}
(\deg_{x_{i_2'}}+\deg_{x_{i_1'}} )(E_{i_2',i_1'}^{p-1} E_{i_1,i_2}(v))=d_0+2p-2.
\end{equation}
On the other hand, for each $f\in (U_{p-1}(\mfk g))(\mbb F v) $, we have
\begin{equation}
(\deg_{x_{i_2'}}+\deg_{x_{i_1'}} )(f)\leq d_0+2p-2.
\end{equation}
and $(\deg_{x_{i_2'}}+\deg_{x_{i_1'}} )(f)\leq d_0+2p-2$ if and only if $f \in \mbb F (E_{i_1',i_2'}^{p-1}(v))$.
Hence $E_{i_1,i_2}(v)=0$ for $i_1 \in J_1$ and $i_2  \in J_2$.\pse

(2.b) Suppose that $|J_1|=1$ and $E_{1,i}(v) \neq 0$ for some $i\in J_2$.\pse

When $\ptl_{y_1}(v) \neq 0$, $\deg_{x_1}(v)=\deg_{x_1}(E_{1,i}(v))$.
For $i'\in J_2$ with $i'\neq i$, we have
\begin{equation}
\deg_{x_1}(E_{i',1}^{p-1}E_{1,i}(v)) =p-1+\deg_{x_1}(v),\,\,
\deg_{x_{i'}}(E_{i',1}^{p-1}E_{1,i}(v)) =p-1+\deg_{x_3}(v),\;
\end{equation}
yields that $E_{1,i}(v)=0$.
When  $\ptl_{y_1}(v) = 0$, we have \begin{equation}\deg_{y_1} (E_{i',1}^{p-1}E_{1,i} (v))=p-1,\,\,  \deg_{x_{i'}}(E_{i',1}^{p-1}E_{1,i}(v)) =p-1+\deg_{x_3}(v),\end{equation} which also yields that $E_{1,i}(v)=0$.
Hence $\mbb Fv$ is a $U(\mfk a)$-invariant subspaces and $\mfk g_+(v)=0$ for the cases (2.a) and (2.b).

When $\ell_1,\ell_2>0$, we have $v\in \mbb F x_{n_1+1}^{\ell_1} y_{n_2}^{\ell_2}$. However,
\begin{equation}
E_{n_2,n_1+1}(x_{n_1+1}^{\ell_1} y_{n_2}^{\ell_2})
=\ell_1 x_{n_2} x_{n_1+1}^{\ell_1-1} y_{n_2}^{\ell_2} -
\ell_2 y_{n_1+1} x_{n_1+1}^{\ell_1} y_{n_2}^{\ell_2-1} \neq 0.
\end{equation}
Hence $\ell_1\leq 0$ or $\ell_2 \leq 0$, and we have $V_0\subset \mbb Fv$ by Lemma \ref{lem:a6.3}.
Then $n_1=1$, $\ell_1\leq 0$, $\ell_2=0$ and $\mbb F v = \mbb F x_1^{-\ell_1}$.\psp

(3) The case of $ n_1 + 1 = n_2 $.\psp

(3.a) Suppose that $|J_1|\geq 2$ and $|J_3|\geq 2$.\pse

Compare with the case (1.a), the main difference lies in $ \eta_2 = x_{n_2} y_{n_2} $ and we do not have the expression (\ref{b5.85}).  Similar to (1,a), we know that $v$ can be written as a rational function $g$ in $\{ x_{n_2},y_{n_2} \}$.
Since $\ell_1\leq 0$ or $\ell_2 \leq 0$, $ \deg_{x_{n_2}}(v) $ and $ \deg_{y_{n_2}}(v) $ cannot both be positive, so we have $ v \in \mathbb{F}[x_{n_2}] $ or $ v \in \mathbb{F}[y_{n_2}] $.
There are only three possible cases:  (3.a.1)  $ \mbb F v = \mbb F x_{n_2}^{\ell_1}$  with $\ell_1 > 0$  , $\ell_2 = 0 $; (3.a.2)   $\mbb F v =  \mbb F y_{n_2}^{\ell_2}$   with  $ \ell_2 > 0 $, $\ell_1 = 0$; (3.a.3) $v \in \mbb F$  with  $\ell_1=\ell_2=0$.\pse

(3.b) Suppose that $|J_1|= 1$ and $|J_3|\geq 2$, then $n_1=1$ and $n_2=2$.\pse

Similar to (1.b), (\ref{d3.60})-(\ref{d3.64}) still holds. Hence $E_{1,2}(v)=0$ and $v$ can be written as a rational function $g$ in $\{ x_{1},y_{1},x_{2},y_{2} \}$.
Then $\mfk g_+(v)=0$ and $\mbb F v $ is a $U(\mfk a)$-invariant subspaces.
According to the Theorem 6.5.1 in \cite{Xx} and Lemma \ref{lem:a6.3} and the proof of Lemma \ref{lemd3.4}, $V_0 =\mbb F v$ and $\dim (V_0)=1$. There are only two possible cases: (3.b.1)   $ \mbb F v =  \mbb F x_{1}^{-\ell_1}y_2^{\ell_2}$  with  $\ell_1 \leq 0$  , $\ell_2 \geq 0 $; (3.b.2)   $\mbb F v =   \mbb F x_{ 2}^{\ell_1}$  and $ \ell_1 > 0 $, $\ell_2 = 0$.\pse

(3.c)  The case of $|J_3| =1$ and $|J_1| \geq 2$ is symmetrical to the case (3.b), and we have two possible cases: (3.c.1)   $ \mbb F v = \mbb F x_{n-1}^{\ell_1}y_n^{-\ell_2}$  with  $\ell_1 \geq 0$  , $\ell_2 \leq 0 $; (3.c.2)   $\mbb F v =  \mbb F y_{n-1}^{\ell_2}$  and $ \ell_1 = 0 $, $\ell_2 > 0$. \pse

(3.d) When $|J_1|=|J_2|=1$, we have $n_1=1$, $n_2=2$ and $n=3$.\pse

Note that
\begin{equation}
V_0= \begin{cases}
\mbb F  x_{2}^{\ell_1}  y_{3}^{-\ell_2}    & \text{if } \ell_1\geq 0,\ell_2 \leq 0;  \pse\\
\mbb F x_{1}^{-\ell_1}  y_{2}^{\ell_2}   & \text{if } \ell_1\leq 0,\ell_2\geq 0;  \pse\\
\mbb F x_{1}^{-\ell_1}  y_{3}^{-\ell_2}  & \text{if } \ell_1,\ell_2\leq 0.
\end{cases}
\end{equation}
Similar to (1.d), we also have $\mfk g_+ (v)=0$ and $V_0=\mbb F v$.\pse

(3.e) Suppose that $ 2<n_1 + 1 = n_2 = n $.\pse

For $i,i' \in J_1$ with $i\neq i$, we have $E_{n ,i'}^{p-1}E_{i, } (v) \in U_{p-1}(v)$ for $p$ large enough, we have $ E_{i,n}(v)=0$ by considering  $(\deg_{x_{i'}}+\deg_{x_n})(E_{n ,i'}^{p-1}E_{i, } (v))$.
Hence $v$ can be written as a rational function $g$ in $\{ x_{n},y_{n} \}$. When $ \ell_1, \ell_2 > 0 $, $ v \in \mathbb{F}[x_{n_2}, y_{n_2}] $ is not possible. Thus, we only need to consider the cases where $ \ell_1 \leq  0 $ or $ \ell_2 \leq  0 $. There are only two possible cases: (3.e.1)   $ \mbb F v = \mbb F x_{n}^{ \ell_1} $  with  $\ell_1 \geq 0$  , $\ell_2 =0 $; (3.e.2)   $  \mbb F v =\mbb F y_{n}^{ \ell_2} $  with  $\ell_1 = 0$  , $\ell_2 \geq 0 $.\pse

(3.f) Suppose that $n_1=1$, $n_2=n=2$. \pse

Since $E_{2,1}^{p-1}E_{1,2} (v) \in U_{p-1}(v)$ for $p$ large enough, we have $ E_{1,2}(v)=0$ by considering  $(\deg_{x_1}+\deg_{x_2})(E_{2,1}^{p-1}E_{1,2} (v))$.
Thus  we have
 \begin{equation}
\mbb F v= \begin{cases}
\mbb F x_{1}^{-\ell_1}  y_{2}^{\ell_2}    & \text{if } \ell_1\leq 0,\ell_2\geq 0;  \pse\\
\mbb F x_{1}^{ \ell_1}     & \text{if } \ell_1\geq 0,\ell_2 = 0;  \pse\\
\mbb F   \eta^{\ell_1+\ell_2}(x_1^{  \ell_2}   y_{2}^{- \ell_1} )  & \text{if } \ell_1,\ell_2\geq 0.
\end{cases}
\end{equation}
Note that
\begin{eqnarray}
& &\eta^{\ell_1+\ell_2 } ( x_{n-1}^{ \ell_2 }y_{n}^{-\ell_1}  )\nonumber \\
&=& \left( \sum_{i=1}^{n-1} y_i\ptl_{x_i} +x_n y_n \right)^{\ell_1+\ell_2 }( x_{n-1}^{ \ell_2 }y_{n}^{-\ell_1}  ) \nonumber\\
&=& \sum_{j=0}^{\ell_1+\ell_2 } \left( \sum_{i=1}^{n-1} y_i\ptl_{x_i}\right)^{\ell_1+\ell_2 -j}(x_n y_n)^{j}( x_{n-1}^{ \ell_2 }y_{n}^{-\ell_1}  ) \nonumber\\
& \equiv &  \left(   y_{n-1}\ptl_{x_{n-1}}\right)^{ \ell_2 }(x_n y_n)^{\ell_1}( x_{n-1}^{ \ell_2 }y_{n}^{-\ell_1}  )
\quad(\mbox{mod }  x_n y_n  )\nonumber\\
& \equiv &  \ell_2!  x_n^{\ell_1} y_{n-1}^{\ell_2 } \quad(\mbox{mod }  x_n y_n  ),
\end{eqnarray}
when $n_1+1=n_2=n$ and $\ell_1,\ell_2>0$, then
we have
\begin{equation}
\mbb F \eta^{\ell_1+\ell_2 } ( x_{n-1}^{ \ell_2 }y_{n}^{-\ell_1}  ) =
\mbb F  T\left(x_n^{\ell_1} y_{n-1}^{\ell_2 } \right) \subset V_0.
\end{equation}
Therefore, $\dim (V_0)=1$ when  $n_1+1=n_2=n$ and $\ell_1,\ell_2>0$ if and only if $n_1=1$ and $V_0=\mbb F   \eta^{\ell_1+\ell_2}(x_1^{  \ell_2}   y_{2}^{- \ell_1} )$.\psp

(4) When $ n_1 = n_2 $.\pse

 Using the above method along with Lemma 6.6.2 and the Theorem 6.64 in \cite{Xx}, we can also conclude that: (4.1) $v \in \mbb F$ and $\ell_1=\ell_2=0$; (4.2) $\mbb F v=  \mbb F (x_1y_2-x_2y_1)^{ \ell_2}$ and $\ell_2 =-\ell_1 > 0$  when $n_1=2<n $; (4.3) $\mbb F v=  \mbb F (x_ny_{n-1}-x_{n-1}y_n )^{ \ell_1}$ and $\ell_1 =-\ell_2 > 0$  when $n_1=n-2 $; (4.4) $\mbb F v=  \mbb F  x_1^{-\ell_1} y_2^{-\ell_2}$ and $\ell_1,\ell_2  \leq 0$  when $n_1=1$ and $n=2$.

\psp

In conclusion, $\mbb Fv=V_0$ if $I_{(i)} =\left(\operatorname{Ann}_{S( \mfk g)} (\ol{M}) \right)^{i}$.
Furthermore, it is easy to check that $I_{(i)} =\left(\operatorname{Ann}_{S( \mfk g)} (\ol{M}) \right)^{i}$ when $\mbb Fv=V_0$ from Lemma 4.1, Lemma 4.2, Lemma 4.3 and Lemma 4.5 in \cite{ZX}.

 \end{proof}

Now we have our main theorem in this section:

\begin{theorem} \label{prope3.1}
Let $\mfk g=sl(n)$ and $M={\msr H}_{\la\ell_1,\ell_2\ra} $ be an infinite-dimensional $\mfk g$-module under the conditions stated in Lemma \ref{lemc2.1}. $M_0$ is a subspace of $M $ satisfying $(U(\mfk g))(M_0)=M$, then $\mfk p_{M }(k) \leq \mfk p_{M,M_0}(k)$. The necessary and sufficient condition for $\mfk p_{M }(k)=\mfk p_{M,M_0}(k)$ is $M_0=V_0$ and $\dim (V_0)=1$.
Specifically, $M_0$ corresponds to one of the following cases:

(1) $M_0=\mathbb{F}$ when $\ell_1=\ell_2=0$;

(2) $M_0=\mathbb{F}x_1^{-\ell_1}$ when $n_1=1$, $\ell_1 \leq 0$ and $\ell_2=0$;

(3) $M_0=\mathbb{F}y_n^{-\ell_2}$ when $n_2=n-1$, $\ell_1 = 0$ and $\ell_2\leq 0$;

(4) $M_0=\mathbb{F}x_1^{-\ell_1}y_n^{-\ell_2}$ when $n_1=1$, $n_2=n-1$, $\ell_1\leq 0$ and $\ell_2\leq 0$;

(5) $M_0=\mathbb{F} y_n^{ \ell_2}$ when $n_2=n$, $\ell_1=0$ and $\ell_2\geq 0$;

(6) $M_0=\mathbb{F}x_{n_2}^{ \ell_1}$ when $n_1+1=n_2$, $\ell_1 \geq 0$ and $\ell_2= 0$;

(7) $M_0=\mathbb{F} y_{n_2}^{ \ell_2}$ when $n_1+1=n_2$, $\ell_1=0$ and $\ell_2\geq 0$;

(8) $M_0=\mathbb{F} x_{1}^{ -\ell_1} y_{n_2}^{ \ell_2}$ when $n_1+1=n_2=2$, $\ell_1 \leq 0$ and $\ell_2\geq 0$;

(9) $M_0=\mathbb{F} x_{n_2}^{ \ell_1} y_{n}^{ -\ell_2}$ when $n_1+1=n_2=n-1$, $\ell_1 \geq 0$ and $\ell_2\leq 0$;

(10) $M_0=  \mbb F (x_1y_2-x_2y_1)^{ \ell_2}$ when $n_1=n_2=2<n$, $\ell_2=-\ell_1 > 0$;

(11) $M_0=  \mbb F (x_ny_{n-1}-x_{n-1}y_n )^{ \ell_1}$ when $n_1=n_2=n-2$, $\ell_1 =-\ell_2 > 0$;

(12) $M_0=  \mbb F  T\left(x_2^{\ell_1} y_{1}^{\ell_2 } \right)$ when $n_1=1$, $n_2=n=2$ and $\ell_1,\ell_2 > 0$.
\end{theorem}

\begin{proof}
For $u \in M_0$, we set $M_{\{u\}}= U(\mfk g) (\mbb F u)$, which is a $ U(\mfk g) $-submodule of $M$.
According to Lemma \ref{lemd3.4}, $M\subset M_{\{u\}}$ or $\eta^m({\msr H}_{\la\ell_1-m,\ell_2-m\ra})\subset M_{\{u\}}$ with $m= \ell_1+\ell_2+n_2-n_1-1$.

Assume that there exists $u_0\in M_0$ such that $M_{\{u_0\}}=M$,
then we have
\begin{equation}
\mfk p_{M,M_0}(k) \geq \mfk p_{M_{\{u_0\}},\mbb F u_0}(k) \geq \mfk p_{M}(k)
\end{equation}
according to Lemma \ref{lemd3.1} and Lemma \ref{lem:b5.2}.

Suppose that $\eta^m({\msr H}_{\la\ell_1-m,\ell_2-m\ra})\subset M_{\{u_0\}}$ with $u_0\in M_0$, 
we define a linear map $\zeta$ from ${\msr H}_{\la\ell_1-m,\ell_2-m\ra}$ to $M_{\{u_0\}}$ by $\zeta(f)=\eta^m(f)$.
Thus, $\zeta$ is injective.
There exists $u_0' \in {\msr H}_{\la\ell_1-m,\ell_2-m\ra}$ such that $f_0(\lambda u_0) = \eta^m( u_0')$ for some $f_0\in U_k(\mfk g)$ and $\lambda \in\mbb F$, and we have $\zeta\left( U_k(\mfk g)  (\mbb F u_0')\right)  \subset  U_k(\mfk g)(\mbb F f_0(\lambda u_0)) $.
Hence
\begin{equation}
\mfk p_{M,M_0}(k)  \geq  \mfk p_{M_{\{u_0\}},\mbb F u_0 }(k) \geq  \mfk p_{\eta^m({\msr H}_{\la\ell_1-m,\ell_2-m\ra}),\mbb F f_0(\lambda u_0) }(k) \geq\mfk p_{{\msr H}_{\la\ell_1-m,\ell_2-m\ra}, \mbb F u_0' }(k)
\end{equation}
by Lemma \ref{lemd3.1}.
According to the main theorem in \cite{ZX}, we have $\msr V( {\msr H}_{\la\ell_1-m,\ell_2-m\ra})=\msr V( {\msr H}_{\la\ell_1 ,\ell_2 \ra})$.
Therefore,
\begin{equation}
\mfk p_{M,M_0}(k)    \geq \mfk p_{ {\msr H}_{\la\ell_1-m,\ell_2-m\ra} }(k)
 = \mfk p_{M}(k).
\end{equation}

Now we suppose that $ \mfk p_{M}(k) = \mfk p_{M, M_0}(k) $, then all the inequalities mentioned above must hold as equalities.
When $M_{\{u_0\}}\neq M$, the linear map from $M_{\{u_0\}}$ to $M$ satisfying $U_k(\mfk g)( \mbb F u_0) \mapsto U_k(\mfk g)( \mbb F u_0) \subsetneq U_k(\mfk g)M_0$ for sufficiently large $k$, thus
\begin{equation}
\mfk p_{M,M_0}(k)  > \mfk p_{M_{\{u_0\}},\mbb F u_0 }(k).
\end{equation}
Then $\mfk p_{M,M_0}(k)  >\mfk p_{M}(k)$ holds.

Hence we suppose that $M_{\{u_0\}}=M$, using the conclusion in Lemma \ref{lem:b5.5},  we have $ \mathbb{F}u_0 = V_0 $ and $\dim( V_0)=1$.
Moreover, $\dim(U_k(\mfk g)(M_0)) =\dim(U_k(\mfk g) (\mbb F u_0))$ for sufficiently large $k$.

Assume that there exists $v\in M_0  \setminus \mbb F u_0$, then $ U_k(\mfk g)(\mbb Fv)\subset  U_k(\mfk g)(\mbb Fu_0)$ for sufficiently large $k$.
Denote
\begin{equation} \label{e4.80}
m_0=\min \{ k\in \mbb F \vert v\in U_k(\mfk g)(\mbb Fu_0) \}>0
\end{equation}
 and $v=f(\lambda' u_0)$, where $f\in U_{m_0}(\mfk g)$ and $\lambda'  \in \mbb F$.
Note that $I=I (\msr V(M_{\{u_0\}}))= \operatorname{Ann}_{S( \mfk g)} (\ol{M_{\{u_0\}}}) $ is a prime ideal in this case, since $\mathbb{F}u_0 = V_0$.
We denote by $\ol{g} \in S^p(\mfk g)$ the image of $g$ under the natural map from $U_k(\mfk g)$ to $U_k(\mfk g)/U_{k-1}(\mfk g)$.
For sufficiently large $k$ and $g\in U_k(\mfk g)$ such that $\ol{g} \not\in I$, we have
\begin{equation}
 (gf)(\lambda'  u_0)  =g(v) \in  U_k(\mfk g)(\mbb Fv)  \subset U_k(\mfk g)(\mbb Fu_0)\subset U_{k+m_0-1}(\mfk g) (\mbb Fu_0).
\end{equation} 
Then $\ol{gf} \in I$ yields that $\ol{f} \in I$.
Thus $v=f(\lambda' u_0) \in U_{m_0-1}(\mbb Fu_0)$, which contradicts (\ref{e4.80}).
Therefore, $M_0= \mbb F u_0 =V_0$.

\end{proof}

This completes the proof of Theorem 2 that we stated in the introduction.

\section{Leading Coefficient of $\mfk p_{M,M_0}$}

First, we will show that the leading coefficient of $\mfk p_{M,M_0}$ is independent of the choice of $M_0$.
For $\mfk g=sl(n)$, $M=\msr H_{\la \ell_1,\ell_2 \ra}$, we compute $lc(\mfk p_{M,V_0})$ and find the relationship between $lc(\mfk p_{M,M_0})$  and  $lc(\mfk p_{M })$.
For the case where $n_1 < n_2<n$ and $\ell_1\leq 0$ or $\ell_2 \leq 0$, we identify a subspace that is isomorphic to $\ol{M_k}$ and find a corresponding basis. To construct this basis, we introduce the concept of ``3-chain", which is a natural extension of the 3-chain notion introduced in \cite{ZX}.
Moreover, we define a map from the basis of $\ol{M_k}$ to the basis of $\ol{(\msr H_{\la 0,0 \ra})_k}$
to compute the leading coefficient.
Another case which is relatively more complicated is $n_1=n_2<n$ and $\ell_1\ell_2<0$. The main difficulty lying in the structure of $ M_0$. We introduce the concepts of $\omega$-2-chain and $\omega$-3-chain, from which we construct a basis and give a formula of $\mfk p_{M,V_0}$.

\begin{proposition}\pse
Let $M_0,M_0'$ be finite-dimensional subspaces of $M$ satisfying $(U(\mfk g))(M_0)=(U(\mfk g))(M_0')=M$, then we have
\begin{equation} \label{e3.3}
lc(\mfk p_{M,M_0})=lc(\mfk p_{M,M_0'}).
\end{equation}
\end{proposition}
\begin{proof}
Note that the filtration generated by $M_0$ is equivalent to the filtration generated by $M_0'$, that is $M_k \subset M_{k+d}'$ and $M_{k}' \subset M_{k+d}$ for some $d\in \mbb N$ and each $k\in \mbb N$.
According to (\ref{lemd3.1}), we have
\begin{equation}
\mfk p_{M,M_0} (k)\leq \mfk p_{M,M_d'} (k) ,\,\, \mfk p_{M,M_0'} (k) \leq  p_{M,M_d} (k)
\end{equation}
for $k $ large enough.
Then $lc(\mfk p_{M,M_0})\leq lc(\mfk p_{M,M_d'})\leq lc(\mfk p_{M,M_{2d}})$ and (\ref{e3.3}) holds.

\end{proof}

Hence we can denote $lc(M)=lc(\mfk p_{M,M_0})$.

In this section, we will compute $lc(\mfk p_{M,M_0})$ when $ \mfk g=sl(n)$ and $ M={\msr H}_{\la\ell_1,\ell_2\ra} $.
Note that $\deg(\mfk p_{M,M_0})=\deg(\mfk p_{M })$  and $\deg(\msr V(M))=lc(\mfk p_{M })\cdot (\deg(\mfk p_{M }))!$, but we will see that $lc(M)=lc(\mfk p_{M })$ and $lc(M)\cdot (\deg(\mfk p_{M }))!=\deg(\msr V(M))$ does not always hold true.

In the rest of this section, we assume that $\mfk g=sl(n)$, $M=\msr H_{\la \ell_1,\ell_2 \ra}$ under the conditions stated in Lemma \ref{lemc2.1}, $M_0=V_0$ (cf. (\ref{a2.11})), and $M_k=U_k(\mfk g)V_0$ for $k\in \mbb N$. Then $lc(M ) =lc(\mfk p_{M,V_0} )$ by the proposition \ref{e3.3}.

\subsection{Case: $n_1<n_2=n$}

Since $M=0$ if $\ell_2>0$ and $n_1<n_2=n$, we assume that $\ell_2\geq 0$ in this subsection.
Recall that when $\ell_1,\ell_2>0$ and $n_1<n_2=n$, we have
\begin{equation}
V_0
 = \text{Span} \left\{  TN\begin{pmatrix}
0 & \ell_1 & 0 \\
k_{21} & k_{22} & 0
\end{pmatrix}
\Big\vert
k_{21}+k_{22}=\ell_2
\right\}.
\end{equation}
and
\begin{eqnarray}
M_k =U_k(\mfk g)(V_0)
=\text{Span} \left\{
TN\begin{pmatrix}
t & \ell_1+t & 0 \\
k_{21} & k_{22} & 0
\end{pmatrix} \Big\vert 0\leq t \leq k,k_{21}+k_{22}=\ell_2
\right\}
\end{eqnarray}
according to Lemma 4.4 in \cite{ZX}.

When $n_1+1<n_2$, we have
\begin{eqnarray}
& &\mfk p_{M,V_0}(k)\nonumber\\&=& \left\vert \left\{   N\begin{pmatrix}
k & \ell_1+k & 0 \\
k_{21} & k_{22} & 0
\end{pmatrix}
\Big\vert
k_{21}+k_{22}=\ell_2
\right\} \right\vert \nonumber\\
&=&
\binom{n +\ell_2-1}{\ell_2}\binom{n_1+k-1}{k}\binom{n -n_1-1+\ell_1+k-1}{\ell_1+k}\nonumber \\
& &
+\binom{n -1+\ell_2-1}{\ell_2}\binom{n_1+k-1}{k}\binom{n -n_1 +\ell_1+k-1}{\ell_1+k}\nonumber \\
& &
-\binom{n -1+\ell_2-1}{\ell_2}\binom{n_1+k-1}{k}\binom{n -n_1-1 +\ell_1+k-1}{\ell_1+k}
\nonumber  \\
&=&\left(
1+\frac{\ell_1}{n -1}+\frac{\ell_1+k}{n -n_1-1}
\right)
\binom{n  +\ell_2-2}{\ell_2}\binom{n_1+k-1}{k}\binom{n -n_1  +\ell_1+k-2}{\ell_1+k}\qquad \quad
\end{eqnarray}
Then
\begin{eqnarray}
lc (\mfk p_{M,V_0} )=  \frac{1}{(n_1-1)!(n -n_1-1)!}
\binom{n  +\ell_2-2}{\ell_2} .
\end{eqnarray}

When $n_1+1=n_2=n$, we have
\begin{eqnarray}
 \mfk p_{M,V_0}(k) &=&
\left\vert \left\{
x^\alpha y^\beta \in \msr A_{\la\ell_1,\ell_2\ra}  \Big\vert  \sum_{t=1}^{n-1} \alpha_t  = k, \sum_{s=1}^{n-1} \beta_s =\ell_2  \right\}
\right\vert
 \nonumber\\
&=&
\binom {n +k -2}{ k }\binom {n +\ell_2 -2}{\ell_2 },
\end{eqnarray}
and
\begin{equation}
lc(\mfk p_M)= \frac{1}{(n -2)!} \binom {n +\ell_2 -2}{\ell_2 }.
\end{equation}

Note that $\msr V(M)\cong \msr{V}_2(\ol{n_1+1,n} ,\ol{1,n_1}) $, then
\begin{equation}
lc(\mfk p_M) =\frac{\deg(\msr V(M))}{(\deg(\mfk p_M))!}=
\begin{cases}
\frac{1}{(n_1-1)!(n -n_1-1)!}    & \text{if } n_1>1 \text{ and } n_1+1<n ;  \pse\\
\frac{1}{(n-2)!}     & \text{if }  n_1=1 \text{ or } n_1+1=n .
\end{cases}
\end{equation}
by (\ref{c2.10}).

Similarly, when $n_1<n_2=n_2$, $\ell_1 \leq 0$ and $\ell_2\geq 0$, we have
\begin{eqnarray}
M_k =U_k(\mfk g)(V_0)
=\text{Span} \left\{
TN\begin{pmatrix}
-\ell_1+t &  t & 0 \\
k_{21} & k_{22} & 0
\end{pmatrix} \Big\vert 0\leq t \leq k,k_{21}+k_{22}=\ell_2
\right\}
\end{eqnarray}
and
\begin{eqnarray}
lc (\mfk p_{M,V_0} )=  \frac{1}{(n_1-1)!(n -n_1-1)!}
\binom{n  +\ell_2-2}{\ell_2} .
\end{eqnarray}

Therefore, we draw the following conclusion:
\begin{lemma}\pse\label{leme4.1}
For $\mfk g=sl(n)$ and infinite-dimensional $\mfk g$-module $M={\msr H}_{\la\ell_1,\ell_2\ra} $, if $n_1<n_2=n$, we have
\begin{equation}
\frac{lc(M)}{lc(\mfk p_{M })}
=\binom {n +\ell_2 -2}{\ell_2 }.
\end{equation}
\end{lemma}

\subsection{Case: $n_1=n_2<n$ and $\ell_1,\ell_2 \leq 0$}

Recall that
\begin{equation}V_0 =   \mbox{Span}\Big\{ \prod_{i\in J_1,j\in J_3}x_{i}^{\al_i}y_{j}^{\be_j}
\vert \sum_{i\in J_1}\al_{i} =-\ell_1,\sum_{j\in J_3}\be_{j} =-\ell_2 \Big\},  \end{equation}
and according to (4.157) in \cite{ZX}, we have
\begin{equation}M_k =   \mbox{Span}\Big\{\prod_{i\in J_1,j\in J_3}x_{i}^{\al_i}y_{j}^{\be_j}\phi(h) \mid
\sum_{i\in J_1}\al_{i} =-\ell_1,\sum_{j\in J_3}\be_{j} =-\ell_2 , h \in Z^{(r)}, r\in \ol{0,k}  \Big\}  . \end{equation}

Denote the set
\begin{eqnarray}S_k &=  & \Big\{ f=\prod_{i\in J_1,j\in J_3}x_{i}^{\al_i}y_{j}^{\be_j}h \big\vert
h \in Z^{(k)},
\sum_{i\in J_1}\al_{i} =-\ell_1,\sum_{j\in J_3}\be_{j} =-\ell_2 ,\nonumber \\
& &\quad\qquad \text{ no 3-chains in }\msr I(f) \Big\}  . \end{eqnarray}
Then $\phi(S_k)$ is linearly independent and $T(\phi(S_k)) \oplus M_{k-1}=M_k$ by Lemma 3.3 in \cite{ZX}.

 Suppose that $n_1>1$ and $n-n_1>1$.
For $f\in S_k$, we denote $i_1(f)=\max \{i \in J_1\vert  \al_i>0 \})$ and $j_1(f)=\min(\{j \in J_3\vert  \be_j>0 \} $.
For $d,i_1\in \mbb N$, we denote
\begin{eqnarray}\label{e4.14}
&  & G_0(d,i_1,\ell_1,\ell_2) \nonumber\\
&=&
\left\{
    \begin{aligned}
    &  \binom {d  -\ell_2 -2}{-\ell_2-1}
\binom {i_1-\ell_1-2}{-\ell_1-1} , &\text{ if }&\ell_1,\ell_2< 0; \\
    &    \binom {d -\ell_2 -2}{-\ell_2-1}  , &\text{ if }&\ell_1=0,\ell_2< 0; \\
    &    \binom {i_1-\ell_1-2}{-\ell_1-1}  , &\text{ if }&\ell_1<0,\ell_2 = 0; \\
    &  1,  &\text{ if }&\ell_1=\ell_2 = 0.
    \end{aligned}
    \right.
\end{eqnarray}
Then for $i_1\in J_1$ and $j_1\in J_3$, we have
\begin{eqnarray}
& &|\{g=x^\al y^\be |  \max \{i \in J_1\vert  \al_i>0 \} =i_1,\min(\{j \in J_3\vert  \be_j>0 \})=j_1\}|
\nonumber \\
&=&
 G_0(n-j_1+1,i_1,\ell_1,\ell_2)
 \end{eqnarray}
and
\begin{eqnarray}
& &|\{ h\in Z^{(k)}| f=gh\in S_k \text{ for some } g=x^\al y^\be,i_1(f)=i_1, j_1(f)= j_1\}|
\nonumber \\
&=&
 P_2(n-n_1,n_1,j_1-n_1,i_1,k).
 \end{eqnarray}

Hence
\begin{eqnarray}\label{e4.17}
& &\mfk p_{M,V_0}(k) \nonumber\\
&=&\dim(M_k/M_{k-1})=|S_k|  \nonumber\\
&=&
\begin{cases}
\sum\limits_{i_1=1}\limits^{n_1}\sum\limits_{j_1=n_1+1}\limits^{n}
G_0(n-j_1+1,i_1,\ell_1,\ell_2) P_2(n-n_1,n_1,j_1-n_1,i_1,k ) \text{ if } \ell_1,\ell_2<0 ;  \pse\\
\sum\limits_{j_1=n_1+1}\limits^{n} G_0(n-j_1+1,1,\ell_1,\ell_2) P_1(n_1,n-n_1,n-j_1+1 , k)  \text{ if } \ell_1=0,\ell_2< 0 ;  \pse\\
\sum\limits_{i_1=1}\limits^{n_1} G_0(n,i_1,\ell_1,\ell_2) P_1(n-n_1,n_1, i_1,k) \text{ if } \ell_1<0,\ell_2 = 0;  \pse\\
h_3(n-n_1,n_1,k)    \text{ if } \ell_1=0,\ell_2= 0 .
\end{cases}
\end{eqnarray}

According to the corollary  \ref{core2.1} and the corollary \ref{core2.2}, when $1<n_1<n-1$,
\begin{eqnarray}
\deg_k\left( P_2(n-n_1,n_1,j_1-n_1,i_1,k )
 \right)=
\begin{cases}
2n-5  &\text{ if }  i_1=1,j_1=n ;  \pse\\
 n +j_1 -i_1-3 & \text{ else}.
\end{cases}
 \end{eqnarray}
and
\begin{eqnarray}\label{e4.19}
\deg_k\left( P_1(n-n_1,n_1, i_1,k )
 \right)= \begin{cases}
2n-5  &\text{ if }  i_1=1 ;  \pse\\
 2n-i_1-3 & \text{ else}.
\end{cases}
 \end{eqnarray}
Hence when $\ell_1,\ell_2<0$ and $|J_1|,|J_3| >1$,
\begin{eqnarray}
& &lc(M)=lc( \mfk p_{M,V_0})\nonumber \\
&=&lc\left(\sum_{(i_1,j_1)\in \{(1,n),(1,n-1),(2,n)\}}G_0(n-j_1+1,i_1,\ell_1,\ell_2) P_2(n-n_1,n_1,j_1-n_1,i_1,k ) \right)
\nonumber \\
&=&
(1 -\ell_1 -\ell_2 )lc_k(P_2(n-n_1,n_1,n-n_1,1,k ) )=(1 -\ell_1 -\ell_2 )lc(\mfk p_M).
 \end{eqnarray}

Similarly, we also have $ lc( \mfk p_{M,V_0})=(1-\ell_1) lc(\mfk p_M)  $ when $\ell_2=0$ and $ lc( \mfk p_{M,V_0})=(1-\ell_2)lc(\mfk p_M) $ when $\ell_1=0$ from (\ref{e4.14}),(\ref{e4.17}) and (\ref{e4.19}).

When $n_1=1$, we have $\deg(\msr V(M))=1$ and
\begin{eqnarray}
\mfk p_{M,V_0}(k)
&=& \left|\left\{ x_1^{-\ell_1} y^\be h \Big\vert \be_1=0  ,\sum_{j\in J_3}\be_j =-\ell_2, h \in Z^{(k)}  \right\}\right|
\nonumber \\
&=& \binom{  n + k-2}{k} \binom{  n  -\ell_2-2}{-\ell_2}.
\end{eqnarray}
Hence,
\begin{equation}
lc(\mfk p_{M,V_0}(k)) = \frac{1}{(n-2)!} \binom{  n  -\ell_2-2}{-\ell_2}.
\end{equation}
Similarly, when $n-n_2=1$,
\begin{equation}
lc(\mfk p_{M,V_0}(k)) = \frac{1}{(n-2)!} \binom{  n  -\ell_1-2}{-\ell_1}.
\end{equation}

Therefore, we draw the following conclusion.
\begin{lemma}\label{leme4.2}
\pse
If $\ell_1,\ell_2 \leq 0$ and $n_1=n_2<n$, we have
\begin{equation}
\frac{lc(M)}{lc(\mfk p_{M })}
=
\begin{cases}
1-\ell_1-\ell_2 &\text{ if }  1<n_1<n-1 ;  \pse\\
\binom{  n  -\ell_2-2}{-\ell_2}  &\text{ if }  n_1=1 ;  \pse\\
\binom{  n  -\ell_1-2}{-\ell_1}  &\text{ if } n_1=n-1.
\end{cases}
\end{equation}
\end{lemma}

\subsection{Case: $n_1<n_2<n $ and $\ell_1\leq 0$ or $\ell_2 \leq 0$}

For $n_1,n_2 ,n \in \mathbb{N}$ with $n_1< n_2\leq n$,  we define an associative algebra isomorphism $\Psi: \msr A \rta  \msr A$ by
\begin{eqnarray}\label{a2.25}
& &\Psi(x_i)=y_{n-i+1},\;\;\Psi(y_i)=x_{n-i+1},\Psi(x_j)=y_{n-j+1},\;\;\Psi(y_j)=x_{n-j+1}, \qquad j\in J_3,i\in J_1;\nonumber \\
& &\Psi(x_{n_1+s})=y_{n-n_2+s},\;\; \Psi(y_{n_1+s})=x_{n-n_2+s}, \qquad s\in \ol{1,n_2-n_1}.
\end{eqnarray}
According to (\ref{a1.7}), for $f\in {\msr H}_{\la\ell_1,\ell_2\ra}^{n_1,n_2,n}$, we have $\Psi(f)\in {\msr A}_{\la\ell_2,\ell_1\ra}^{n-n_2,n-n_1,n}$,
and
\begin{eqnarray}& &\td\Dlt^{n-n_2,n-n_1,n}(\Psi(f)) \nonumber\\
=& &\sum_{i=1}^{n-n_2}x_i\ptl_{y_i}(\Psi(f))-\sum_{r=n-n_2+1}^{n-n_1}\ptl_{x_r}\ptl_{y_r}(\Psi(f))+\sum_{s=n-n_1+1}^n
y_s\ptl_{x_s}(\Psi(f)) \nonumber\\
=& &\sum_{i=1}^{n-n_2}\Psi(y_{n-i+1}\ptl_{x_{n-i+1}}(f))-\sum_{r=n_1+1}^{n_2}\Psi(\ptl_{x_r}\ptl_{y_r}(f))+\sum_{s=n-n_1+1}^n\Psi(
x_{n-s}\ptl_{y_{n-s}}(f)) \nonumber\\
=& &\Psi(\td\Dlt^{n_1,n_2,n}(f))=0.
\end{eqnarray}
Thus $\Psi: {\msr H}_{\la\ell_1,\ell_2\ra}^{n_1,n_2,n} \rta {\msr H}_{\la\ell_2,\ell_1\ra}^{n-n_2,n-n_1,n}$ is a linear space isomorphism.
Moreover, according to (\ref{a2.11}),(\ref{a2.22}) and (\ref{a2.24}), we have
\begin{equation}\label{a2.27}
\dim\left(({\msr H}_{\la\ell_1,\ell_2\ra}^{n_1,n_2,n})_k \right)=
\dim\left(\Psi\left({\msr H}_{\la\ell_1,\ell_2\ra}^{n_1,n_2,n})_k \right)\right)=
\dim\left( ({\msr H}_{\la\ell_2,\ell_1\ra}^{n-n_2,n-n_1,n} )_k\right)
\end{equation}
for $k\geq 0$.

Therefore, in the rest of this section, we suppose that $\ell_2 \leq 0$.

Let $\{(j_1,i_1),(j_2,i_2),...,(j_r,i_r)\}\subset \mbb Z^2$.
For $k_1,k_2,k_3,k_4\in \mbb N$, $\{s_{1,t_1} \in J_1 \vert  1 \leq t_1\leq k_1\}$, $\{s_{2,t_2} \in J_3 \vert  1 \leq t_2\leq k_2\}$, $\{s_{3,t_3} \in J_3 \vert  1 \leq t_3\leq k_3\}$, $\{(j_{t_4},i_{t_4})
 \in J_3 \times J_1 \vert  1 \leq t_4\leq k_4\}$ and $g\in \mathbb{F}[x_i,y_j\vert i\in J_2,j\in J_1\cup J_2]$,
\begin{eqnarray}& &\msr I\left[\left(\prod_{t_1=1}^{k_1} x_{s_{1,t_1}}\right)  \left(\prod_{t_2=1}^{k_2} y_{s_{2,t_2}}\right)   \left(\prod_{t_3=1}^{k_3} x_{s_{3,t_3}}\right)  \left( \prod_{t_4=1}^{k_4} z_{j_{t_4},i_{t_4}}\right)g\right]\nonumber\\ &=&
\{(n+1,s_{1,t_1}),(s_{2,t_2},0),(s_{3,t_3},-1),(j_{t_4},i_{t_4})
\vert  1 \leq  t_1 \leq  k_1,\nonumber\\
& &1 \leq  t_2 \leq  k_2,1 \leq  t_3 \leq  k_3,1 \leq  t_4 \leq  k_4 \},\qquad\end{eqnarray}

We set
\begin{eqnarray}\label{a3.2}
S_k=S^{\la \ell_1,\ell_2 \ra}_k=\Big\{\;g h  & \vert&  g=\prod_{i\in \ol{1,n} }\prod_{ j\in \ol{n_1+1,n}} x_i^{\al_i} y_j^{\be_j}  \in N(k-r),h\in Z^{(r)}; \nonumber \\
& &\text{no 3-chains in $\msr I(gh)$};r=0,1,\cdots k\Big\},
\end{eqnarray}
and
\begin{eqnarray}
W_k=\text{Span}\{ \phi( S_k)\}.
\end{eqnarray}
Then $T(W_k)$ is a subspace of $M_k$.

Before proving $T(W_k)$ is actually a basis of $M_k/M_{k-1}$, we need some preparation, which have been proved in \cite{ZX}.

 (1) The Lemma 3.1 in \cite{ZX} is shown below..
\begin{lemma}\label{lem3.1}
Fix $k,k_{13},k_{21}\in\mbb N$. For each $j$, $i_{1,j} \in J_1$, $i_{3,j} \in J_3$, $v_0\in \mathbb{F}\left[X_{J_1},X_{J_2 \setminus \{n_1+1\}},Y_{J_3} \right]\\ \cup  \mathbb{F}\left[X_{J_1},Y_{J_2 \setminus \{n_1+1\}},Y_{J_3} \right]$, we have
\begin{eqnarray}
& &
\left( -1\right)^{k} \frac{k! }{\left( k- k_{13}\right) !} \  T\left( v_0 \cdot
\left( \prod_{j=1}^{k}  x_{i_{1,j}}\right)
\left( \prod_{j=1}^{k_{13}}  x_{i_{3,j}}\right) \cdot  x_{n_1+1}^{k-k_{13}}
 \right)\nonumber\\& =&
 E_{i_{3,k_{13}},n_1+1} \cdots  E_{i_{3,1},n_1+1}  E_{n_1+1,i_{1,1}}  \cdots E_{n_1+1,i_{1,k}}
 \left( v_0 \right)
\end{eqnarray}
when $k \geq k_{13}$,
and
\begin{eqnarray}
& &\left( -1\right)^{k} \frac{k! }{\left( k- k_{21}\right) !} \  T\left(v_0 \cdot
\left( \prod_{j=1}^{k_{21}}  y_{i_{1,j}}\right)
\left( \prod_{j=1}^{k }  y_{i_{3,j}}\right) \cdot  y_{n_1+1}^{k-k_{21}}
 \right) \nonumber\\&=&
 E_{n_1+1,i_{1,1}}  \cdots E_{n_1+1,i_{1,k_{21}}}  E_{i_{3,1},n_1+1} \cdots  E_{i_{3,k},n_1+1}
 \left( v_0 \right)
\end{eqnarray}
if $k \geq k_{21}$.
\end{lemma}

(2) The Corollary 2.1 in \cite{ZX} is shown below..
\begin{lemma}\label{lem3.2}
Fix $k,\alpha_{n_1+1},\beta_{n_1+1} \in\mbb N$.
For $j$, $i_{1,j} \in J_1$, $i_{3,j} \in J_3$, $v_1\in \mathbb{F}[X_{J_1},X_{J_2 \setminus \{n_1+1\}},\\ Y_{J_3}] \cup  \mathbb{F}\left[X_{J_1},Y_{J_2 \setminus \{n_1+1\}},Y_{J_3} \right]$, we have
\begin{align}\label{a3.7}
\begin{split}
 \frac{k! }{ \alpha_{n_1+1}  !} \  T\left( v_1
\left( \prod_{j=1}^{k- \alpha_{n_1+1} }  x_{i_{3,j}}\right) \cdot  x_{n_1+1}^{\alpha_{n_1+1}}
 \right) =
 E_{i_{3,k-\al_{n_1+1}},n_1+1} \cdots  E_{i_{3,1},n_1+1}
 \left( v_1 x_{n_1+1}^{k} \right) ,
\end{split}
\end{align}
when $k \geq \alpha_{n_1+1}$,
and
\begin{align}
\begin{split}
 \frac{k! }{ \beta_{n_1+1}  !} \  T\left(v_1
\left( \prod_{j=1}^{k-\beta_{n_1+1} }  y_{i_{1,j}}\right) \cdot  y_{n_1+1}^{\beta_{n_1+1}}
 \right) =
 E_{n_1+1,i_{1,1}}  \cdots E_{n_1+1,i_{1,k-\be_{n_1+1}}}
 \left( v_1 y_{n_1+1}^{k } \right) ,
\end{split}
\end{align}
where $k \geq \beta_{n_1+1}$.
\end{lemma}

(3)
For $f \in M_k $ (cf. (\ref{a2.22})), we define the P-{\it order} of $f$ in $M_k$ as
\begin{align}\label{a4.1}
\mbox{ord}_k \left( f \right)
=\min \left\lbrace
s \in \mathbb{N} | f \in \mbox{Span}\{ TN_k, TN(k-r) P^r |
r=0,1,2,\cdots , s \}
\right\rbrace  .
\end{align}
In particular, $\mbox{ord}_k \left(f \right)=0  \iff  f\in TN_k$.

The Lemma 4.1 in  \cite{ZX} is shown below.
\begin{lemma}\label{lem3.3}
For $f \in M_{k}$, we have
\begin{align}
{\mfk d} \left( f \right)  \leq k+ \mbox{ord}_k\left(f \right) .
\end{align}
Moreover, ${\mfk d} \left( f \right) = k+ \mbox{ord}_k\left(f \right) $ if and only if $f \notin TN_{k-1}$.
\end{lemma}

(4)
Denote
\begin{equation}\label{a3.11}
J_3'=\ol{n_2+1,n+1},\quad J'_1=\ol{-1,n_1}
\end{equation}Set
 \begin{equation}\label{a3.17}
z_{n+1,0}=0,\;z_{n+1,i}=x_i,\;z_{j,0}=y_j\qquad\for\;\;i\in J_1,\;j \in J_3.
 \end{equation}
Then
\begin{equation}\msr C =\mbb F[z_{j,i}\mid j\in \ol{n_2+1,n+1},\;i\in \ol{0,n_1},\;(j,i)\neq (n+1,0)]\end{equation}
and
\begin{equation}
\msr R_3=\left\la \left|\begin{array}{ccc}
z_{j_1,i_1}  & z_{j_1,i_2}& z_{j_1,i_3}\\
z_{j_2,i_1} & z_{j_2,i_2}& z_{j_2,i_3}\\
z_{j_3,i_1} & z_{j_3,i_2}& z_{j_3,i_3}
\end{array}\right|\;\mid\;i_1,i_2,i_3\in J_1',\;j_1,j_2,j_3\in J_3'\right\ra\end{equation}
be an ideal of $\msr C$.

Then Lemma 3.3 in \cite{ZX} is shown below.
\begin{lemma}
\label{lem:a3.4}
We have $\ker \phi |_{\msr C }=\msr R_3$.\end{lemma}

(5) The following lemma is actually a corollary of Lemma 3.1 in \cite{ZX}.
\begin{lemma}
\label{lem:a5.2}
For $c_1,c_2,n_1,n \in \mbb N$, $c_1,c_2>0$ and $1<n_1<n$, we define an associative algebra homomorphism
$\varphi_{x,y}: {\msr A}_{c_1,c_2}\rta  \msr A $ by
\begin{equation}
\varphi_{x,y}(x_i)=y_i,\;\varphi_{x,y}(x_{i'})=x_{i'},\;\varphi_{x,y}(y_j)=y_j,\;\;\;\for \;
i\in \ol{1,n_1},i'\in \ol{1},j\in \ol{1,n}.
\end{equation}
Then
\begin{equation}
\ker \varphi_{x,y}=  \left\la x_i y_j-x_j y_i \vert i,j \in \ol{1,n}
\right\ra \cap {\msr A}_{c_1,c_2} .
\end{equation}
\end{lemma}
\begin{proof}
Suppose that $f=f(x_1,\cdots,x_n,y_1,\cdots,y_n) \in \ker \varphi_{x,y}$.
Write
\begin{equation}
f=\sum_{i=1}^m f_i ,\;\;f_i=g_{i }h_{i }f'_i
\end{equation}
where $f_i $ are monomials in $\msr A$, $g_i$ are monomials in $\mbb F[x_1,\cdots, x_{n_1}]$,  $h_i$ are monomials in $\mbb F[y_1,\cdots ,y_{n_1}]$ and $f'_i \in \mbb F[x_j,y_j \vert n_1<j\leq n]$.
Denote the distinct elements in $\{f_i' \vert  1\leq i \leq m\}$ by $\{F_1,\cdots,F_{m'}\}$.
 Then
\begin{equation}
f=\sum_{j=1}^{m'}\sum_{\{i: f'_i=F_j\}}f_i
\end{equation}
Since $\varphi_{x,y} (f)=0$, we have
\begin{equation}
\sum_{\{i: f'_i=F_j\}}\varphi_{x,y}( f_i )=
\varphi_{x,y}\left( \sum_{\{i: f'_i=F_j\}}   g_i h_i\right) F_j =0
\end{equation}
 for $j\in \ol{1,m'}$.
Hence we may assume that $f_i'=1$ for $i \in \ol{1,n}$.

Denote $U=\{u_{1,1},\cdots,u_{1,n_1},u_{2,1},\cdots,u_{2,n_1}\}$ as a set of $2n_1$ variables and $\{v_1,v_2\}$ as a set of $2$ variables.
Define associative algebra homomorphisms $\vf_u: \mbb F[U]\rta \mbb F[ v_1,v_2,y_i\vert i\in \ol{1,n_1}]$ by
\begin{equation}
\vf_u(u_{1,j})=v_1 y_j ,\;\vf_u(u_{2,j})=v_2 y_j  \qquad\for\;\; j\in \ol{1,n_1}.\end{equation}

Then we have
\begin{eqnarray}
& &v_1^{c_1}v_2^{c_2} \vf_{x,y} (f(x_1,\cdots,x_{n_1},y_1,\cdots,y_{n_1}))
\nonumber\\& =&
v_1^{c_1}v_2^{c_2} \vf_{x,y} \left(\sum_{i=1}^{m }g(x_1,\cdots,x_{n_1})h(y_1,\cdots,y_{n_1})\right)
\nonumber\\& =&
v_1^{c_1}v_2^{c_2}   \left(\sum_{i=1}^{m }g(y_1,\cdots,y_{n_1})h(y_1,\cdots,y_{n_1})\right)
\nonumber\\& =&
 \sum_{i=1}^{m }g(v_1y_1,\cdots,v_1y_{n_1})h(v_2y_1,\cdots,v_2y_{n_1})
\nonumber\\& =&
  \vf_u \left(\sum_{i=1}^{m }g(u_{1,1},\cdots,u_{1,{n_1}})h(u_{2,1},\cdots,u_{2,{n_1}})\right)
\nonumber\\& =&f(u_{1,1},\cdots,u_{1,{n_1}},u_{2,1},\cdots,u_{2,{n_1}})
\end{eqnarray}
According to Lemma 3.1 in \cite{ZX},
\begin{equation}
f \in \left\la x_i y_j-x_j y_i \vert i,j \in \ol{1,n}
\right\ra \cap {\msr A}_{c_1,c_2} .\end{equation}

\end{proof}

\begin{proposition} \label{prop1}
For $k\in \mbb N$, we have $T(W_k) \oplus M_{k-1}=M_k$ and $\phi(S_k)$ is a linearly independent set.
\end{proposition}

\begin{proof}
Denote
\begin{eqnarray}
\mathbf{A} &=&
 \begin{pmatrix}
z_{n_2+1,n_1} & \cdots &z_{n ,n_1} &z_{n +1,n_1}   \\
z_{n_2+1,n_1-1} & \cdots &z_{n ,n_1-1} &z_{n +1,n_1-1}   \\
\vdots  & \vdots  & \vdots  & \vdots    \\
z_{n_2+1,1} & \cdots &z_{n ,1} &z_{n +1,1}  \\
z_{n_2+1,0} & \cdots &z_{n ,0} &z_{n +1,0}  \\
x_{n_2+1} & \cdots &x_n & -1
\end{pmatrix}\nonumber \\
&=&
 \begin{pmatrix}
z_{n_2+1,n_1} & \cdots &z_{n ,n_1} &x_{n_1}   \\
z_{n_2+1,n_1-1} & \cdots &z_{n ,n_1-1} &x_{ n_1-1}   \\
\vdots  & \vdots  & \vdots  & \vdots    \\
z_{n_2+1,1} & \cdots &z_{n ,1} &x_{ 1}  \\
y_{n_2+1 } & \cdots &y_{n  } & 0  \\
x_{n_2+1} & \cdots &x_n & -1
\end{pmatrix}.
\end{eqnarray}
It is easy to verify that the 3-minors of $\phi(\mathbf{A})$ are all zero. Using the method in the proof of Proposition \ref{propf2.2} (or the proof of Lemma 3.3 in \cite{ZX}), we can show that $T(W_k) + M_{k-1} = M_k$.

Next, we will prove that $M_{k-1} \cap T(W_k) =\emptyset$ and $ \phi(S_k)$ is a basis of $W_k$.

We assume that exists $f \in \text{Span}\{S_k\}$ and $T(\phi(f)) \in M_{k-1}$.
Write \begin{eqnarray}\label{a3.39}
f=\sum_{r=0}^{r_0}  \sum_{i=1}^{N_r}a_{i,r} g_{i,r}h_{i,r}
\end{eqnarray}
where $g_{i,r}h_{i,r} \in S_k$, $g_{i,r}= \prod_{i\in \ol{1,n} }\prod_{ j\in \ol{n_1+1,n}} x_i^{\al_i} y_j^{\be_j} \in N(k-r) $, $h_{i,r} \in Z^{(r)}$ and there are no 3-chains in $\msr I(g_{i,r}h_{i,r})$, for $i\in \ol{1,N_r}$, $r\in \ol{0,r_0}$, $r_0= \mbox{ord}_k(T\phi(f)) \in \ol{0,k}$.
For convenience, we assume that $g_{i,r}h_{i,r}$ are distinct for different $(i,r)$ and $a_{i,r}\neq 0$ for each $(i,r)$.

We may assme that $g_{i,r} \in \mathbb{F}[x_i,y_j \vert i,j\in \ol{1,n}\setminus J_2]$. Otherwise, we  set
\begin{eqnarray}
g_{i,r}=u_{i,r}v_{i,r}
\end{eqnarray}
where $u_{i,r} \in \mathbb{F}[x_i,y_j \vert i,j\in \ol{1,n}\setminus J_2]$ and $v_{i,r}\in \mathbb{F}[x_i,y_j \vert i,j\in J_2]$.
Since $T\phi(f) \in M_{k-1}$, we have
\begin{eqnarray}
\phi(f)=\sum_{r=0}^{r'_0}  \sum_{i=1}^{N'_r} a'_{i,r}g'_{i,r}\phi(h'_{i,r})=\sum_{r=0}^{r'_0}  \sum_{i=1}^{N'_r} a'_{i,r}u'_{i,r}v'_{i,r}\phi(h'_{i,r})
\end{eqnarray}
where $g'_{i,r} \in N(k-1-r)$, $h'_{i,r} \in Z^{(r)}$, $u'_{i,r} \in \mathbb{F}[x_i,y_j \vert i,j\in \ol{1,n}\setminus J_2]$ and $v'_{i,r}\in \mathbb{F}[x_i,y_j \vert i,j\in J_2]$ for $r'_0\in \ol{1,k-1}$.

Set $\msr T=\{(i,r) \vert v_{i,r}=v_{1,r_0}\}$ and $\msr T'=\{(i,r) \vert v'_{i,r}=v_{1,r_0}\}$, then
\begin{eqnarray}
\sum_{(i,r)\in  \msr T'} a'_{i,r}u'_{i,r}\phi(h'_{i,r})=\sum_{(i,r)\in \msr T}a_{i,r} u_{i,r}\phi(h_{i,r})
\end{eqnarray}
For $j=1,2,3$ and $  x^\al y^\be \in \msr A$, we denote
\begin{equation}
\deg_{x_{_{J_j}}}( x^\al y^\be )= \sum_{i\in J_j}\al_{i}, \qquad
\deg_{y_{_{J_j}}}( x^\al y^\be )= \sum_{i\in J_j}\be_{i}\;. \end{equation}

Denote $\deg_{x_{_{J_2}}}(v_{1,r_0}) =k_{12}$, $\deg_{y_{_{J_2}}}(v_{1,r_0})=k_{22}$, and we set
\begin{eqnarray}
\ol{f}=\sum_{(i,r)\in \msr T} a_{i,r}u_{i,r} h_{i,r} .
\end{eqnarray}
Then $T(\phi(\ol{f}))\in \msr H_{\la \ell_1-k_{12},\ell_2-k_{22} \ra}$,
\begin{eqnarray}
\mfk d_{\la \ell_1-k_{12},\ell_2-k_{22} \ra}^{n_1,n_2,n} (u_{i,r})=\mfk d_{\la \ell_1,\ell_2 \ra}^{n_1,n_2,n} (u_{i,r})-k_{12}-k_{22}
\end{eqnarray}
for $(i,r) \in \msr T$.

Moreover,
\begin{eqnarray}
T\left(
\sum_{(i,r)\in \msr T'} a'_{i,r} u'_{i,r}\phi(h'_{i,r})\right) \in (\msr H_{\la \ell_1-k_{12},\ell_2-k_{22} \ra})_{k-1-k_{12}-k_{22}} .
\end{eqnarray}
Note that $g_{i,r} h_{i,r}=g_{i',r'} h_{i',r'}$ can be deduced from $u_{i,r} h_{i,r}=u_{i',r'} h_{i',r'}$.
By induction on $k$, we have $ \ol{f} =0$ and $a_{i,r}=0$ for each $(i,r) \in \msr T$ if $k_{12}+k_{22}>0$.

Since $ \ol{f} =0$ is independent of the selection of $v_{1,r_0}$, we have $a_{i,r}=0$ for $i\in \ol{1,N_r}$ and $r\in \ol{0,r_0}$. In other words, we can prove that $f\in M_{k-1}+T(W_k)$ and $M_{k-1}+T(W_k) = M_k$.

Hence, we assme that $g_{i,r} \in \mathbb{F}[x_i,y_j \vert i,j\in \ol{1,n}\setminus J_2]$ in the following.

(a) Suppose that $\ell_1\leq 0$.

We set
\begin{equation}
g_{i,r}=\left( \prod_{j=1}^{k_{13}}  x_{s_{i,r,j}} x_{t_{i,r,j}}\right)v_{i,r}
\end{equation}
and
\begin{equation}
h_{i,r}= \prod_{j=1}^{r}  z_{p_{i,r,j},q_{i,r,j}} ,
\end{equation}
where $s_{i,r,j}\in J_3$, $t_{i,r,j} \in J_1$, $(p_{i,r,j},q_{i,r,j})\in J_3\times J_1$,$v_{i,r}  \in M_0 \cap \mathbb{F}[x_i,y_j|i\in J_1,j \in J_3]$ and $2k_{13}+r=k$.

According to Lemma \ref{lem3.1}, we have
\begin{eqnarray}
& &T(\phi(f)) \nonumber\\
&=&T\left( \sum_{r=0}^{r_0}   \sum_{i=1}^{N_r} a_{i,r}g_{i,r}\phi(h_{i,r})  \right)\nonumber\\
&=&T\left(
\sum_{r=0}^{r_0}  \sum_{i=1}^{N_r} \left(\prod_{j'=1}^{r}  \phi\left(z_{p_{i,r,j'},q_{i,r,j'}}   \right) \prod_{j=1}^{k_{13}}  x_{s_{i,r,j}} x_{t_{i,r,j}}\right)a_{i,r} v_{i,r}
\right)\nonumber\\
&=&\sum_{r=0}^{r_0}  \sum_{i=1}^{N_r} \left( \frac{(-1)^{k_{13}}}{k_{13}!}
\prod_{j'=1}^{r} E_{p_{i,r,j'},q_{i,r,j'}}
\left( \prod_{j=1}^{k_{13}} E_{s_{i,r,j},n_1+1} \right)
\left( \prod_{j=1}^{k_{13}} E_{n_1+1,t_{i,r,j}} \right)
a_{i,r} v_{i,r}
\right) \nonumber\\
&\in & M_{k-1}
\end{eqnarray}
We set
\begin{equation}\label{a3.53}
H=\sum_{r=0}^{r_0}  \sum_{i=1}^{N_r} \left(\phi\left(\prod_{j'=1}^{r}  z_{p_{i,r,j'},q_{i,r,j'}}   \right)\left( \prod_{j=1}^{k_{13}}  y_{s_{i,r,j}}y_{n_1+1} x_{t_{i,r,j}}x_{n_1+1}\right)
a_{i,r} v_{i,r} \right)
\end{equation}

Suppose that the set of distinct elements of $\{ r \vert \ \text{exists } i \text{ such that } a_{i,r}\neq 0  \}$ is $\{\jmath_1,\dots,\jmath_N\}$, and $\jmath_1<\jmath_2<\cdots<\jmath_N=r_0$. For $\iota\in \ol{1,N}$, we set
\begin{equation}
H_\iota=   \sum_{i=1}^{N_{\jmath_\iota}} \left(\phi\left(\prod_{j'=1}^{\jmath_\iota}  z_{p_{i,\jmath_\iota,j'},q_{i,\jmath_\iota,j'}}   \right)\left( \prod_{j=1}^{k_{13}}  y_{s_{i,\jmath_\iota,j}}y_{n_1+1} x_{t_{i,\jmath_\iota,j}}x_{n_1+1}\right)
a_{i,\jmath_\iota} v_{i,\jmath_\iota} \right)
\end{equation}

For $\iota' \in \mathbb{N}$ and $F=\sum_{i=1}^{m'} F_i\in \msr A$, where $F_i$ are distinct monomials in $\msr A$, we define
\begin{equation}
{\mfk r} (F,\iota')=\min \{ \deg_{x_{_{J_3}}}(F_i) + \deg_{y_{_{J_1}}} (F_i) \mid \mfk d (F_i)=\iota', i\in \ol{1,m'} \}.
\end{equation}
Then we have
\begin{equation} \label{a3.56}
{\mfk r} (T(\phi(f)),k+ \jmath_\iota) \leq {\mfk r}\left(  H_\iota , \mfk d (H_\iota) \right)=\jmath_\iota.
\end{equation}

Since $T(\phi(f))\in M_{k-1}$ and $T(\phi(f)) \notin TN_{k-1}$, according to Lemma \ref{lem3.3}, we have
\begin{equation}
{\mfk d} \left( T(\phi(f)) \right)= \mbox{ord}_{k-1}\left(T(\phi(f))\right) + k-1
\Rightarrow \mbox{ord}_{ k-1}\left(T(\phi(f))\right)  \geq 1+ \jmath_\iota.
\end{equation}
For $r= \mbox{ord}_{ k-1}\left( T(\phi(f)) \right)  \geq 1+ \jmath_\iota $ in this case, we can write
\begin{align}
 T(\phi(f))=\sum_{s_r=1}^{l_r} G_{r,s_r} G'_{r,s_r} +  \sum_{s_{r-1}=1}^{l_{r-1}} G_{r-1,s_{r-1}} G'_{r-1,s_{r-1}} + \cdots +G_{0},
\end{align}
where $ G_{i,s_i}  \in P^{\left(i \right)}$, $G'_{i,s_i}  \in TN(k-1-i)$ for $  s_i\in\ol{1,l_i} $, and $G_{0} \in TN_{k-1}$.
Note that for $  s_i\in\ol{1,l_i} $ and $i\in \ol{1,r}$,
\begin{equation}
{\mfk d} ( G_{i,s_i} G'_{i,s_i} ) = k-1-i+2i=k-1+i
\end{equation}
Since $\ol{h}_\iota\neq 0$ and ${\mfk d} (\ol{h}_\iota)=p+\jmath_\iota$, we have
\begin{equation}  \sum_{s_{\jmath_\iota+1}=1}^{l_{\jmath_\iota+1}} G_{\jmath_\iota+1,s_{\jmath_\iota+1}} G'_{\jmath_\iota+1,s_{\jmath_\iota+1}} \neq 0.\end{equation}
Therefore,
\begin{equation}
{\mfk r} ( T(\phi(f)),k+\jmath_\iota) = {\mfk r}\left( \sum_{s_{\jmath_\iota+1}=1}^{l_{\jmath_\iota+1}} G_{\jmath_\iota+1,s_{\jmath_\iota+1}} G'_{\jmath_\iota+1,s_{\jmath_\iota+1}} , k+ \jmath_\iota \right) \geq \jmath_\iota+1,
\end{equation}
which contradicts (\ref{a3.56}).
Hence $H_\iota=0$ for $w\in \ol{1,N}$.

In particular, $H_{N}=0$.
Then
\begin{eqnarray}
0=\sum_{i=1}^{N_{r_0}} \left(\phi\left(\prod_{j'=1}^{r_0}  z_{p_{i,r_0,j'},q_{i,r_0,j'}}   \right)\left( \prod_{j=1}^{k_{13}}  y_{s_{i,r_0,j}} x_{t_{i,r_0,j}} \right)
a_{i,r_0} v_{i,r_0} \right)
\end{eqnarray}

Define an associative algebra homomorphism
$\psi: \msr C\rta  \msr C $ by
\begin{eqnarray}& &\psi(z_{j,i})=z_{j,i},\;\;\psi(x_t)=x_t, \;\psi(x_j)=y_j,\;\psi(y_r)=y_r, \nonumber\\
& &\quad i\in J_1,\;j\in J_3,t \in \ol{1,n_2},r \in \ol{1,n}.\end{eqnarray}

For $i\in \ol{1,N_{r_0}}$, we set
\begin{eqnarray}
\ol{f_{i,r_0}}=  \left(\prod_{j'=1}^{r_0}  z_{p_{i,r_0,j'},q_{i,r_0,j'}} \right) \left( \prod_{j=1}^{k_{13}}  y_{s_{i,r_0,j}} x_{t_{i,r_0,j}} \right)
a_{i,r_0} v_{i,r_0}
\end{eqnarray}
Then
\begin{equation}\phi \left(\sum_{i=1}^{N_{r_0}}  \ol{f_{i,r_0}} \right)=0\end{equation}
and
\begin{equation}\psi \left( a_{i,r_0}   g_{i,r_0}  h_{i,r_0} \right)=\ol{f_{i,r_0}} .\end{equation}

Since there are no 3-chains in $\msr I(g_{i,r_0}h_{i,r_0})$ for each $i\in \ol{1,N_{r_0}}$, there are no 3-chains in $ \msr I( \ol{f_{i,r_0}})$ for each $i\in\ol{1,N_{r_0}}$.
Therefore,
\begin{equation} \sum_{i=1}^{N_{r_0}}  \ol{f_{i,r_0}} =0 \end{equation}
and
\begin{equation}  \label{a3.65}
\sum_{i=1}^{N_{r_0}} a_{i,r_0}   g_{i,r_0}  h_{i,r_0}  \in \ker \psi   .\end{equation}

For a monomial $x^\al y^\be \in \msr A$, we denote the degree in $\{x_{i_j} \mid   i_j \in J_j  \} $ as $\deg_{x_{_{J_j}}} \left( x^\al y^\be\right)$ for $j=1,2,3$.
Define $\mfk w(x^\al y^\be) =\deg_{x_{J_3}}(x^\al y^\be)+
\deg_{y_{J_3}}(x^\al y^\be)$.
For $F\in \ker \psi  $, we can decompose $F$ by different $\mfk w$:
\begin{equation}
F
=\sum_{\{s: \mfk w(F_{s,t})=w_s\}}\sum_{t} c_{s,t} F_{s,t},\end{equation}
where $F_{s,t}$ are monomials with $\mfk w(F_{s,t})=w_s$ and $c_{s,t} \in \mbb F$ for each $s,t$.

Thus $\psi (F_{s,t}) $ are also monomials and $\mfk w \left( \psi(F_{s,t}) \right)=w_s$ for each $s,t$. Therefore, for a given $s$, we have
\begin{eqnarray}\label{a3.68}
 \sum_{t} c_{s,t} F_{s,t}   \in
\la x_{t}y_{t'}-x_{t'}y_{t} \vert\; t,t'\in J_3\ra
\end{eqnarray}
by Lemma \ref{lem:a5.2}.
Hence,
\begin{equation}\label{e3.59}
\sum_{i=1}^{N_{r_0}} a_{i,r_0}   g_{i,r_0}  h_{i,r_0}  \in \la x_{t}y_{t'}-x_{t'}y_{t} \vert\; t, t'\in J_3\ra.
\end{equation}

There do not exist $i_3,j_3 \in J_3$ such that $i_3<j_3$ and $ x_{i_3}y_{j_3} \mid g_{i,r}$.
Otherwise, there exists $i_1\in J_1$ such that $ x_{i_1} \mid g_{i,r}$ because of $\ell_1 \leq 0$. Then
$(i_3,-1)\prec   (j_3,0) \prec (n+1,i_1)$ is a 3-chain in $ \msr I (g_{i,r}h_{i,r} )$, which contradicts (\ref{a3.2}).
Then $g_{i,r} \equiv g' \quad (\mbox{mod}\;  \la x_{t}y_{t'}-x_{t'}y_{t} \vert\; t\neq t'\in J_3\ra)$ forces $g'=g_{i,r} $.
According to (\ref{e3.59}), $\sum_{i=1}^{N_{r_0}} a_{i,r_0}   g_{i,r_0}  h_{i,r_0} =0$, whcih contradicts $a_{i,r}\neq 0$. Therefore $f=0$.

Since elements in $S_k$ are monomials, $\phi(S_k)$ is a linearly independent set.

(b) Consider $\ell_1 > 0$.

We set
\begin{equation}
g_{i,r}=\left( \prod_{j=1}^{k_{13}}  x_{s_{i,r,j}}  \right)v_{i,r}
\end{equation}
and
\begin{equation}
h_{i,r}= \prod_{j=1}^{r}  z_{p_{i,r,j},q_{i,r,j}} ,
\end{equation}
where $s_{i,r,j}\in J_3$, $(p_{i,r,j},q_{i,r,j})\in J_3\times J_1$,$v_{i,r}  \in   \mathbb{F}[x_i,y_j|i\in J_1,j \in J_3]$.
Moreover, we have $k_{13} \geq \ell_1$, $\deg_{x_{J_1}} (v_{i,r}) =k_{13}-\ell_1$, $k-r=2 k_{13} -\ell_1$,  $k\geq \ell_1$, $r\leq r_0 \leq k-\ell_1$.

According to Lemma \ref{lem3.2}, we have
\begin{eqnarray}
& &T(\phi(f)) \nonumber\\
&=&T\left( \sum_{r=0}^{r_0}   \sum_{i=1}^{N_r} a_{i,r}g_{i,r}\phi(h_{i,r})  \right)\nonumber\\
&=&T\left(
\sum_{r=0}^{r_0}  \sum_{i=1}^{N_r} \left(\prod_{j'=1}^{r}  \phi\left(z_{p_{i,r,j'},q_{i,r,j'}}   \right) \prod_{j=1}^{k_{13}}  x_{s_{i,r,j}}  \right)a_{i,r} v_{i,r}
\right)\nonumber\\
&=&\sum_{r=0}^{r_0}  \sum_{i=1}^{N_r} \left( \frac{a_{i,r}}{k_{13}!}
\prod_{j'=1}^{r} E_{p_{i,r,j'},q_{i,r,j'}}
\left( \prod_{j=1}^{k_{13}} E_{s_{i,r,j},n_1+1} \right)
(v_{i,r} x_{n_1+1}^{k_{13}}  )
\right) \nonumber\\
&\in & M_{k-1}
\end{eqnarray}
with $ v_{i,r} x_{n_1+1}^{k_{13}} \in M_{k_{13}-\ell_1} $.

By Lemma \ref{lem:a5.2}, for $r\in \ol{0,r_0}$, we also have
\begin{equation} \label{a3.83}
\sum_{i=1}^{N_{r }} a_{i,r }   g_{i,r }  h_{i,r_0}  \in \la x_{t}y_{t'}-x_{t'}y_{t} \vert\; t, t'\in J_3\ra .
\end{equation}
We may assume that
\begin{equation}
\sum_{i=1}^{N_{r }} a_{i,r }   g_{i,r }  h_{i,r }  =v_r+v'_r,
\end{equation}
where $\deg_{x_{J_1}}(v_r) =0$ and each monomial in $v'_r$ has a factor in $\{x_{i} \vert i\in J_1\}$.
According to (\ref{a3.83}), we may assume that
\begin{equation}
\sum_{i=1}^{N_{r }} a_{i,r }   g_{i,r }  h_{i,r }  =
\sum_{i}(x_{t_i}y_{t'_i}-x_{t'_i}y_{t_i}) u_i +
\sum_{j}(x_{s_j}y_{s'_j}-x_{s'_j}y_{s_j}) u'_j
\end{equation}
where $u_i,u'_j \in \msr C$ are monomials such that $\deg_{x_{J_1}}(u_i) = 0$ and $\deg_{x_{J_1}}(u'_j) >0$ for each $i,j$.

Then
\begin{equation}
v_r=\sum_{i}(x_{t_i}y_{t'_i}-x_{t'_i}y_{t_i}) u_i ,\;
v'_r=\sum_{j}(x_{s_j}y_{s'_j}-x_{s'_j}y_{s_j}) u'_j \in \ker \psi.
\end{equation}

Since there are no 3-chain in $\msr I(v'_0)$, we also have $v'_r=0$ just like what we discussed in the case of $\ell_1 \geq 0$. Then we have $f= \sum_{r=0}^{r_0} v_r$. In other words, $ g_{i,r } \in \mathbb{F}[x_i,y_i|i\in J_3]$ for $i\in \ol{1,N_r}$, $r\in \ol{0,r_0}$.

According to Lemma \ref{lem3.3}, we have
\begin{equation}
\mbox{ord}_{k-1}(T(\phi(f)))=\mfk d(T\phi(f)) -(k-1)=\mbox{ord}_{k }(T(\phi(f)))+1 =r_0+1.
\end{equation}

Thus we may assume that
\begin{align}
T(\phi(f))=\sum_{s_{r_0+1}=1}^{l_{r_0+1}}c_{r_0+1,s_{r_0+1}} G_{r_0+1,s_{r_0+1}} H_{r_0+1,s_{r_0+1}} +  \sum_{s_{r_0}=1}^{l_{r_0}}c_{r_0,s_{r_0}} G_{r_0,s_{r_0}} H_{r_0,s_{r_0}} + \cdots +G_{0}
\end{align}
with $c_{r,s_{r}} \in \mathbb{F}$, $G_{r,s_{r}} \in TN(k-1-r)$ and $H_{r,s_{r}} \in \phi( Z^{(r)} )$ for $r\in \ol{1,r_0+1}$, $s_r \in \ol{1,l_{r}}$ and $G_{0} \in N_{k-1}$.
Note that
\begin{equation}
\xi=\sum_{s_{r_0+1}=1}^{l_{r_0+1}}c_{r_0+1,s_{r_0+1}} G_{r_0+1,s_{r_0+1}} H_{r_0+1,s_{r_0+1}} \neq 0,
\end{equation}
otherwise $\mbox{ord}_{k-1}(T(\phi(f))) <r_0+1$.

For a polynomial $f\in \msr A$, we denote
\begin{equation}
\mfk m (f)=\deg_{x_{J_1}}(f)+\deg_{y_{J_1}}(f).
\end{equation}
Then $\mfk m\left( H_{r_0+1,s_{r_0+1}} \right) = r_0+1$ for $s_{r_0+1} \in \ol{1,l_{r_0+1}}$.
For each monomial $u\in \msr A$ in $\xi$, we have $\mfk m (u) =r_0+1$.
Hence we have
\begin{equation} \label{a3.91}
\mfk m\left( T\phi(f)\right) \geq \mfk m\left( \xi\right)=r_0+1.
\end{equation}

On the other hand, since $ g_{i,r } \in \mathbb{F}[x_i,y_i|i\in J_3]$ for $i\in \ol{1,N_r}$, $r\in \ol{0,r_0}$, we have
\begin{equation}
\mfk m\left(  T(g_{i,r}) \right) =0.
\end{equation}
Hence
\begin{eqnarray}
\mfk m\left( T\phi(f)\right)
& = & \mfk m\left(\sum_{r=0}^{r_0}  \sum_{i=1}^{N_r}a_{i,r} T(g_{i,r})\phi(h_{i,r}) \right)
\nonumber\\
& =& \max_{0\leq r\leq r_0,1\leq i\leq N_r} \left\{
\mfk m\left(   T(g_{i,r})\phi(h_{i,r}) \right)\right\}
\nonumber\\
& \leq& \max_{0\leq r\leq r_0,1\leq i\leq N_r}  \left\{ r \right\} \leq r_0,
\end{eqnarray}
which contradicts (\ref{a3.91}).
Therefore, $f=0$.
\end{proof}

\begin{corollary}\label{core4.1}
When $n_1<n_2<n$, $\ell_2 \leq 0$, we have $\mfk p_{M,V_0} (k) =|S_k|$.
\end{corollary}

Note that $ \mfk p_{M,V_0} (k)=\mfk p_{M }(k)$ when $\ell_1=\ell_2=0$ according to Proposition \ref{prope3.1}.

For $f =gh \in S^{\la 0,0 \ra}_k$, $g= x^{\al} y^{\be} \in N(k-r) $ with $\be_i=0$ for $i\in J_1$ ,$h  \in Z^{(r)}$, we write
\begin{equation}
g= \left( \prod_{s=1}^{k_{12}+k_{13}} x_{i_{1,s}}\right)\left( \prod_{s=1}^{k_{12} } x_{i_{2,s}}\right)
\left( \prod_{s=1}^{ k_{13}} x_{i_{3,s}}\right)\left( \prod_{t=1}^{ k_{22}} y_{j_{2,t}}\right)
\left( \prod_{t=1}^{ k_{22}} y_{j_{3,t}}\right)
\end{equation}
with $n_1\geq i_{1,1} \geq i_{1,2}\geq\cdots \geq i_{1,k_{12}+k_{13}} \geq 1$, $n_2 \geq i_{2,1} \geq i_{2,2}\geq\cdots \geq i_{2,k_{12 }} \geq n_1+1$, $n_2+1 \leq i_{3,1} \leq i_{3,2} \leq \cdots \leq i_{3,k_{13}}\leq n$, $n_2\geq j_{2,1} \geq j_{2,2}  \geq \cdots  \geq j_{2,k_{22 }} \geq n_1+1$, $n_2+1\leq j_{3,1} \leq j_{3,2} \leq \cdots \leq j_{3,k_{23}}\leq n$ and $2k_{13}+k_{22}+k_{12}=k-r$.
We define the following map $\Upsilon$ from $S^{\la 0,0 \ra}_k$ to $S_2^{\ol{n_2+1,n+n_2-n_1},\ol{ n_1-n_2+1, n_1}}(n ,1,k)$ by
\begin{equation}
\Upsilon(f)=  \left( \prod_{s=1}^{k_{12}} z_{n+1+n_2-i_{2,s},i_{1,s}}\right)
\left( \prod_{t=1}^{ k_{22}} z_{j_{3,t},j_{2,t}}\right)
\left( \prod_{s=1}^{ k_{13}} z_{n+n_2-n_1,i_{1,s+k_{12}}}z_{i_{3,s},n_1-n_2+1}\right) h.
\end{equation}
 
For $f \in S_2^{\ol{n_2+1,n+n_2-n_1},\ol{ n_1-n_2+1, n_1}}(n ,1,k)$, we write
\begin{equation}
f=\left(\prod_{q=1}^{k_3} z_{p_{3,q},p_{1,q}} \right)\left(\prod_{s=1}^{k_1} z_{i_{2,s},i_{1,s}} \right)\left(\prod_{t=1}^{k_2} z_{j_{3,t},j_{2,t}} \right) .
\end{equation}
with $p_{3,q}\in J_3$, $p_{1,q} \in J_1$, $n+1 \leq i_{2,k_2}\leq \cdots  \leq i_{2,2}  \leq i_{2,1}  \leq n+n_2-n_1 $, $i_{1,s} \in J_1$, $j_{3,t} \in J_3$, $n_1-n_2+1 \leq j_{2,1} \leq j_{2,2} \cdots \leq j_{2, k_2}\leq 0$ and $k_1+k_2+k_3=k$ for $q\in\ol{1,k_3}$, $s\in \ol{1,k_1}$, $t\in\ol{1,k_2}$.

We define $\Upsilon^{-1}$ from $S_2^{\ol{n_2+1,n+n_2-n_1},\ol{ n_1-n_2+1, n_1}}(n ,1,k)$ to $S^{\la 0,0 \ra}_k$  by
\begin{eqnarray}
& &\Upsilon^{-1} (f) \nonumber\\
&=& \left(\prod_{q=1}^{k_3} z_{p_{3,q},p_{1,q}} \right)\left(\prod_{s=k_0+1}^{k_1 } x_{n+n_2+1-i_{2,s}}x_{i_{1,s}} \right)\left(\prod_{t=k_0+1}^{k_2 } y_{j_{3,t}}y_{j_{2,t}+n_2} \right)  \left(\prod_{s=1}^{ k_0} x_{j_{3,s}}x_{i_{1,s}} \right) \qquad\quad
\end{eqnarray}
where
\begin{eqnarray}
k_0& =& \min\{\min \{s\vert i_{2,s}<n+n_2-n_1 \}  ,\max \{t \vert j_{2,t}>n_1-n_2+1 \}  \}-1
\nonumber \\
&=&\min\left\{\sum_{i\in J_1} \deg_{z_{n+n_2-n_1,i}}(f),\sum_{j\in J_3} \deg_{z_{j,n_1-n_2+1}}(f)\right\}.
\end{eqnarray}
Then we have $\Upsilon(\Upsilon^{-1} (f))=f$ and $\Upsilon^{-1}(\Upsilon  (f'))=f'$ for $f\in S_2^{\ol{n_2+1,n+n_2-n_1},\ol{ n_1-n_2+1, n_1}}(n ,1,k)$ and $f' \in S^{\la 0,0 \ra}_k$.

Hence $\Upsilon$ is one-to-one, which yields that
\begin{eqnarray}
\mfk p_{M }(k)=|S^{\la 0,0 \ra}_k|&= &|S_2^{\ol{n_2+1,n+n_2-n_1},\ol{ n_1-n_2+1, n_1}}(n ,1,k)|
\nonumber \\
&=&P_2(n-n_1,n_2 ,n-n_2,n_2-n_1+1,k)
\nonumber \\
&=&P_2(n_2,n-n_1 ,n_1,n_2-n_1+1,k)
\end{eqnarray}
by (\ref{c2.30}).

Therefore, $ \mfk p_{M,V_0} (k)=\mfk p_{M }(k)=P_2(n_2,n-n_1 ,n_1,n_2-n_1+1,k) $, which is another proof of the  Proposition \ref{propf2.2}.

\begin{corollary}\label{core4.2}
If $\ell_1\leq 0$ or $\ell_2 \leq 0$ and $n_1<n_2<n$, we have
\begin{equation}  \label{e4.101}
 lc(M) =lc(\mfk p_{M })
\end{equation}
\end{corollary}

\begin{proof}

When $\ell_1 <0$ and $\ell_2\leq 0$, we define the map $\mfk t_{\la \ell_1,\ell_2 \ra}$ from $S^{\la \ell_1,\ell_2 \ra}_k$ to $S^{\la 0,0 \ra}_{k-\ell_1-\ell_2}$ by
\begin{equation}
\mfk t_{\la \ell_1,\ell_2 \ra}(f) =x_1^{-\ell_1} y_n^{-\ell_2 }f.
\end{equation}
Note that $\mfk t_{\la \ell_1,\ell_2 \ra}$ is injective, hence
\begin{equation}  \label{e4.103}
|S^{\la \ell_1,\ell_2 \ra}_k| \leq |S^{\la 0,0 \ra}_{k-\ell_1-\ell_2}|= \mfk p_{M}(k-\ell_1-\ell_2)
\end{equation}
for sufficiently large $k$.
According to Proposition \ref{prope3.1}, Corollary \ref{core4.1} and (\ref{e4.103}),
\begin{equation}
 \mfk p_{M }(k)\leq  \mfk p_{M,V_0}(k) \leq  \mfk p_{M }(k-\ell_1-\ell_2)
\end{equation}
for sufficiently large $k$.
Therefore,
\begin{equation}
lc(M)=lc(\mfk p_{M }).
\end{equation}

When $\ell_1>0$ and $\ell_2\leq 0$, for $f=gh\in S^{\la \ell_1,\ell_2 \ra}_k$ with $g= x^{\al} y^{\be} \in N(k-r) $ and $h  \in Z^{(r)}$, we set
\begin{equation}
j_0(f)=\min \{ j \in J_3 \vert  \al_j+\be_j >0\} ,
\end{equation}
Denote
\begin{equation}
Q^{j }_k
=\{f\in S^{\la \ell_1,\ell_2 \ra}_k \vert \deg_{x_{_{J_1}}}(f)=0 ,j_0(f)=j\}
\end{equation}
for $j\in J_3$, and \begin{equation}
Q^{0}_k=\bigcup_{j\in J_3} Q^{j}_k
=\{f\in S^{\la \ell_1,\ell_2 \ra}_k \vert \deg_{x_{_{J_1}}}(f)=0 , \deg_{x_{J_3}}(f)+ \deg_{y_{J_3}}(f)>0\}.
\end{equation}
For $f\in Q^{j_0}_k$, we write $f=gh$, $g=x^\al y^\be$ and $h\in Z^{(r)}$.
Then
 \begin{equation}
|\{  h\in Z^{(r)} \vert gh\in Q^{j_0}_k \text{ for some } g=x^\al y^\be  \}|
=P_1(n_1,n-n_2,n-j_0+1,r).
\end{equation}
by the definition of $P_1$(cf. (\ref{e2.31})).
According to Corollary \ref{core2.1},
 \begin{equation}
\deg_r(|\{  h\in Z^{(r)} \vert gh\in Q^{j_0}_k \text{ for some } g=x^\al y^\be  \}|)
 \leq n+2n_1-2n_2+j_0-4.
\end{equation}
Moreover,
\begin{eqnarray}
& &\deg_{k-r}(|\{ g=x^\al y^\be \vert gh\in Q^{j_0}_k \text{ for some }  h\in Z^{(r)}  \}|)
\nonumber \\
& \leq &\deg_{k-r}\left(
\sum_{k_{13}=0}^{\ell_1}\sum_{k_{23}=0}^{k-r-k_{13}}
k_{13}^{n_2-n_1-2}k_{23}^{n_2-n_1-2}(k_{13}+k_{23})\binom{n-j_0+k_{13}}{k_{13}}\binom{n-j_0+k_{23}}{k_{23}}
\right)
\nonumber \\
& = &\deg_{k-r}\left(
\sum_{k_{13}=0}^{\ell_1} \sum_{k_{23}=0}^{k-r-k_{13}}
k_{13}^{n_2-n_1-2}k_{23}^{n_2-n_1-1}  k_{13}^{n-j_0}  k_{23}^{n-j_0}
\right)\nonumber \\
& =&n+n_2-n_1-j_0-1.
\end{eqnarray}
Thus
\begin{eqnarray}
& &\deg_{k }(| Q^{j_0}_k|)
\nonumber \\
& \leq &(n+2n_1-2n_2+j_0-4)+(n+n_2-n_1-j_0-1)+1
\nonumber \\
& = &2n-4+(n_1-n_2)<2n-4,
\end{eqnarray}
yields that \begin{equation}
\deg_{k }(| Q^{0}_k|)<2n-4
\end{equation}
Note that
\begin{equation}
\deg_{k }(\mfk p_{M,V_0})=\deg_{k }(\mfk p_{M })=\deg_{k }(P_2(n_2,n-n_1 ,n_1,n_2-n_1+1,k) )=2n-4
\end{equation}
by the  Proposition \ref{propf2.2} and Corollary \ref{core2.2}.
Hence
\begin{equation}  \label{e4.115}
lc_k(\mfk p_{M,V_0})=lc_k(| S^{\la \ell_1,\ell_2 \ra}_k|)= lc_k(| S^{\la \ell_1,\ell_2 \ra}_k  \setminus Q^{0}_k |).
\end{equation}
Define the map $\mfk t_{\la \ell_1,\ell_2 \ra}$ from $S^{\la \ell_1,\ell_2 \ra}_k\setminus Q^{0}_k$ to $S^{\la 0,0 \ra}_{k+\ell_1-\ell_2}$ by
\begin{equation}
\mfk t_{\la \ell_1,\ell_2 \ra}(f) =x_1^{ \ell_1} y_n^{-\ell_2 }f.
\end{equation}
Note that $\mfk t_{\la \ell_1,\ell_2 \ra}$ is still injective and we also have
\begin{equation}
lc(M)=lc_k(|S^{\la \ell_1,\ell_2 \ra}_k\setminus Q^{0}_k|)=lc_{k}(|S^{\la 0,0 \ra}_{k }|)=lc(\mfk p_{M })
\end{equation}
by Proposition \ref{prope3.1}, Corollary \ref{core4.1} and (\ref{e4.115}).

According to (\ref{a2.27}), (\ref{e4.101}) still holds when $\ell_1\leq 0$ and $\ell_2 >0$.
\end{proof}

\subsection{Case: $1<n_1=n_2<n $, $\ell_1 \leq  0$ and $\ell_2 \geq  0$.}

We set $\ell_1=-m_1-m_2$ and $\ell_2= m_2$ with $m_1,m_2 \in \mathbb{N}$. In this case,
\begin{eqnarray}M&=&{\msr H}_{\la
-m_1-m_2,m_2\ra}\nonumber\\&=&\mbox{Span} \Big\{[\prod_{r=1}^{n_1}x_r^{l_r}]
[\prod_{1\leq p<q\leq
n_1}(x_py_q-x_qy_p)^{k_{p,q}}][\prod_{r=1}^{n_1}\prod_{s=n_1+1}^n(x_rx_s-
y_ry_s)^{l_{r,s}}]\nonumber\\& &\qquad\;\;\mid
l_r,k_{p,q},l_{r,s}\in\mbb{N};\sum_{r=1}^{n_1}l_r=m_1;\sum_{1\leq
p<q\leq n_1}k_{p,q}=m_2\Big\}\end{eqnarray}
by (6.6.53) in \cite{Xx}.
Recall that we take
\begin{eqnarray}\label{e5.11}
M_0=V_0 =  & &\mbox{Span}\Big\{  \left. \left[ \prod_{ 1\leq p<q\leq n_1} (x_py_q-x_q y_p)^{k_{p,q}}\right] \prod_{r=1}^{n_1}x_r^{\al_r} \right\vert
\nonumber \\
 & &\qquad \qquad
k_{p,q}\in \mathbb{N}, \sum_{ 1\leq p<q\leq n_1}k_{p,q} =m_2,\sum_{ 1\leq  i \leq n_1}\al_i =m_1
\Big\} . \end{eqnarray}
According to (4.163) in \cite{ZX}, we have
\begin{equation}M_k =   \mbox{Span}\{ M_0  \phi(Z^{(i)}) \mid i\leq k
\} . \end{equation}
In this case, in order to ensure $\mfk d^{n_1,n_1,n}_{\la -m_1- m_2,m_2\ra}(M_0)=0$, we have
\begin{equation}
{\mfk d}_{\la -m_1- m_2,m_2\ra} \left(x^\alpha y^\beta \right) =2 \sum_{i \in J_3} \alpha_i  +2\sum_{i \in J_1}\beta_i  .
\end{equation}
Set
\begin{equation}W'_k= \mbox{Span} \{  M_0  \phi(Z^{(k)}) \} .\end{equation}
For $f\in W_k$, we assume that $f=\sum_{i} f_i$, where $f_i$ are monomials in $\msr A$. Then we have
\begin{equation} \deg_{x_{J_1}} (f_i)+\deg_{y_{J_1}} (f_i) =k. \end{equation}
Hence
\begin{equation}\label{a5.13}
W'_k \oplus M_{k-1}=M_k,\end{equation}
 for $k\geq 1$.

Set
\begin{equation}O=\{(j,i) \in J_1\times J_1\vert j>i\}\end{equation}
and
\begin{equation}\Omega = \{\omega_{j,i} \mid (j,i)\in O \} \end{equation}
 as a set of $\frac{n_1(n_1-1)}{2}$ variables.
Set
\begin{equation} \msr E =\mbb F[ \Omega ]\end{equation}
the polynomial algebra in $\Omega$. Denote by $\msr E_m$ the subspace of its homogeneous polynomials with degree $m$ and by $\Omega^{(m)}$ the set of monomials with degree $m$.

Let
\begin{equation}\label{a5.15}
\msr F=\mbb F[z_{j,i},\Omega \mid j\in J_3',\;i\in \ol{0,n_1},\;(j,i)\neq (n+1,0)]\end{equation}
(cf.(\ref{a3.17})) be the algebra of polynomials in the variables $\{z_{j,i},x_i,y_j,\omega_{j',i'} \mid j\in J_3,\;i\in J_1,\;(j,i)\neq (n+1,0),(j',i')\in O\}$.

Define associative algebra homomorphisms $\td{\phi}: \msr F \rta \msr A$ by
\begin{equation}
\td{\phi} ( f )=\phi( f) ,\;\; \td{\phi} (\omega_{j,i})= x_j y_i-x_j y_j \qquad\for\;\;f\in \msr C ,\;i<j\in J_1. \end{equation}

Then $V_0 = \text{Span}\Big\{   \td{\phi}(g) \prod_{i=1}^{n_1}x_i^{\al_i}
\vert g\in \Omega^{(m_2)},\sum_{ 1\leq  i \leq n_1}\al_i =m_1  \}$ and
\begin{equation}W'_k= \mbox{Span} \left\{  \td{\phi}(g  h)\prod_{i=1}^{n_1}x_i^{\al_r}  \Big\vert
\sum_{ 1\leq  i \leq n_1}\al_i =m_1,\;
 g\in \Omega^{(m_2)},\; h\in   Z^{(k)}  \right\} .\end{equation}

For $g\in \msr C$, $N \in \mbb N$ and $ 1\leq i_t<j_t  \leq n_1$ for $1\leq t\leq N$, we define
\begin{equation}   \td{\msr I} \left(g \prod_{t=1}^N \omega_{i_t,j_t}
\right)
=\msr I(g) \cup
 \{ (n+n_1-j_t+2, i_t)  \vert 1\leq t\leq N\}
\end{equation}

If there exist three distinct $1\leq k_1,k_2 \leq r$ such that
$i_{k_1}<i_{k_2} $, $j_{k_1}<j_{k_2} $ and $j_{k_2}>n+1$, we say that $\{(j_1,i_1),(j_2,i_2),...,(j_r,i_r)\}$ contains a $\omega $-2-chain $\left( j_{k_1},i_{k_1}\right) \prec   \left( j_{k_2},i_{k_2}\right)  $.

If there exist three distinct $1\leq k_1,k_2,k_3 \leq r$ such that
$i_{k_3} \leq i_{k_1}<i_{k_2} $, $j_{k_1}<j_{k_2} \leq n+1< j_{k_3} < n +n_1+2-i_{k_2} $, we say that $\{(j_1,i_1),(j_2,i_2),...,(j_r,i_r)\}$ contains a $\omega$-3-chain $  \{\left( j_{k_1},i_{k_1}\right) , \left( j_{k_2},i_{k_2}\right),( j_{k_3},i_{k_3}) \}  $ and $ \left( j_{k_1},i_{k_1}\right) \prec   \left( j_{k_2},i_{k_2}\right) \prec( n+n_1+2-i_{k_3},n+n_1+2-j_{k_3})   $.

We set
\begin{eqnarray}
S'_k&=&\Big\{
  g  h \prod_{i=1}^{n_1}x_i^{\al_r}  \Big\vert
\;
 g\in \Omega^{(m_2)},\; h\in   Z^{(k)}; \text{no 3-chains in } \msr I(h\prod_{i=1}^{n_1}x_i^{\al_r}),
  \nonumber\\
& & \qquad
\text{no $\omega$-2-chains or $\omega$-3-chains in } \td{\msr I}(gh\prod_{i=1}^{n_1}x_i^{\al_r}) ,
\sum_{ i=1}^{n_1}\al_i =m_1
\Big\}.
\end{eqnarray}

\begin{lemma}\label{lem:a5.1}
For $k \geq 0$, $\text{Span}\{ \td{\phi}(S'_k)\}=W_k'$.
\end{lemma}

\begin{proof}
For $f= gh  \prod_{i=1}^{n_1}x_i^{\al_i} \in W_k'$ and $\td{\phi}(f)=\td{\phi}(gh) \prod_{i=1}^{n_1}x_i^{\al_i} \in W_k'$ with $\sum_{ 1\leq  i \leq n_1}\al_i =m_1$, $g\in \Omega^{(m_2)}$ and $h\in   Z^{(k)} $, if $\td{\msr I}(f) $ does not contain $\omega$-2-chains or $\omega$-3-chains, we have proved that $f \in W_k'$ from Lemma 3.4 (The Lemma 3.3 in \cite{ZX}). Hence we suppose that $\td{\msr I}(f) $ contains $\omega$-2-chains or $\omega$-3-chains.
We will discuss the following four cases separately and make a $\td{\phi}$-invariant transformation on $f$ in each case.

(1) Suppose the $\td{\msr I}(f)$ contains a $\omega$-2-chain $(s,t) \prec (n+n_1+2-j,i)$ with $1\leq i<j \leq n_1$, $s\in J_3'$ and $1 \leq t <i$.
Write $f= z_{s,t}\omega_{j,i} f'$, then
\begin{eqnarray}
\td{\phi}(f)&=&  \phi( z_{s,t})(x_j y_i-x_iy_j)\td{\phi}(  f')
\nonumber\\
&=&
\td{\phi}(  f')\phi( z_{s,i})(x_j y_t-x_ty_j)
-\td{\phi}(  f')\phi( z_{s,j})(x_i y_t-x_ty_i)
\nonumber\\
&=&
\td{\phi}(  f' z_{s,i}  \omega_{j,t} )
-\td{\phi}(  f'  z_{s,j}\omega_{i,t} )  .
\end{eqnarray}
Moreover, we have
\begin{eqnarray}\td{\msr I}(f )&=&
\{(s,t),(n+n_1 - j+2, i ), \td{\msr I}(f')\}
\nonumber\\&<&
\{(s,i),(n+n_1 - j+2, t), \td{\msr I}(f')\} =
\td{\msr I} (f' z_{s,i}  \omega_{j,t}).
\end{eqnarray}
and
\begin{eqnarray}\td{\msr I}(f )&=&
\{(s,t),(n+n_1 - j+2, i ), \td{\msr I}(f')\}
\nonumber\\&<&
\{(s,i),(n+n_1 - j+2, t), \td{\msr I}(f')\} =
\td{\msr I} (f' z_{s,i}  \omega_{j,t}).
\end{eqnarray}

(2) Similar to the Case (1), we suppose the $\td{\msr I}(f)$ contains a $\omega$-2-chain $(n+n_1+2-s,t) \prec (n+n_1+2-j,i)$ with $1\leq i<j \leq n_1$, $1\leq t < s \leq n_1$.
Write $f= \omega_{s,t}\omega_{j,i} f'$, then
\begin{eqnarray}
\td{\phi}(f)&=&
\td{\phi}(  f')(x_s y_i-x_s y_i)(x_j y_t-x_ty_j)
-\td{\phi}(  f')(x_s y_j-x_s y_j)(x_i y_t-x_ty_i)
\nonumber\\
&=&
\td{\phi}(  f' \omega_{s,i}\omega_{j,t} )
-\td{\phi}(  f' \omega_{s,j} \omega_{i,t} )  .
\end{eqnarray}
And we have
\begin{eqnarray}\td{\msr I}(f )&=&
\{(n+n_1+2-s ,t),(n+n_1 - j+2, i ), \td{\msr I}(f')\}
\nonumber\\&<&
\{(n+n_1+2-s,i),(n+n_1 - j+2, t), \td{\msr I}(f')\} =
\td{\msr I} (f' z_{s,i}  \omega_{j,t}).
\end{eqnarray}
and
\begin{eqnarray}\td{\msr I}(f )&=&
\{(n+n_1+2-s,t),(n+n_1 - j+2, i ), \td{\msr I}(f')\}
\nonumber\\&<&
\{(n+n_1+2-s,i),(n+n_1 - j+2, t), \td{\msr I}(f')\} =
\td{\msr I} (f' z_{s,i}  \omega_{j,t}).
\end{eqnarray}

(3) Suppose the $\td{\msr I}(f)$ contains a $\omega$-3-chain $\{ \left( s_1,t_1\right), \left( s_2,t_2\right) ,(n+n_1+2-t_3,t_1)  \}$ with $1\leq t_1<t_2 <t_3\leq n_1$, $s_1<s_2\in J_3'$.
Write $f=  z_{s_1,t_1} z_{s_2,t_2}\omega_{t_3,t_1} f'$, then
\begin{eqnarray}
\td{\phi}(f)&=&
\td{\phi}(  z_{s_1,t_1}z_{s_2,t_3}\omega_{t_2,t_1}f')
+\td{\phi}( z_{s_1,t_2}z_{s_2,t_1}\omega_{t_3,t_1}f')
-\td{\phi}( z_{s_1,t_3}z_{s_2,t_1}\omega_{t_2,t_1}f').
\end{eqnarray}
Moreover, we have
\begin{eqnarray}\td{\msr I}(f )&=&
\{(s_1 ,t_1),(s_2,t_2),(n+n_1 +2-t_3, t_1 ), \td{\msr I}(f')\}
\nonumber\\&<&
\{(s_1 ,t_1),(s_2,t_3),(n+n_1 +2-t_2, t_1 ), \td{\msr I}(f')\}
\nonumber\\&=&
\td{\msr I} (z_{s_1,t_1}z_{s_2,t_3}\omega_{t_2,t_1}f'),
\end{eqnarray}
\begin{eqnarray}\td{\msr I}(f )&=&
\{(s_1 ,t_1),(s_2,t_2),(n+n_1 +2-t_3, t_1 ), \td{\msr I}(f')\}
\nonumber\\&<&
\{(s_1 ,t_2),(s_2,t_1),(n+n_1 +2-t_3, t_1 ), \td{\msr I}(f')\}
\nonumber\\&=&
\td{\msr I} (z_{s_1,t_2}z_{s_2,t_1}\omega_{t_3,t_1}f'),
\end{eqnarray}
\begin{eqnarray}\td{\msr I}(f )&=&
\{(s_1 ,t_1),(s_2,t_2),(n+n_1 +2-t_3, t_1 ), \td{\msr I}(f')\}
\nonumber\\&<&
\{(s_1 ,t_3),(s_2,t_1),(n+n_1 +2-t_2, t_1 ), \td{\msr I}(f')\}
\nonumber\\&=&
\td{\msr I} (z_{s_1,t_3}z_{s_2,t_1}\omega_{t_2,t_1}f').
\end{eqnarray}

(4) Suppose the $\td{\msr I}(f)$ contains a $\omega$-3-chain $\{ \left( s_1,t_1\right), \left( s_2,t_2\right) ,(n+n_1+2-t_3,t_0)  \}$ with $1\leq t_0< t_1<t_2 <t_3\leq n_1$, $s_1<s_2\in J_3'$.
Write $f=  z_{s_1,t_1} z_{s_2,t_2}\omega_{t_3,t_0} f'$, then
\begin{eqnarray}
\td{\phi}(f)&=&
\td{\phi}(  z_{s_1,t_3}z_{s_2,t_2}\omega_{t_1,t_0}f')
+\td{\phi}( z_{s_1,t_1}z_{s_2,t_3}\omega_{t_2,t_0}f')
+\td{\phi}( z_{s_1,t_2}z_{s_2,t_1}\omega_{t_3,t_0}f')
\nonumber\\
& &-\td{\phi}( z_{s_1,t_2}z_{s_2,t_3}\omega_{t_1,t_0}f')
-\td{\phi}( z_{s_1,t_3}z_{s_2,t_1}\omega_{t_2,t_0}f').
\end{eqnarray}
Moreover, we also have
\begin{eqnarray}\td{\msr I}(f )&<&
\td{\msr I} ( z_{s_1,t_3}z_{s_2,t_2}\omega_{t_1,t_0}f'),\;
  \td{\msr I}(f )< \td{\msr I}(  z_{s_1,t_1}z_{s_2,t_3}\omega_{t_2,t_0}f'),
\nonumber\\
  \td{\msr I}(f )&< &\td{\msr I} (  z_{s_1,t_2}z_{s_2,t_1}\omega_{t_3,t_0}f'),
 \td{\msr I}(f )< \td{\msr I} (z_{s_1,t_2}z_{s_2,t_3}\omega_{t_1,t_0}f'),
\nonumber\\ & &\td{\msr I}(f ) < \td{\msr I} (z_{s_1,t_3}z_{s_2,t_1}\omega_{t_2,t_0}f').
\end{eqnarray}

Continue this process for $f$, until there are no $\omega$-2-chains or $\omega$-3-chains. Then there exists $g \in \text{Span}\{S_k'\}$ such that $ \td{\phi}(f)=\td{ \phi}(g)$. Therefore, $W_k' =\text{Span}\{\td{\phi}(S_k')\}$.
\end{proof}

For $0<m\in\mbb N$, we set
\begin{equation}\wht J_1^m=\{\{i_1,i_2,...,i_m\}\mid i_s\in J_1\},\quad\wht J_3^m=\{\{j_1,j_2,...,j_m\}\mid i_t\in J_1\}.\end{equation}
For \begin{equation}
f=\left(\prod_{t_1=1}^{k_1} x_{s_{1,t_1}}\right)  \left(\prod_{t_2=1}^{k_2} y_{s_{2,t_2}}\right)  \left( \prod_{t_3=1}^{k_3} z_{j_{t_3},i_{t_3}}\right)
\left( \prod_{t_4=1}^{k_4} \omega_{p_{t_4},q_{t_4}}\right) g\in
\msr E
\end{equation}
with $g \in \mbb F \left[ x_i,y_j \vert i\in J_2\cup J_3 ,j\in J_1\cup J_2 \right]$; $s_{1,t_1},i_{t_3},p_{t_4},q_{t_4} \in J_1$; $s_{2,t_2},j_{t_3} \in J_3$ and $k_1$, $k_2$, $k_3$, $k_4 \in \mbb N$, we define
\begin{eqnarray}
& &\msr I_1\left( f \right) = \{s_{1,1},...,s_{1,k_1},i_1,...,i_{k_3}
,p_1,...,p_{k_4},q_1,...,q_{k_4}\} \in \wht J_1^{k_1+k_3+2k_4},
\nonumber \\
& &
\msr I_3\left( f \right) = \{s_{2,1},...,s_{2,k_2},j_1,...,j_{k_3}\}
\in \wht J_3^{k_2+k_3}.\end{eqnarray}
Set
\begin{equation}
\wht J_1^{k_1+k_3,k_4} =\{s \in \wht J_1^{k_1+k_3,k_4}| \text{there exists $g\in \Omega^{(k_4)}$ such that } \msr I_1(g) \subset s  \} .
\end{equation}

In the following, we use the lexical order $\leq$ on $\mbb N^2$.
Fix ${\cal I}_1\in  \wht J_1^{k_1+k_3,k_4}$ and ${\cal I}_3\in \wht J_3^{k_2+k_3}$, we define
\begin{eqnarray} \label{a5.35}
& &G_{k_1,k_2,k_3,k_4}^{{\cal I}_1,{\cal I}_3}
\nonumber\\
=& &
 \Big\{ f\,=\,gh\left(\prod_{t_1=1}^{k_1} x_{s_{1,t_1}}\right)  \left(\prod_{t_2=1}^{k_2} y_{s_{2,t_2}}\right)
\Big\vert\; g \in \Omega^{(k_4)}, h\in   Z^{(k_3)},
\msr I_1(f)={\cal I}_1, \;
\nonumber\\& &
\msr I_3(f)={\cal I}_3,\;\mbox{ no 3-chains in }\;\msr I(f ),\;
\mbox{no $\omega$-2-chains or $\omega$-3-chains in }\;\msr I(f );\;
\nonumber\\ & &
1\leq s_{1,1}\leq\cdots\leq s_{1,k_1} \leq n_1,\; n_2+1\leq  s_{2,1}\leq\cdots \leq s_{1,k_2} \leq n.
\Big\}
\end{eqnarray}
Set
\begin{equation} V_{k_1,k_2,k_3}^{{\cal I}_1,{\cal I}_3}=\text{Span}\{G_{k_1,k_2,k_3}^{{\cal I}_1,{\cal I}_3}\}.\end{equation}
Note that $k_2k_4=0$ in (\ref{a5.35}), otherwise there exists $\omega_{p ,q } \mid g$ for some $(p,q) \in O$ and $\msr I(f )$ contains a $\omega$-2-chain $(s_{k_2},0) \prec (n+n_1+2+p_{k_4},q_{k_4})$.

\begin{lemma}\label{lem:a5.3}
For $k_1,k_2,k_3,k_4  \in \mathbb{N}$, ${\cal I}_1\in  \wht J_1^{k_1+k_3,k_4}$ and ${\cal I}_3\in \wht J_3^{k_2+k_3}$,
$ \td{\phi}  \left(G_{k_1,k_2,k_3,k_4 }^{{\cal I}_1,{\cal I}_3} \right)$ is a linearly independent set in $\msr A$.
\end{lemma}

\begin{proof}
According to Lemma 3.2 in \cite{ZX}, we have kown that $\phi(G_{k_1,k_2,k_3,0}^{{\cal I}_1,{\cal I}_3})$ is a linearly independent set in $\msr A$. Hence we suppose that $k_2=0$ and $k_4\geq 0$ in the following.

For convenience, we make the following convention:
\begin{equation} \omega_{s,t}=-\omega_{t,s} \qquad \for\;\;  s,t  \in J_1;t \geq s.\end{equation}
In particular, $\omega_{s,s}=0$.
For $i_1,i_2\in J_1$ with $i_1\neq i_2 $, we define an associative algebra endomorphism $\G_{i_1,i_2}$ of $\msr F$ (cf.(\ref{a5.15})) by
\begin{eqnarray}\label{a5.43}
& &\G_{i_1,i_2}(z_{j,i_1})=z_{j,i_2},\; \G_{i_1,i_2}(z_{j,i'})=z_{j,i'},\;
  \G_{i_1,i_2}(\omega_{i_1,t})= \omega_{i_2,t} ,\;\nonumber\\
& & \G_{i_1,i_2}(\omega_{s,i_1})= \omega_{s,i_2} ,\;
   \G_{i_1,i_2}(\omega_{s',t'})= \omega_{s',t'} \quad\for\;\;j\in J_3',\;
\nonumber\\
& & \; i_1\neq i'\in \ol{0,n_1},\; 1\leq t <i_1<s \leq n_1,\; i_1\not\in \{s',t'\}.
\end{eqnarray}

For any ${\cal I}\in (\wht J_1^m)$, we denote by
$\mfk c({\cal I})$ the number of distinct elements in ${\cal I}$. We want to prove that $\phi(G_{k_1,0,k_3,k_4}^{{\cal I}_1,{\cal I}_3})$ is linearly dependent by induction on $ \mfk c({\cal I}_1)+k_4$.

First $ \mfk c({\cal I}_1)+k_4 \geq 1$ by definition.
When $ \mfk c({\cal I}_1)=1$, $k_4=0$. We have kown that $\phi(G_{k_1,k_2,k_3,0}^{{\cal I}_1,{\cal I}_3})$ is a linearly independent set in $\msr A$, according to Lemma 3.2 in \cite{ZX}. Hence the lemma holds.

Next, we assume the lemma holds for $ \mfk c({\cal I}_1)+k_4<l \in\mbb N$. Consider $\mfk c({\cal I}_1)+k_4=l $ and $k_4>0$.

Suppose that $\phi(G_{k_1,0,k_3,k_4}^{{\cal I}_1,{\cal I}_3})$ is linearly dependent. Then there exists $0\neq f\in V_{k_1,0,k_3,k_4}^{{\cal I}_1,{\cal I}_3}$ such that $\phi(f)=0$.
Write
\begin{equation}
f = \sum_{i=1}^N a_i f_i \in V_{k_1,0,k_3,k_4}^{{\cal I}_1,{\cal I}_3} \cap \ker \td{\phi}
\end{equation}
where $f_i \in G_{k_1,0,k_3,k_4}^{{\cal I}_1,{\cal I}_3} $ and $0\neq a_i \in \mbb F$ for $i\in \ol{1,N}$.
For convenience, we assume that $f_1,\cdots,f_N$ are distinct.

Denote
\begin{equation} i_1 = \min \{ {\cal I}_1 \}\end{equation} and
\begin{equation}
i_2 =\min \{ i\in {\cal I}_1 \vert i>i_1 \}.
\end{equation}

According to (\ref{a5.43}), we have
\begin{equation}
\td{\phi}(\G_{i_1,i_2}(f)) =\td{\phi}(f)\vert_{x_{i_1}=x_{i_2},y_{i_1}=y_{i_2}}=0.
\end{equation}
Suppose that $\msr I\left(\G_{i_1,i_2}(f)\right)$ contains a 3-chain $(s_1,i_2) \prec (s_2,t_2) \prec (s_3,t_3)$ with $s_3\leq n+1$. Then $(s_1,i_1) \prec (s_2,t_2) \prec (s_3,t_3)$ is a 3-chain in $\msr I\left(f\right)$. Hence $\msr I\left(\G_{i_1,i_2}(f)\right)$ does not contain 3-chains.
Suppose that $\msr I\left(\G_{i_1,i_2}(f)\right)$ contains a $\omega$-2-chain $(s_1,i_2) \prec  (n+n_1+2-s_2,t_2)$ with $s_2>t_2$. Then $(s_1,i_1) \prec (n+n_1+2-s_2,t_2)$ is a $\omega$-2-chain in $\msr I\left(f\right)$. Hence $\msr I\left(\G_{i_1,i_2}(f)\right)$ does not contain $\omega$-2-chains.

Suppose that $\msr I\left(\G_{i_1,i_2}(f)\right)$ contains a $\omega$-3-chain $\{(s_1,t_1),(s_2,t_2) , (n+n_1+2-s_3,t_3)\}$ with $s_2>t_2$.
There are two cases. First, $t_1=i_2$, then $t_3=i_2$. Since $\msr I\left(f\right)$ does not contain $\omega$-2-chains, $\{ (s_1,i_1),(n+n_1+2-s_3,i_2)\} \not\subset \msr I\left(f\right)$. Then $(n+n_1+2-s_3,i_1) \in  \msr I\left(f\right)$, otherwise $\{(s_1,i_2),(s_2,t_2) , (n+n_1+2-s_3,i_2)\}$ is a $\omega$-3-chain in $\msr I\left(f\right)$. However, whether it is $\{(s_1,i_2),(s_2,t_2) , (n+n_1+2-s_3,i_1)\}$ or $\{(s_1,i_1),(s_2,t_2) , (n+n_1+2-s_3,i_1)\}$ in $\msr I\left(f\right)$, the are both $\omega$-3-chains. Hence $\msr I\left(\G_{i_1,i_2}(f)\right)$ does not contain $\omega$-3-chains.

Furthermore, we assume that ${\cal I}_1={\cal I}'_1 \cup \{i_1 \cdots,i_1\}$ with $i_1\not\in {\cal I}'_1$. Denote $ {\cal I}_1^\sharp=\msr I_1(\G_{i_1,i_2}(f_i))$, then we have $ {\cal I}_1^\sharp  ={\cal I}'_1 \cup \{ i_2,\cdots,i_2\}$ and $ \msr I_3(\G_{i_1,i_2}(f_i)) ={\cal I}_3 $ for $i\in \ol{1,N}$.

Therefore, $\G_{i_1,i_2}(f ) \in  V_{k_1,0,k_3,k_4}^{{\cal I}_1^\sharp,{\cal I}_3}$.

Note that
\begin{equation}\mfk c({\cal I}_1^\sharp)=\mfk c({\cal I}_1)-1.\end{equation}
and
\begin{equation}
\mfk c({\cal I}_1^\sharp)+ \deg_{\Omega}(\G_{i_1,i_2}(f ) ) =\mfk c({\cal I}_1)+k_4-1<l.
\end{equation}
By inductional assumption, $\phi(G_{k_1,0,k_3,k_4}^{{\cal I}_1^\sharp,{\cal I}_3})$ is linearly independent.
This yields
\begin{equation} \G_{i_1,i_2}(f )=0.\end{equation}

Write
\begin{equation} f_i=f_{i,i_1} f_{i,i_2} f_i',\end{equation}
where $f_{i,i_1}$ are monomials in $\mbb F[z_{j,i_1},\omega_{j',i_1}\vert n_2+1\leq j \leq n+1,i_1<j'\leq n_1]$, $f_{i,i_2}$ are monomials in $\mbb F[z_{j,i_2},\omega_{j',i_2}\vert n_2+1\leq j \leq n+1,i_2<j'\leq n_1]$ and $i_1,i_2 \not\in \msr I_1 (f_i')$.

Denote
\begin{equation}
I_{z,z} =\la z_{s,i_1}z_{s',i_2}-z_{s',i_1}z_{s,i_2}\vert s< s'\in J_3' \ra,
\end{equation}
\begin{equation}
I_{\omega,\omega} =\la \omega_{j,i_1}\omega_{j',i_2}-\omega_{j',i_1}\omega_{j ,i_2}\vert  i_2<j<j'\leq n_1 \ra,
\end{equation}
\begin{equation}
I_{\omega,z} =\la \omega_{j,i_1}z_{s,i_2}-\omega_{j,i_2}z_{s,i_1}\vert s \in J_3', i_2<j\leq n_1 \ra
\end{equation}
are ideals of $\msr F$.

Since there exits $c_1,c_2\in\mbb N$ such that $c_1=\deg(f_{i,i_1})$  and $c_2=\deg(f_{i,i_2})$ for $i\in \ol{1,N}$, we have
\begin{equation}
 f \in I_{z,z}+I_{\omega,\omega}+I_{\omega,z}+\la \omega_{i_2,i_1} \ra
\end{equation}
by Lemma \ref{lem:a5.2}.

Write
\begin{equation}
f= F_1+F_2+F_3+F_4,\;\;F_1\in I_{z,z},\;F_2\in I_{\omega,\omega},\;F_3 \in I_{\omega,z},F_4 \in \la \omega_{i_2,i_1} \ra .
\end{equation}
For $i=1,2,3,4$, we write
\begin{equation}
F_{i}=\sum_{j=1}^{m_i}F_{i,j}
\end{equation}
where $F_{i,j}$ are monomials in $\msr F$, and $\omega_{i_2,i_1}  \not\vert F_{i',j} $ for $i'=1,2,3$, $j \in \ol{1,m_{i'}}$.

Suppose that $F_1\neq 0$, then there exist $j  \in \ol{1,m_{i'}}$, $s_1<s_2\in J_3$ such that $ z_{s_1,i_1} z_{s_2,i_2}\vert F_{1,j}$. Note that $k_4>0$ and $\omega_{i_2,i_1}  \not| F_{1,j} $, there exists $t_0 \in \ol{i_2+1,n_1}$, such that $  \omega_{t_0,i_1} \vert F_{1,j}$.
Hence $\msr I(F_{j,1})$ contains a $\omega$-3-chain $\{ (s_1,i_1),(s_2,i_2), (n+n_1+2-t_0,i_1)\}$,  which leads to a contradiction. Therefore, $F_1 =0$.

Similarly, suppose that $F_2\neq 0$ or $F_3\neq 0$, it can be deduced that $\msr I(F_{2,j})$ or $\msr I(F_{3,j})$ has a $\omega$-2-chain for some $j$, respectively. Then $F_2=F_3=0$.

Hence $\omega_{i_2,i_1} \vert f_i$ for each $i\in \ol{1,N}$ and
\begin{equation}
\frac{f}{w_{i_2,i_1}} \in\; V_{k_1,0,k_3,k_4-1}^{{\cal I}_1\setminus \{i_1,i_2\},{\cal I}_3} \cap \ker \td{\phi}.
\end{equation}
By the inductional assumption, $\frac{f}{w_{i_2,i_1}} =0$. Then $f=0$.
Therefore, $\phi(G_{k_1,k_2,k_3,k_4}^{{\cal I}_1,{\cal I}_3})$ is linearly independent.

\end{proof}

\begin{proposition}
For $k\geq 1$, $|S_k'| = \dim (M_k/M_{k-1})$.
\end{proposition}
\begin{proof}
By Lemma \ref{lem:a5.1}, $\text{Span}\{ \td{\phi}(S'_k)\}=W_k'$ for $k \geq 0$.

We suppose that $f \in \text{Span}\{S_k'\} \cap \ker \td{\phi}$. Since
\begin{equation}  S_k'=  \bigcup_{{\cal I}_1  ,{\cal I}_3}G_{m_1,0,k,m_2}^{{\cal I}_1,{\cal I}_3},\end{equation}
we may assume that
\begin{equation}\label{a5.38}
f = \sum_{{\cal I}_1 \in \wht J_1^{k_1+k_3,k_4} ,{\cal I}_3\in \wht J_3^{k_2+k_3}} f_{{\cal I}_1  ,{\cal I}_3} \;,\;\;
\td{\phi}(f_{{\cal I}_1  ,{\cal I}_3})=\sum_j f'_{{\cal I}_1  ,{\cal I}_3,j},
\end{equation}  where $f_{{\cal I}_1  ,{\cal I}_3} \in V_{m_1,0,k,m_2}^{{\cal I}_1,{\cal I}_3}$ and $f'_{{\cal I}_1  ,{\cal I}_3,j}$ are monomials in $\msr A$.
Since $\msr I_1(f'_{{\cal I}_1  ,{\cal I}_3,j})={\cal I}_1$ and $\msr I_3(f'_{{\cal I}_1  ,{\cal I}_3,j})={\cal I}_3$ for each $j$, we have $f_{{\cal I}_1  ,{\cal I}_3}  \in \ker \td{\phi}$ for each ${\cal I}_1  ,{\cal I}_3$ in (\ref{a5.38}).

According to Lemma \ref{lem:a5.3}, $f_{{\cal I}_1  ,{\cal I}_3} =0$ for each ${\cal I}_1  ,{\cal I}_3$ in (\ref{a5.38}). Then $f=0$. That is, $\text{Span}\{S_k'\} \cap \ker \td{\phi} =\{0\}$.
Hence $\td{\phi}(S'_k)$ is a basis of $W_k'$.

We have known that $ W'_k \oplus M_{k-1}=M_k $ (cf. (\ref{a5.13})).
Then $|S_k'| = \dim (M_k/M_{k-1})$, for $k\geq 1$.
\end{proof}

For $p,q\in \mbb N$ and $p\geq q>0$, we denote
\begin{equation}
O_{p,q}=\{ (s,t) \vert  1\leq  s \leq p-t ,1\leq t\leq q \}
\end{equation}
as a ladder in $\{ (s,t) \vert  1\leq  s \leq p ,1\leq t\leq q \}$.
Denote
\begin{equation}
P_3(p,q,r)= \left\vert \left\{
\prod_{j=1}^r\omega_{s_j,t_j} \,\Big\vert \, (s_j,t_j) \in O_{p,q}\; \for\; j\in \ol{1,r}
\right\} \right\vert.
\end{equation}
Then $P_3$ is the Hilbert polynomial of the ladder-determinantal ideal, which can be calculated using  formulas in \cite{KDM,DGP,CM04}.
In particular, $P_3(p,1,r)= \binom{p+r-2}{r}$ for $p>1$.

For $f=ghf' \in S_k'$ with $g\in \Omega^{(m_2)}$, $h \in Z^{(k)}$, $f' =\prod_{i=1}^{n_1} x_i^{\al_r}$.
Since $\msr I(g)$ does not contain $\omega$-2-chains, we may assume that
\begin{equation}
g=\prod_{i=1}^{m_2} \omega_{s_i,t_i}
\end{equation}
with $1\leq s_{m_2} <\cdots< s_1 \leq n_1$ and $1\leq  t_{m_2} <\cdots< t_1 \leq n_1$.
Then the number of $g$ satisfying
\begin{equation}\label{a5.71}
s_1=\max\{ s \in J_1 \big\vert \omega_{s,t} \vert g \; \text{for some } t<s\} ,\;
t_1=\max\{ t\in J_1 \big\vert \omega_{s,t} \vert g \; \text{for some } s>t\}
\end{equation}
is
\begin{equation}
P_3(s_1,t_1,m_2-1).
\end{equation}

(1) Suppose that $\ell_1+\ell_2=0$, or equivalently, $m_1=0$.
Then $f'=1$.

For $i \in \ol{1,t_1-1}$, we have $i \not\in \msr I_1(h)$. Otherwise, there exists $j\in J_3$, such that $z_{j,i} \vert h$. Then $(j,i)\prec (n+n_1+2-s_1,t_1)$ is a $\omega$-2-chain in $\msr I(f)$.

We write
\begin{equation}\label{a5.73}
h=\prod_{(j,i)\in J_3\times J_1}
z_{j,i}^{\gamma_{j,i}}.
\end{equation}
For $h$ in (\ref{a5.73}) and $(i',i)\in O$, we denote
\begin{equation}
h|_{i,i'}
=\prod_{(j,i)\in J_3\times \ol{i,i'}} z_{j,i}^{\gamma_{j,i}}
\end{equation}

Suppose that $\msr I\left(h|_{t_1,s_1-1} \right)$ contains a 2-chain $(j_1,i_1)\prec (j_2,i_2)$.
Then $\msr I(f)$ contains a $\omega$-3-chain $\{(j_1,i_1),(j_2,i_2),(n+n_1+2-s_1,t_1) \}$ since $s_1>i_2$. Hence $\msr I\left(h|_{t_1,s_1-1} \right)$ does not contain 2-chains.

At the same time, for $g\in \Omega^{(m_2)}$ with (\ref{a5.71}) holding and no 2-chains in $\msr I(g)$, if $h\in Z^{k}$ satisfying $\ol{1,t_1-1} \cap \msr I_1(h)=\emptyset$, no 2-chains in $\msr I\left(h|_{t_1,s_1-1} \right)$ and no 3-chains in $\msr I(h)$. Thus $hg \in S_k'$.

Therefore, for $g\in \Omega^{(m_2)}$ with (\ref{a5.71}) holding, the number of $h\in Z^{k}$ is
\begin{equation}
P_1(n-n_1, n_1-t_1+1, s_1-t_1+1,k).
\end{equation}

Then
\begin{eqnarray}
|S_k|=
\sum_{(s_1,t_1)\in O}
P_1(n-n_1, n_1-t_1+1, s_1-t_1+1,k)P_3(s_1,t_1,m_2-1).
\end{eqnarray}

(2) Suppose that $\ell_1+\ell_2<0$, then $m_1>0$.
Denote
\begin{equation}\label{a5.75}
i_1= \max \{ i\in J_1  \vert \al_{i}>0  \} \in J_1.
\end{equation}
and
\begin{eqnarray}\label{a5.76}
i_2=
\left\{
    \begin{aligned}
    &  \max \{ i\in \ol{t_1,s_1-1} \vert \al_{i}>0  \} , &\text{ if } \{ i\in \ol{t_1,s_1-1} \vert \al_{i}>0  \} \neq \emptyset; \\
    &   t_1-1 , &\text{ if } \{ i\in \ol{t_1,s_1-1} \vert \al_{i}>0  \} =\emptyset.
    \end{aligned}
    \right.
\end{eqnarray}

For $i'\in \msr I_1(f')$, we have $i' \geq t_1$, otherwise $\msr I(f)$ cantains a $\omega$-2-chain $(n+1,i') \prec (n+n_1+2-s_1,t_1)$.

Set the function
\begin{eqnarray}\label{e4.184}
L(i_1,i_2,s_1,t_1)=
\begin{cases}
\binom { i_1+i_2-s_1-t_1 +m_1-1 }{m_1-2}& \text{ if }  i_1\geq s_1   \text{ and }t_1\leq i_2, \pse\\
\binom { i_1 -s_1  +m_1 -1 }{m_1-1} &\text{ if }  i_1\geq s_1   \text{ and } i_2-t_1+1=0, \pse\\
\binom { i_2 -t_1 +m_1 -1 }{m_1-1} &\text{ if }  i_1=i_2  \text{ and }t_1\leq i_2 , \pse\\
0 &\text{ else.}
\end{cases}
\end{eqnarray}

Then the number of $f'$ satisfying (\ref{a5.75}) adn (\ref{a5.76}) is
\begin{eqnarray}
L(i_1,i_2,s_1,t_1)
\end{eqnarray}

For $s_1,t_1,i_1,i_2\in \mbb N$, we denote
\begin{equation}\label{a5.81}
p_0 =\max \{ i_2,t_1\} ,
\quad
q_0=\max \{ i_1,s_1\} \geq p_0.
\end{equation}

We suppose that $  \ol{1,t_1-1} \cap \msr I_1(h) \neq \emptyset $ and $(j,i') \in J_3\times \ol{1,t_1-1} \cap \msr I(h) $. Then $\msr I(f)$ cantains a $\omega$-2-chain $ (j,i')\prec (n+n_1+2-s_1,t_1)$.
We suppose that $i_2>t_1$ $  \ol{1,i_2-1} \cap \msr I_1(h) \neq \emptyset $ and $(j,i') \in J_3\times \ol{1,i_2-1} \cap \msr I(h) $. Then $\msr I(f)$ cantains a $\omega$-3-chain $\{ (j,i'), (n+1,i_2),(n+n_1+2-s_1,t_1)\}$. Hence $  \ol{1,p_0-1} \cap \msr I_1(h) = \emptyset $.

Suppose that $\msr I\left(h|_{p_0,i_1-1 }\right)$ contains a 2-chain $(j,i) \prec (j',i')$. Then $\msr I(f)$ contains a 3-chain $(j,i)\prec (j',i')\prec (n+1, i_1)$.
Suppose that $\msr I\left(h|_{p_0,s_1-1} \right)$ contains a 2-chain $(j,i) \prec (j',i')$. Then $\msr I(f)$ contains a $\omega$-3-chain $(j,i)\prec (j',i')\prec (n+n_1+2-s_1, t_1)$.
Hence $\msr I\left(h|_{p_0,q_0-1} \right)$ does not cotain 2-chains.

On the contrary, as long as $h$ satisfies $  \ol{1,p_0-1} \cap \msr I_1(h) = \emptyset $, no 2-chains in $\msr I\left(h|_{p_0,q_0-1} \right)$ and no 3-chains in $\msr I(h)$, $ghf' \in S_k'$ holds.

Therefore, for $g\in \Omega^{(m_2)}$ with (\ref{a5.71}) holding and $f'$ satisfying (\ref{a5.75}) and (\ref{a5.76}), the number of $h\in Z^{(k)}$ is
\begin{equation}
P_1(n-n_1, n_1-p_0+1, q_0-p_0+1,k).
\end{equation}

Therefore
\begin{eqnarray}
|S_k|&=&
\sum_{(s_1,t_1)\in O}
\sum_{i_1=t_1 }^{n_1}
\sum_{i_2=t_1-1}^{ \min\{i_1,s_1-1\}}
L(i_1,i_2,s_1,t_1) P_3(s_1,t_1,m_2-1)
\nonumber\\
& &
\, P_1(n-n_1, n_1-\max \{ i_2,t_1\}+1, \max \{ i_1,s_1\}-\max \{ i_2,t_1\}+1,k).
\end{eqnarray}

In summary, we have
\begin{lemma}\label{leme4.8}
When $1<n_1=n_2<n $, $\ell_1=-m_1-m_2 $ and $\ell_2= m_2 $ with $m_1,m_2 \in \mathbb{N}$, we have
\begin{eqnarray}\label{e4.189}
\mfk p_{M,V_0} (k)&=&
\sum_{(s_1,t_1)\in O}
\sum_{i_1=t_1}^{n_1}
\sum_{i_2=t_1-1}^{ \min\{i_1,s_1-1\}}
L(i_1,i_2,s_1,t_1) P_3(s_1,t_1,m_2-1)
\nonumber\\
&  &
  P_1(n-n_1, n_1-\max \{ i_2,t_1\}+1, \max \{ i_1,s_1\}-\max \{ i_2,t_1\}+1,k).\qquad\;
\end{eqnarray}
when $m_1>0$.
\begin{eqnarray}\label{e4.190}
\mfk p_{M,V_0} (k)&=&
\sum_{(s_1,t_1)\in O}
P_1(n-n_1, n_1-t_1+1, s_1-t_1+1,k)P_3(s_1,t_1,m_2-1).\qquad
\end{eqnarray}
when $m_1=0$.
\end{lemma}

\begin{corollary}\label{core4.3}
When $1<n_1=n_2<n $, $\ell_1<0$ and $\ell_2>0$, we have
\begin{eqnarray}
lc(M)=
\begin{cases}
1-\ell_1-\ell_2& \text{ if }  n_1<n-1, \pse\\
\binom{ n+\ell_2 -3}{\ell_2 } &\text{ if } \ell_1+\ell_2=0   \text{ and } n_1=n-1, \pse\\
  \binom{n-\ell_1-2}{-\ell_1} &\text{ if }  \ell_1+\ell_2>0   \text{ and }n_1=n-1.
\end{cases}
\end{eqnarray}
\end{corollary}

\begin{proof}
(1) When $|J_3|>1$ and $m_1=0$, we have $n_1-t_1+1 >1$ for $(s_1,t_1) \in 0$, and then
\begin{eqnarray}
\deg_k(P_1(n-n_1, n_1-t_1+1, s_1-t_1+1,k))\leq 2n - s_1-t_1 -2 \leq 2n-5
\end{eqnarray}
by Corollary \ref{core2.1}.
Moreover, \begin{equation}
\deg_k(P_1(n-n_1, n_1-t_1+1, s_1-t_1+1,k))= 2n-5
\end{equation}
if and only if $s_1=2$ and $t_1=1$.
Hence
\begin{eqnarray}
lc_k(\mfk p_{M,V_0} (k))&=&P_3(2,1,m_2-1)lc_k(P_1(n-n_1, n_1 , 2,k))
\nonumber \\
&=& lc_k(h_3(n-n_1, n_1  ,k))=lc_k(\mfk p_{M } (k)).
\end{eqnarray}
according to (\ref{e4.190}).\pse

(2)  When $|J_3|>1$ and $m_1>0$, we have $n_1-\max \{ i_2,t_1\}+1>1$ since $t_1<n_1 $ and $i_2<n_1$.
 According to Corollary \ref{core2.1},
\begin{eqnarray}
& & \deg_k(  P_1(n-n_1, n_1-\max \{ i_2,t_1\}+1, \max \{ i_1,s_1\}-\max \{ i_2,t_1\}+1,k))
\nonumber \\
&\leq&
2n-\max \{ i_2,t_1\}- \max \{ i_1,s_1\}-2 \leq 2n-5
\end{eqnarray}
Moreover, \begin{equation}
 \deg_k(  P_1(n-n_1, n_1-\max \{ i_2,t_1\}+1, \max \{ i_1,s_1\}-\max \{ i_2,t_1\}+1,k))= 2n-5
\end{equation}
if and only if $s_1=2$, $t_1=1$, $i_1\in \{1,2\}$ and $i_2\in \{0,1\}$.
Note that $P_3(2,1,m_2-1)=1$, then
\begin{eqnarray}
lc_k(\mfk p_{M,V_0} (k))&=&\sum_{i_1=1}^2 \sum_{i_2=0}^1 L(i_1,i_2,2,1)
\nonumber \\
&=& 1+m_1
\end{eqnarray}
by (\ref{e4.189}).\pse

(3) When $|J_3|=1$ and $m_1=0$, we have $n_1=n-1$ and
\begin{equation}
 P_1(n-n_1, n_1-t_1+1, s_1-t_1+1,k)
=\binom{n-t_1+k-1}{k} .
\end{equation}
Hence
\begin{eqnarray}
lc_k(\mfk p_{M,V_0} (k))&=&\frac{1}{(n-2)!}\sum_{s_1=2}^{n-1} P_3(s_1,1,m_2-1)
\nonumber \\
&=& \frac{1}{(n-2)!} \binom{ n+m_2-3}{m_2 }  = \binom{ n+m_2-3}{m_2 } lc_k(\mfk p_{M } (k)).
\end{eqnarray}\pse

(4) When $|J_3|=1$ and $m_1>0$, we have $n_1=n-1$ and
\begin{eqnarray}
& &   P_1(n-n_1, n_1-\max \{ i_2,t_1\}+1, \max \{ i_1,s_1\}-\max \{ i_2,t_1\}+1,k)
\nonumber \\
&=&
\binom{n -\max \{ i_2,t_1\}+k -1}{k}.
\end{eqnarray}
Hence
\begin{eqnarray}
& &lc_k(\mfk p_{M,V_0} (k))\nonumber \\&=&\frac{1}{(n-2)!}
\sum_{s_1=2}^{n}
\sum_{i_1=1}^{n-1}
\sum_{i_2=0}^{1}
L(i_1,i_2,s_1,1) P_3(s_1,1,m_2-1)
\nonumber \\
&=&\frac{1}{(n-2)!}
\sum_{s_1=2}^{n}
\left(1+
\sum_{i_1=s_1}^{n-1}
\left( L(i_1,0,s_1,1)+   L(i_1,1,s_1,1) \right )
\right)\binom{s_1+m_2-3}{m_2-1} \nonumber \\
&=& \frac{1}{(n-2)!}
\sum_{s_1=2}^{n}
\left(1+
\sum_{i_1=s_1}^{n-1}   \binom{i_1-s_1+m_1 }{m_1-1}
\right)\binom{s_1+m_2-3}{m_2-1} \nonumber \\
&=& \frac{1}{(n-2)!}\sum_{s_1=2}^{n}\binom{s_1+m_2-3}{m_2-1}\binom{n-s_1+m_1}{m_1} \nonumber \\
&=& \frac{1}{(n-2)!} \binom{n+m_1+m_2-2}{m_1+m_2} \nonumber \\
&=&   \binom{n-\ell_1-2}{-\ell_1} lc(\mfk p_{M})
\end{eqnarray}
by (\ref{e4.184}) and Lemma \ref{lemd2.4}-(5).

\end{proof}

In summary, we have the following theorem:
\begin{theorem}
Let $\mfk g=sl(n)$ and $M={\msr H}_{\la\ell_1,\ell_2\ra} $ be an infinite-dimensional $\mfk g$-module, under the conditions stated in Lemma \ref{lemc2.1}.
Then we have
\begin{eqnarray}\label{e4.202}
\frac{lc(M)}{lc(\mfk p_M)}=
\begin{cases}
1  & \text{ if }   n_1<n_2<n \text{ or } \ell_1=\ell_2=0; \pse\\
1-\ell_1-\ell_2& \text{ if }  1<n_1=n_2<n-1; \pse\\
\binom{ n+\ell_2 -2}{ \ell_2 } &\text{ if } n_1<n_2=n ; \pse\\
\binom{ n-\ell_2 -2}{-\ell_2 } &\text{ if } 1=n_1=n_2<n,   \ell_2\leq 0\text{ and }\ell_1+\ell_2 < 0 ; \pse\\
\binom{ n-\ell_1 -2}{-\ell_1 } &\text{ if } 1<n_1=n_2=n-1, \ell_1\leq 0\text{ and }\ell_1+\ell_2 < 0; \pse\\
\binom{ n+\ell_1 -3}{\ell_1 } &\text{ if } 1=n_1=n_2<n-1  \text{ and } \ell_2=-\ell_1<0  ; \pse\\
\binom{ n+\ell_2 -3}{\ell_2 } &\text{ if } 1<n_1=n_2=n-1  \text{ and }\ell_1 =-\ell_2<0 .
\end{cases}
\end{eqnarray}
\end{theorem}
\begin{proof}
When $n_1=n_2$, similar to (\ref{a2.25}), we define an associative algebra isomorphism $\Psi': \msr A \rta  \msr A$ by
\begin{eqnarray}\label{a5.1}
& &\Psi'(x_i)=y_{n-i+1},\;\;\Psi'(y_i)=x_{n-i+1},\qquad i\in \ol{1,n}.
\end{eqnarray}
Similar to (\ref{a2.25})-(\ref{a2.27}), $\Psi': {\msr H}_{\la\ell_1,\ell_2\ra}^{n_1,n_1,n} \rta {\msr H}_{\la\ell_2,\ell_1\ra}^{n-n_1,n-n_1,n}$ is also a linear space isomorphism. And we have
\begin{equation}
\dim\left(({\msr H}_{\la\ell_1,\ell_2\ra}^{n_1,n_1,n})_k \right)=
\dim\left( ({\msr H}_{\la\ell_2,\ell_1\ra}^{n-n_1,n-n_1,n} )_k\right)
\end{equation}
for $k\geq 0$.
Then $lc\left({\msr H}_{\la\ell_1,\ell_2\ra}^{n_1,n_1,n}\right)=lc\left( {\msr H}_{\la\ell_2,\ell_1\ra}^{n-n_1,n-n_1,n} \right)$.
When $ n_1=n_2<n -1$, $\ell_1 > 0$ and $\ell_2<0$, we have
\begin{eqnarray}
lc(M)=
\begin{cases}
1-\ell_1-\ell_2& \text{ if }  n_1>1, \pse\\
\binom{ n+\ell_1 -3}{\ell_1 } &\text{ if } \ell_1+\ell_2=0   \text{ and } n_1= 1, \pse\\
  \binom{n-\ell_2-2}{-\ell_2} &\text{ if }  \ell_1+\ell_2>0   \text{ and }n_1= 1.
\end{cases}
\end{eqnarray}
according to Corollary \ref{core4.3}.

For other cases, Lemma \ref{leme4.1}, Lemma \ref{leme4.2}, Corollary \ref{core4.2} and Corollary \ref{core4.3} tell us that (\ref{e4.202}) holds.
\end{proof}

This completes the proof of Theorem 3 that we stated in the introduction.

 \end{document}